\documentclass{article}

\usepackage{amsmath} 
\usepackage{hyperref}

\usepackage{float}

\usepackage{placeins}

\ifpdf
\hypersetup{
    colorlinks=true,
    linkcolor=blue,
    filecolor=magenta,      
    urlcolor=cyan,
    pdftitle={Semi-analytic physics informed neural network for convection-dominated boundary layer problems in 2D},
    pdfauthor={Gung-Min Gie, Youngjoon Hong, Chang-Yeol Jung, and Dongseok Lee}
}
\fi

\usepackage{lipsum}
\usepackage{amsfonts}
\usepackage{graphicx}
\usepackage{epstopdf}
\usepackage{algorithmic}

\ifpdf
  \DeclareGraphicsExtensions{.eps,.pdf,.png,.jpg}
\else
  \DeclareGraphicsExtensions{.eps}
\fi

 \usepackage{pst-grad} 
 \usepackage{pst-plot} 

\newcommand{\funding}[1]{\textbf{Funding:} #1}
\newcommand{\email}[1]{\texttt{#1}}

\title{Singular Layer Physics-Informed Neural Network Method for Convection-Dominated Boundary Layer Problems in Two Dimensions\thanks{

\funding{Gie was partially supported by 
Ascending Star Fellowship, Office of EVPRI, University of Louisville; 
Simons Foundation Collaboration Grant for Mathematicians; 
Research R-II Grant, Office of EVPRI, University of Louisville; 
Brain Pool Program through the National Research Foundation of Korea (NRF) (2020H1D3A2A01110658).  
The work of Y. Hong was supported by Basic Science Research Program through the National Research Foundation of Korea (NRF) funded by the Ministry of Education (NRF-2021R1A2C1093579) and by the Korea government(MSIT) (RS-2023-00219980). 
Jung was supported by the National Research Foundation of Korea(NRF) grant
funded by the Korea government(MSIT) (No. 2023R1A2C1003120).}}}

\author{Gung-Min Gie\thanks{ Department of Mathematics, University of Louisville, Louisville, KY 
  (\email{gungmin.gie@louisville.edu}).}
  \and
Youngjoon Hong\thanks{Department of Mathematical Sciences, Korea Advanced Institute of Science and Technology (KAIST), Daejeon, Korea 
  (\email{hongyj@kaist.ac.kr}).}
  \and
  Chang-Yeol  Jung\thanks{Department of Mathematical Sciences, Ulsan National Institute of Science and Technology, 
Ulsan, Korea 
  (\email{cjung@unist.ac.kr}).} 
  \and 
Dongseok Lee\thanks{Department of Mathematical Sciences, Korea Advanced Institute of Science and Technology (KAIST), Daejeon, Korea 
  (\email{lorafa@kaist.ac.kr}).}
}

\usepackage{amsopn}

\usepackage{amsthm}

\numberwithin{equation}{section}
\numberwithin{figure}{section}

 \usepackage[left=1.05in,top=1.15in,right=1.05in]{geometry}
\usepackage{verbatim}

\usepackage{caption}
\usepackage{subcaption}
\usepackage{graphicx}

\date{\today}

\renewcommand{\epsilon}{\varepsilon}
\newcommand{\ep}{\varepsilon}

\newcommand{\blds}{\boldsymbol}

\newcommand{\be}{\begin{equation}}
\newcommand{\ee}{\end{equation}}
\newcommand{\bes}{\begin{equation*}}
\newcommand{\ees}{\end{equation*}}
\newcommand{\bse}{\begin{subequations}}
\newcommand{\ese}{\end{subequations}}

\theoremstyle{plain}
\newtheorem{theorem}{Theorem}[section]

\theoremstyle{plain}
\newtheorem{lemma}{Lemma}[section]

\theoremstyle{plain}

\begin{document}
\maketitle

\begin{abstract}

This research explores neural network-based numerical approximation of two-dimensional convection-dominated singularly perturbed problems on square, circular, and elliptic domains. Singularly perturbed boundary value problems pose significant challenges due to sharp boundary layers in their solutions. Additionally, the characteristic points of these domains give rise to degenerate boundary layer problems. The stiffness of these problems, caused by sharp singular layers, can lead to substantial computational errors if not properly addressed. Conventional neural network-based approaches often fail to capture these sharp transitions accurately, highlighting a critical flaw in machine learning methods. To address these issues, we conduct a thorough boundary layer analysis to enhance our understanding of sharp transitions within the boundary layers, guiding the application of numerical methods. Specifically, we employ physics-informed neural networks (PINNs) to better handle these boundary layer problems. However, PINNs may struggle with rapidly varying singularly perturbed solutions in small domain regions, leading to inaccurate or unstable results. To overcome this limitation, we introduce a semi-analytic method that augments PINNs with singular layers or corrector functions. Our numerical experiments demonstrate significant improvements in both accuracy and stability, showcasing the effectiveness of our proposed approach.

\end{abstract}

\tableofcontents

\section{Introduction}
The use of neural networks for approximating solutions to differential equations has gained significant attention in recent research \cite{goswami2020transfer, jin2021nsfnets, kochkov2021machine, cuomo2022scientific}. Various unsupervised neural network methods have been developed in this domain, including physics-informed neural networks (PINNs) \cite{PINN001}, the deep Ritz method (DRM) \cite{yu2018deep}, and the Galerkin-based Neural Network \cite{ainsworth2022galerkin, choi2022unsupervised, CKH2024}. These methods share the common characteristic of defining the loss function based on the residual of the differential equation being considered. 
PINNs, in particular, utilize collocation points in the space-time domain as inputs, making them well-suited for solving complex, time-dependent, multi-dimensional equations with intricate domain geometries \cite{PINN002, PINN003, PINN004, PINN005, PINN006, PINN007, PINN008}. They have become popular in scientific machine learning, facilitating the integration of physics-based and data-driven modeling within a deep learning framework. 
However, the robustness of PINNs in certain problem types remains an ongoing concern. PINNs exhibit limitations in accurately capturing complex and highly nonlinear flow patterns, such as turbulence, vortical structures, and boundary layers \cite{pinn_lim, pinn_bl01, pinn_bl02, pinn_bl03}. Addressing these challenges is of great significance in scientific machine learning research, as developing robust and reliable models is crucial for advancing the field.

The robustness of PINNs is frequently tested when it comes to approximating boundary layers. Convergence issues for PINNs can arise in the context of singularly perturbed differential equations within bounded regions, particularly when the perturbation parameter is sufficiently small. These challenges also pose hurdles for operator learning approaches like DeepONet and FNO \cite{lu2021learning, li2020fourier}. Overcoming these limitations is crucial for achieving more accurate and reliable predictions across a wide range of PINN applications. 
Various techniques have recently been proposed to address the approximation of singular problems. Fourier feature networks \cite{wang2021eigenvector}, cPINN \cite{Nam03}, XPINN \cite{xpinn}, and similar approaches have emerged to tackle the spectral bias of deep neural networks, which limits their capacity to learn high-frequency functions. However, none of these techniques have been specifically designed to handle thin-layer problems.
A more recent study conducted by the authors of \cite{blpinn} proposed a theory-guided neural network to address boundary layer problems. This approach introduces a network architecture consisting of two interconnected networks: the inner and outer networks. The inner network focuses on capturing the asymptotic expansion approximation of the solution within the boundary layer region, while the outer network handles the approximation outside of the boundary layer. Both the inner and outer networks consist of multiple parallel PINN networks, with each network representing a specific order approximation of the solution. The coupling between the inner and outer networks is enforced through the matching boundary condition loss.

In this article, we introduce a novel semi-analytic machine learning approach to address boundary layer problems. Our framework draws inspiration from rigorous singular perturbation analysis and asymptotic expansion methods used in studying singularly perturbed differential equations. For general references, see, for example, \cite{bookSP} and the numerous cited articles therein. It is worth noting that boundary layer analysis for singular perturbation problems is extensively studied and well-established in applied mathematics, with significant contributions from the fluid dynamics community. The novelty of our approach lies in developing a version of PINNs that leverages boundary layer analysis to approximate sharp transitions. Our approach can derive explicit boundary layer profile functions that are theoretically guaranteed, which are then essentially used in our new PINN scheme. Building on techniques from boundary layer analysis, our work addresses multi-dimensional problems with complex domains.

A singularly perturbed boundary value problem, such as the one presented in equations \eqref{eq:main}, is widely recognized for inducing a thin layer near the boundary, commonly referred to as a boundary layer. Within this boundary layer, the solution undergoes a rapid transition. Extensive scholarly research has been conducted on the mathematical theory of singular perturbations and boundary layers. The numerical approximation of singular perturbation problems often faces significant computational errors near the boundary due to the stiffness of the solution within the boundary layer. To achieve accurate approximations, traditional methods usually require substantial mesh refinement near the boundary. However, instead of relying on extensive mesh refinements, incorporating a global basis function known as the corrector has proven highly effective. The corrector captures the singular behavior of the solution within the boundary layers, offering a more efficient and accurate approach to solving these problems. In this paper, we deliver the challenge posed by boundary layers by performing a boundary layer analysis for each singular perturbation problem. In particular, we provide a complete singular perturbation analysis for the elliptical domain case, which constitutes a new analytical result. Through this analysis, we identify the corrector function, which accurately captures the singular behavior within the boundary layer. Subsequently, we develop a new semi-analytic approach by constructing neural networks enriched with the corrector function, embedded within the structure of a two-layer PINN with hard constraints. We refer to this approach as a \textit{singular layer PINN (SL-PINN)}. To validate the effectiveness of our proposed method, we conduct numerical simulations for each example presented. The results confirm that our SL-PINNs naturally capture the singular behavior of boundary layers and provide highly accurate approximations for the singularly perturbed boundary value problems discussed in this article. This paper primarily focuses on studying the two-dimensional convection-diffusion equations in various domains, such as a unit square, circle, and ellipse, denoted by \(\Omega\):
\begin{equation} \label{eq:main}
\begin{split}
   L_{\epsilon}u^{\ep} := - \epsilon \Delta u^{\ep} - u^{\ep}_y & = f, \quad \text{  in  } \Omega,\\
    u^{\ep}  & = 0, \quad \text{  at  }\partial \Omega.
\end{split}    
\end{equation}
The objective of this study is to develop semi-analytical neural networks that can effectively incorporate physics-based information and improve their performance by integrating corrector functions. Our approach to constructing a simple two-layer neural network shares similarities with PINNs but emphasizes the inclusion of hard constraints to enforce boundary conditions. We utilize a straightforward neural network, \(\hat{u}\), multiplied by \(g(x, y)\) to satisfy the boundary condition as follows:
\begin{equation}\label{e:PINN}
    \overline{u}({x,y} ; \, {\blds \theta}) 
        =
           g({x,y}) \, 
            \hat{u}({ {{x,y}}}; \, {\blds \theta}), 
\end{equation}
where $\hat{u}$ is defined by the two-layer NN, with $g({x,y}) = 0$ on $\partial \Omega$.
We define a two-layer neural network
\begin{equation}\label{e:NN}
    \hat{u}({{{ x,y}}}; \, {\blds \theta})
    =
        \sum_{j=1}^{n}c_{i}\sigma(w_{1 j}x+w_{2 j}y+b_{j}),
\end{equation}
where $n$ is the number of neurons. 
The network parameters are denoted by
\begin{equation}\label{e:NN2}
    \blds \theta
    =
       (w_{1 1}, ...,w_{1  n},w_{2 1},...,w_{2 n},b_{1}, ...,b_{n},c_{1},...c_{n}),
\end{equation}
and we choose the logistic sigmoid as an activation function,
$
    \sigma(z) 
        =
            1/(1 + e^{-z}).
$
The PINNs with hard constraints described in \eqref{e:NN} utilize a simplified structure, allowing for the calculation of the loss function using explicit derivatives of \(\hat{u}\). This approach is particularly advantageous when dealing with boundary layers, as it helps avoid potential computational errors arising from sharp transitions. By using direct derivatives of exponential functions, we can further reduce potential learning errors. Our approach employs a two-layer neural network, which is not only computationally efficient but also sufficient for achieving accurate numerical approximations. Additionally, by incorporating symbolic computation, our methodology can be conveniently extended to an \(M\)-layer neural network. While traditional PINNs typically leverage an \(M\)-layer architecture, our SL-PINNs employ a two-layer structure. Despite this, SL-PINNs significantly outperform conventional PINNs. Future work will include conducting error analysis for the two-layer neural network, as Barron space is specifically designed for such networks.

\section{Channel Domain}

In this section, we investigate the boundary layer behavior of the singularly perturbed convection-diffusion problem (\ref{eq:main}) in a channel domain. As \(\ep\) becomes small, the solution to (\ref{eq:main}) exhibits rapid transitions near the boundary \(y=0\). To address this, we introduce a corrector through boundary layer analysis to validate our neural network-based approach. We present a series of numerical simulations as numerical evidence.



\subsection{Boundary Layer Analysis}
We consider two-dimensional convection-diffusion equations in a channel domain as an introductory example. To emphasize the presence of the boundary layer, we employ periodic boundary conditions in the $x$ direction and zero boundary conditions in the $y$ direction:
\begin{equation} \label{e:square}
\begin{split}
   - \epsilon \Delta u^{\ep} - u^{\ep}_y & = f, \quad \text{  in  } \Omega=(0,1)\times(0,1), \\
     u^{\ep}  & = 0, \quad \text{  at  } y=0 \text{ and } y=1,\\
     u^{\ep} & \text{ is periodic in }  x.
\end{split}    
\end{equation}
When $\epsilon$ gets smaller in \eqref{e:square}, a boundary layer occurs near $y=0$. 
To address this issue, we first introduce the expansion \( {u^{\ep} \approx} u^0 + \varphi^0 \), where \( u^0 \) represents the outer expansion for the smooth part, and \(\varphi^0\) represents the inner expansion capturing the boundary layer. The outer component \( u^0 \) is obtained by examining the limit problem, setting \(\epsilon\) to zero in the governing equation:
\begin{equation} \label{e:square_limit}
\begin{split}
-u^{0}_y & = f \quad \text{  in  } \Omega,\\
u^{0} & =0 \text{ at } y=1,\\
u^{0} & \text{ is periodic in }  x.
\end{split}
\end{equation}
To determine \(\varphi^0\) for the inner expansion, we introduce the stretched variable \(\bar{y} = y / \epsilon\), treating it as an \(\mathcal{O}(1)\) quantity within the boundary layer. By formally identifying the dominant terms with respect to \(\epsilon\), we derive the corrector equation,
\begin{equation} \label{e:square_cor}
\begin{split}
\varphi^{0}_{\bar{y}\bar{y}}+\varphi^{0}_{\bar{y}} & =0 \quad \text{  in  } \Omega,\\
\varphi^{0} & =-u^{0}(x,0) \text{ at } \bar{y}=0, \\
\varphi^{0} & \text{ is periodic in }  x.
\end{split}
\end{equation}
One can easily find an explicit solution of \eqref{e:square_cor} 
\be \label{eq:sq_cor_sol}
    { {\varphi}^0}
    = -u^0(x,0) \frac{e^{-y/\ep} - e^{-1/\ep}}{1 - e^{-1/\ep}}
    = -u^0(x,0) e^{-y/\ep} + e.s.t., 
\ee
where $e.s.t.$ stands for an exponentially small term.
A convergence analysis for the boundary layer problem is essential to ensure that the proposed corrector accurately represents boundary layer behavior, thereby justifying and constructing the SL-PINN scheme. The convergence analysis for \eqref{e:square} was investigated in \cite{jung2005numerical}. To illustrate our approach to the numerical scheme, we build upon the results from \cite{jung2005numerical}, introducing the compatibility conditions and convergence theory that support our SL-PINN scheme. For more details, refer to \cite{bookSP} and \cite{jung2005numerical}. We assume the conditions:
\begin{equation}\label{compat:channel}
 \frac{\partial^j}{\partial y^j} u^0(0,1)= \frac{\partial^j}{\partial y^j} u^0(1,1) = 0, \quad \text{for } 0 \leq j \leq 1,
\end{equation}
\begin{equation}
f = f_{xx} = 0 \quad \text{at } x = 0, 1.
\end{equation}
Under these conditions, we obtain the convergence result:\\

\begin{theorem}\label{conver_thm_channel}
With the compatibility condition (\ref{compat:channel}), The following estimate holds:
\begin{equation}
|u^\varepsilon - u^0 - \varphi^0|_{L^2(\Omega)} + \sqrt{\varepsilon} |u^\varepsilon - u^0 - \varphi^0|_{H^1(\Omega)} \leq \kappa \varepsilon,
\end{equation}
where \(u^\varepsilon\), \(u^0\) is the solution of (\ref{e:square}), (\ref{e:square_limit}), and \({\varphi}^0\) is the corrector in (\ref{e:square_cor}).
\end{theorem}
Under the compatibility conditions, Theorem (\ref{conver_thm_channel}) establishes that the solution of \eqref{e:square} converges to the limit solution \eqref{e:square_limit} in the \(L^2\) norm as \(\epsilon \rightarrow 0\). Building upon these results, we ensure that the corrector function effectively represents the boundary layer profile. Therefore, this corrector function can be seamlessly integrated into our neural network scheme, specifically the SL-PINN.

\subsection{Numerical Simulations}
Based on boundary layer analysis, we modify the 2-layer PINN in \eqref{e:PINN} to incorporate the corrector profile described in \eqref{eq:sq_cor_sol}. Our new SL-PINN is proposed in the form:
\begin{equation}\label{e:app_sol_enriched_con-def}
    \widetilde{u}(x,y ; \, {\blds \theta}) 
            =x(x-1)((y-1) \hat{u}(x,y , {\blds \theta}) +\hat{u}(x,0, {\blds \theta}) { \bar {\varphi}^0}),
\end{equation}
where $\bar{\varphi}^0$ is the modified corrector function given by $\bar{\varphi}^0=e^{-\frac{y}{\epsilon}}$, and $\hat{u}$ is defined in \eqref{e:PINN}.
The residual loss function is defined by
\begin{equation}
\begin{split}
    Loss=\left( \frac{1}{N}\sum_{i=0}^{N}| L_{\epsilon}(\widetilde{u}(x_i,y_i ; \, {\blds \theta}) )-f(x_i,y_i) |^{p}\right)^{1/p} \quad  \text{ where }(x_i,y_i) \in \Omega,
\end{split}    
\end{equation}
where $p = 1, 2$.
Given that the exponentially decaying function $\bar{\varphi}^0$ satisfies the corrector equation (\ref{e:square_cor}), we can explicitly derive the residual loss:
\begin{equation}  \label{e:square_loss}
\begin{split}
-\ep \Delta \widetilde{u} - \widetilde{u}_y - f
 & = -2\ep((y - 1)\hat{u}(x,y,{\blds \theta})+\hat{u}(x,0,{\blds \theta}){\bar{\varphi}}^0)-\ep (x^2-x)\hat{u}_{xx}(x,0,{\blds \theta}) {\bar{\varphi}}^0 \\
 &-2\ep((2x-1)(y-1)\hat{u}_{x}(x,y , {\blds \theta})+(x^2-x)\hat{u}_{y}(x,y , {\blds \theta}))\\
 &-2\ep(2x-1)\hat{u}_{x}(x,0,{\blds \theta}) {\bar{\varphi}}^0\\ 
 &-\ep (x^2-x)(y-1)(\hat{u}_{xx}(x,y , {\blds \theta})+\hat{u}_{yy}(x,y , {\blds \theta}))-(x^2 - x)(\hat{u}(x,y,{\blds \theta})\\
 &+(y-1)\hat{u}_{y}(x,y , {\blds \theta}))-f.
 \end{split}
 \end{equation}
We observe that the stiff terms in the residual loss are effectively eliminated by using the corrector equation. It is noteworthy that our SL-PINN provides accurate approximations, as all terms in \eqref{e:square_loss} remain bounded as $\epsilon$ approaches $0$, irrespective of the small parameter $\epsilon$.
Therefore, our SL-PINN can be effectively applied even when $\epsilon \approx 10^{-8}$, as shown in Table \ref{tab:1}. This is a regime where other neural network-based methods exhibit significant failures. 
\begin{figure}[htbp]
     \centering
     \begin{subfigure}[b]{0.4\textwidth}
         \centering
         \includegraphics[width=\textwidth]{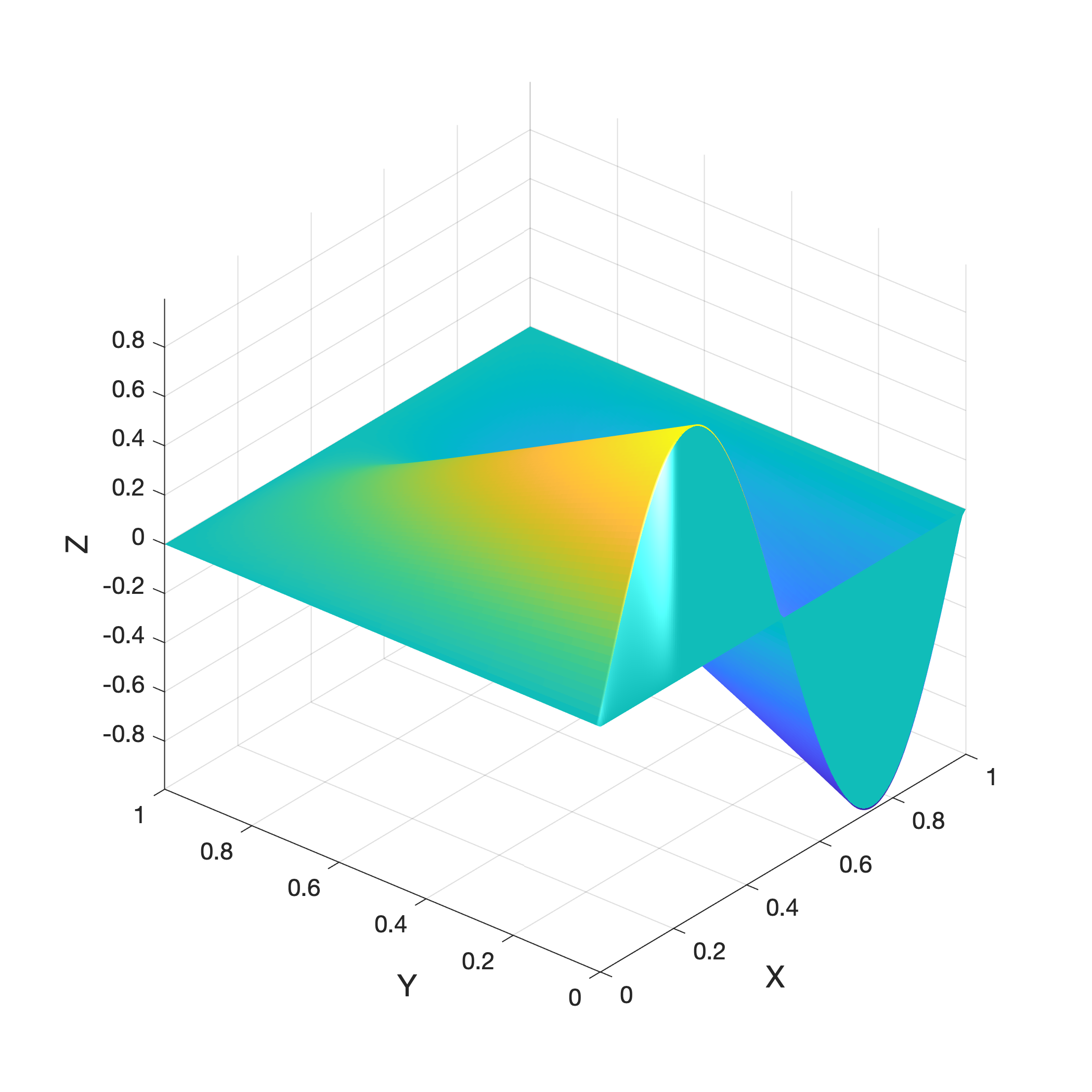}
         \caption{True solution}
     \end{subfigure}
     \begin{subfigure}[b]{0.4\textwidth}
         \centering
         \includegraphics[width=\textwidth]{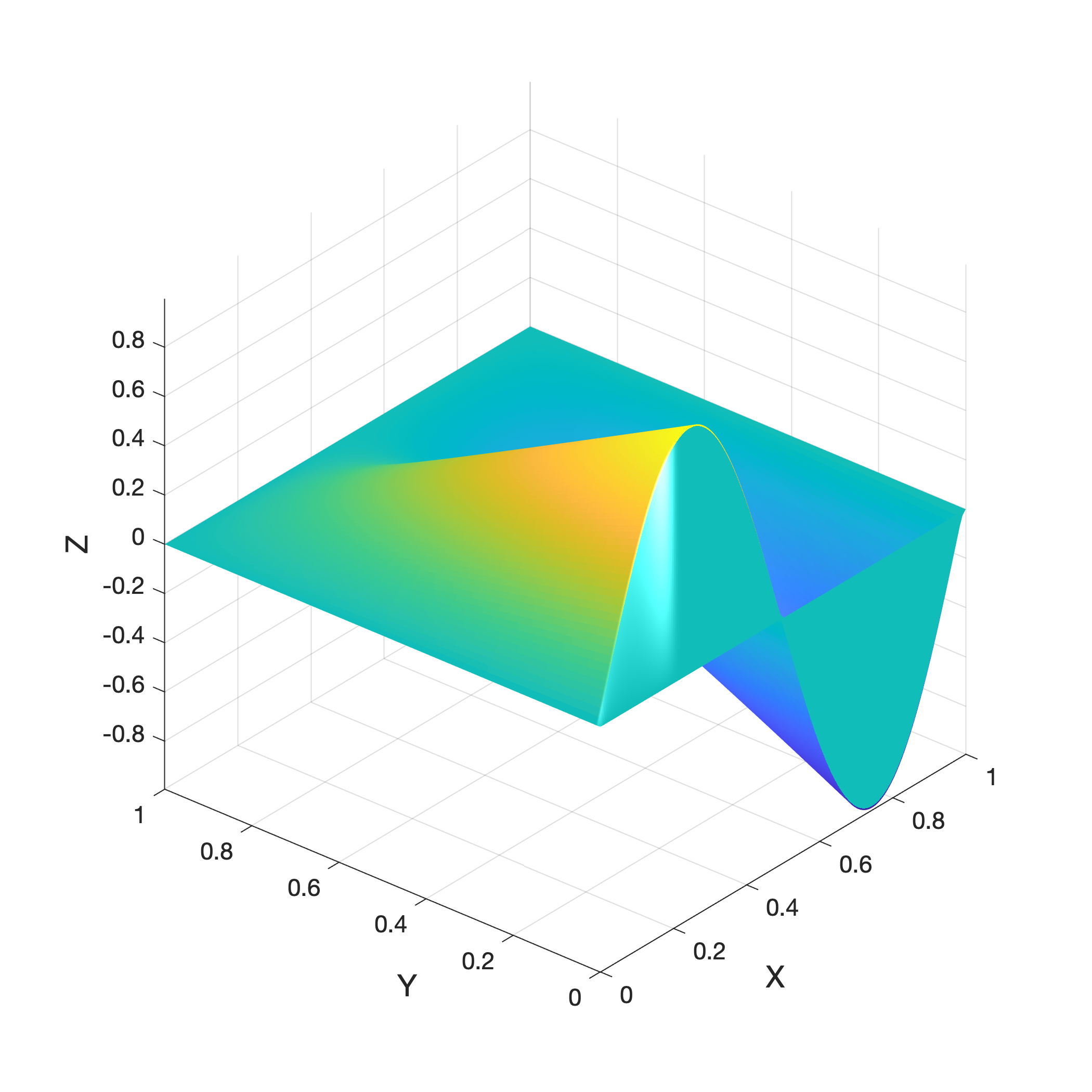}
         \caption{SL-PINN using $L^2$ training}
     \end{subfigure}
     \begin{subfigure}[b]{0.4\textwidth}
         \centering
         \includegraphics[width=\textwidth]{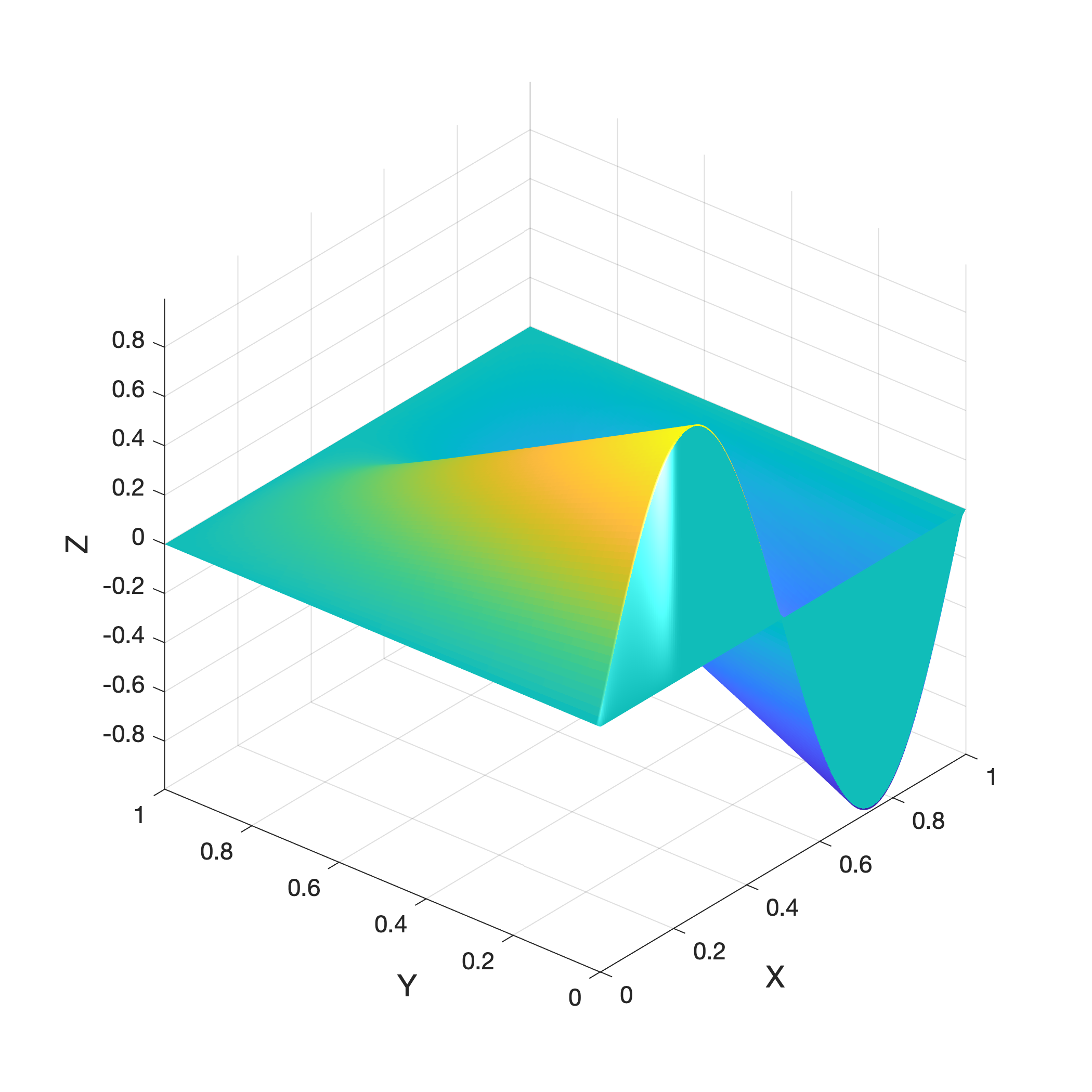}
         \caption{SL-PINN using $L^1$ training}
     \end{subfigure}
     \begin{subfigure}[b]{0.4\textwidth}
         \centering
         \includegraphics[width=\textwidth]{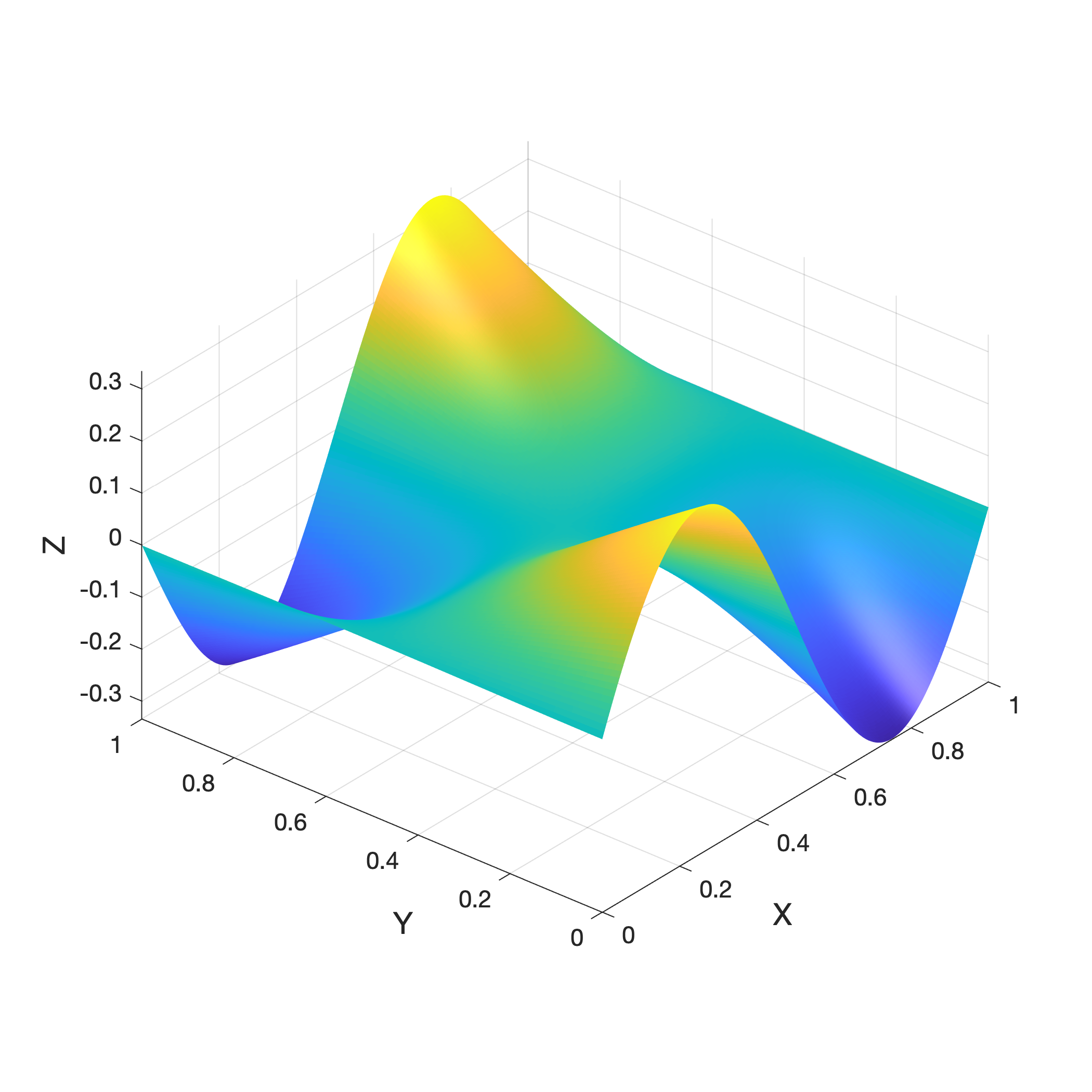}
         \caption{Conventional PINN}
     \end{subfigure}
        \caption{Numerical prediction of \eqref{e:square} with $\ep = 10^{-6}$ and $f=\sin(2 \pi x)$. For our simulations, we select a uniform grid of $50$ discretized points in each of the $x$ and $y$ directions.}\label{fig1}
\end{figure}

\begin{figure}[htbp]
\centering
\includegraphics[scale=0.4]{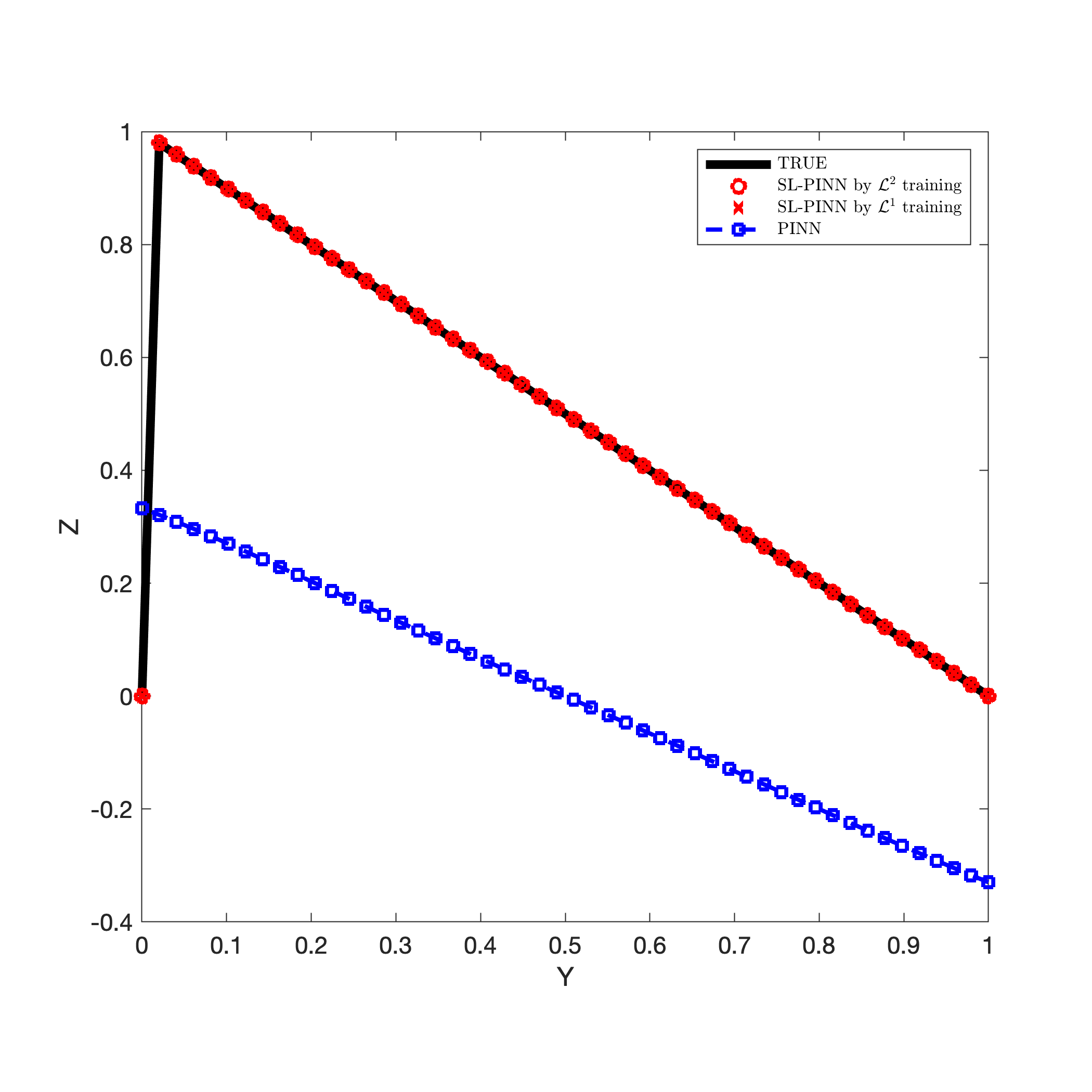}
\caption{The one-dimensional profile of predicted solutions along the line $x = 0.25$.} \label{fig1-1}
\end{figure}

We compare the performance of the standard five-layer PINN approximation with our new SL-PINN approximation, which utilizes only two layers. Figure \ref{fig1} presents the numerical solutions of \eqref{e:square} with $\epsilon = 10^{-6}$ and $f = \sin 2 \pi x$. The results in Figure \ref{fig1} demonstrate that the conventional PINN method fails to approximate the solution of the singular perturbation problem accurately. 
In contrast, our new scheme, employing a two-layer neural network with a small number of neurons, outperforms the conventional PINN. The numerical results shown in Figure \ref{fig1} and Table \ref{tab:1} provide strong evidence that the SL-PINN significantly outperforms the conventional PINN method, owing to the corrector function embedded in the scheme. Notably, our SL-PINN, enhanced with the corrector, produces stable and accurate approximate solutions regardless of the small parameter $\epsilon$.
A more detailed examination can be seen in Figure \ref{fig1-1}, which shows the one-dimensional profile of predicted solutions at $x=0.25$. This comparison clearly illustrates that the $L^1$ and $L^2$ training approaches provide accurate approximations. However, the conventional PINN method falls short in capturing the sharp transition near the boundary layer.

\section{Circular Domain}
In this section, we investigate the boundary layer behavior of the singularly perturbed convection-diffusion problem (\ref{eq:main}) in a circular domain, as well as in time-dependent and nonlinear cases. As \(\epsilon\) becomes small, the solution to these equations exhibits rapid transitions near the boundary \(\Gamma_{-} := \{(x, y) \mid x^2 + y^2 = 1, y \leq 0\}\). Additionally, unlike the channel domain, the circular domain introduces singularities at \((\pm 1, 0)\) where the upper semi-circle and lower semi-circle meet. To address these issues, we introduce a compatibility condition and employ a corrector through boundary layer analysis to validate our neural network-based approach. We present a series of numerical simulations as numerical evidence, demonstrating the robustness of our method even in non-compatible cases.

\subsection{Boundary Layer Analysis}

We shift our focus to the convection-diffusion equation within a circular domain. We start by examining the time-independent problem and then extend our study to the time-dependent equations.
We consider the convection-diffusion equation in a circular domain such that
\begin{equation} \label{e:circle_eq}
\begin{split}
   L_{\epsilon}u^{\ep} {:=}- \epsilon \Delta u^{\ep} - u^{\ep}_y & =f, \quad \text{  in  } \Omega=\{(x,y) | x^{2}+y^{2}< 1\}\\
    u^{\ep}  & = 0, \quad \text{  at  } \partial \Omega.
\end{split}    
\end{equation}
We begin by introducing the expansion \( {u^{\ep} \approx} u^0 + \varphi^0 \), where \( u^0 \) represents the outer expansion for the smooth part, and \(\varphi^0\) represents the inner expansion that captures the boundary layer.
To derive the corrector function from (\ref{e:circle_eq}), we begin by examining the limit problem, setting $\epsilon$ to $0$ in the governing equation:
\begin{equation} \label{e:circle_limit}
\begin{split}
L_{0}u^{0} :=-u^{0}_y & = f \quad \text{  in  } \Omega,\\
u^{0} & =0 \text{ at } \Gamma_{+},
\end{split}
\end{equation}
where $\Gamma_{+}:= \{ (x, y)|x^2 + y^2 = 1, y > 0\}$ is the upper semi-circle.
Note that the boundary condition in \eqref{e:circle_limit} is imposed only on the upper semi-circle. Consequently, we expect the boundary layer to form near the lower semi-circle {$\Gamma_{-}:= \{ (x, y)|x^2 + y^2 = 1, y \leq 0\}$ }.
The singularly perturbed problem described by equation \eqref{e:circle_eq} necessitates careful treatment in boundary layer analysis due to the presence of a degenerate boundary layer near the characteristic points at $(\pm 1,0)$ \cite{hong2014numerical}.
To develop the boundary layer analysis, we use boundary-fitted coordinates defined by \( x = (1 - \eta) \cos \tau \) and \( y = (1 - \eta) \sin \tau \), where \( 0 \leq \eta < 1 \) and \( \tau \in [0, 2\pi] \). In this context, \(\eta\) represents the normal component and \(\tau\) represents the tangential component.
By setting $u^{\ep}(x,y)=v^{\ep}(\eta,\tau)$, we derive that
\be
\begin{split}
P_{\epsilon} v^{\epsilon} 
    & := \frac{1}{(1-\eta)^2}
    \left(
    -\epsilon \left( v^{\epsilon}_{\tau \tau} - (1-\eta) v^{\epsilon}_{\eta} + (1-\eta)^{2} v^{\epsilon}_{\eta \eta} \right) 
    + (1-\eta)^{2} \sin \tau v^{\epsilon}_{\eta} - (1-\eta) \cos \tau v^{\epsilon}_{\tau}
    \right)\\
    & = f, \quad  \text{in } D, \\
v^{\epsilon}(0, \tau) & = 0, \quad \text{at } 0 \leq \tau \leq 2\pi,
\end{split}
\ee
where $D:=\{(\eta,\tau)| 0< \eta <1 , 0\leq \tau \leq 2\pi\}$.
Performing the boundary layer analysis as in \cite{hong2014numerical}, we are led to the following equation for the corrector equation
\be \label{e:cor:circle}
\begin{split}
    - \ep \varphi^0_{\eta \eta } + (\sin \tau) \varphi^0_{\eta}  & = 0, \quad \text{  for  } 0 < \eta < 1, \quad \pi < \tau < 2 \pi, \\
     \varphi^0  & = - u^0(\cos \tau, \sin \tau), \quad {\text{at } \eta=0}, \\
     \varphi^0  & \rightarrow 0 \quad \text{ as } \eta \rightarrow 1.
\end{split}
\ee
Hence, we can find an explicit solution of the correction equation \eqref{e:cor:circle}:
\be
\varphi^0 = -u^0(\cos\tau, \sin\tau)\exp\left(\frac{\sin \tau}{\ep} \eta \right) \chi_{[\pi, 2\pi]}(\tau),
\ee
where $\chi$ stands for the characteristic function.
To match the boundary condition in our numerical scheme, we introduce a cut-off function to derive an approximate form 
\be
\bar \varphi^0 = -u^0(\cos\tau, \sin\tau)\exp\left(\frac{\sin \tau}{\ep} \eta \right) \chi_{[\pi, 2\pi]}(\tau) \delta(\eta),
\ee
where $\delta(\eta)$ is a smooth cut-off function such that $\delta(\eta) = 1$ for $\eta \in [0, 1/2]$ and $=0$ for $eta \in [3/4, 1]$.
A convergence analysis for the boundary layer problem is essential to ensure that the proposed corrector accurately represents boundary layer behavior, thereby justifying the construction of the SL-PINN scheme. We briefly introduce the compatibility conditions and convergence theory supporting our SL-PINN scheme. For simplicity, we consider the first compatibility condition: 
\begin{equation}\label{compat:circle}
\frac{\partial^{p_{1}+p_{2}} f}{\partial x^{p_{1}} \partial y^{p_{2}}} = 0 \quad \text{at} \quad (\pm 1, 0) \quad \text{for} \quad 0 \leq 2p_{1} + p_{2} \leq 2 , \quad p_{1}, p_{2} \geq 0.
\end{equation}
For more details on compatibility conditions and a complete convergence analysis, see e.g., \cite{hong2014numerical}.


\begin{theorem}\label{conver_thm_circle}
    With the compatibility condition (\ref{compat:circle}), The following estimate holds:
\begin{equation}
|u^\varepsilon - u^0 - \bar{\varphi}^0|_{L^2(\Omega)} + \sqrt{\varepsilon} |u^\varepsilon - u^0 - \bar{\varphi}^0|_{H^1(\Omega)} \leq \kappa \sqrt{\varepsilon},
\end{equation}
where \(u^\varepsilon\), \(u^0\) is the solution of (\ref{e:circle_eq}), (\ref{e:circle_limit}), and \(\bar{\varphi}^0\) is the corrector in (\ref{e:cor:circle}).
\end{theorem}
Under the compatibility conditions, theorem (\ref{conver_thm_circle}) demonstrates that the solution of \eqref{e:circle_eq} converges to the limit solution \eqref{e:circle_limit} in the \(L^2\) norm as \(\epsilon \rightarrow 0\). Building on these results, we ensure that the corrector function effectively represents the boundary layer profile. Consequently, this corrector function can be seamlessly integrated into our neural network scheme, specifically the SL-PINN.

\subsection{Numerical Simulations for Time-indepedent problem} \label{subsec:plain}

We now establish the semi-analytic SL-PINN method as
\begin{equation}\label{e:EPINN_scheme_circle01}
    \widetilde{v}(\eta,\tau ; \, {\blds \theta}) 
            =(\hat{v}(\eta,\tau, {\blds \theta})-\hat{v}(0,\tau, {\blds \theta}) {\hat{ \varphi}^0})C(\eta,\tau),
\end{equation}
where $\hat{ \varphi}^0$ is given by 
\begin{equation}
     \hat{ \varphi}^0=\exp\left(\frac{\sin \tau}{\ep} \eta \right) \chi_{[\pi, 2\pi]}(\tau) \delta(\eta),
\end{equation}
and 
\begin{equation} \label{e:circle_C_term}
\begin{split}
    C(\eta,\tau)=\begin{cases}
    1-(1-\eta)^3 , \text{ if } 0 < \tau < \pi,\\
    1-(1-\eta)^3-((1-\eta)\sin\tau)^3, \text{ if } \pi \leq \tau \leq 2\pi,
    \end{cases}
\end{split}
\end{equation}
and $\hat{v}(\eta,\tau , {\blds \theta}) = \hat{u}(x,y , {\blds \theta})$.
Then, the residual loss is defined by
\begin{equation} \label{e:circle_loss_first}
\begin{split}
    Loss= \left( \frac{1}{N} \sum_{i=0}^{N} \left| P_{\epsilon}\widetilde{v}(\eta_{i}, \tau_{i}; \boldsymbol{\theta}) - f \right|^{p} \right)^{1/p} \quad \text{for} \ (\eta_{i}, \tau_{i}) \in D,
\end{split}    
\end{equation}
where $p=1,2$.
Since the boundary layer behavior occurs near the lower semi-circle, we divide the residual loss \eqref{e:circle_loss_first} into two parts: $0 < \tau < \pi$ and $\pi \leq \tau \leq 2 \pi$.
For the case of $0 < \tau < \pi$, the calculation of the residual loss is straightforward:

\begin{equation} \label{e:DD1c}
\begin{split}
    & P_{\epsilon}\widetilde{v}((\eta,\tau, {\blds \theta})) - f \\
    & =
    -\ep \left(\frac{1}{(1-\eta)^2} - (1-\eta)\right) \hat{v}_{\tau \tau}(\eta, \tau, {\blds \theta}) 
    + \left((1-\eta)^2 - \frac{1}{(1-\eta)}\right) \cos(\tau) \hat{v}_{\tau}(\eta, \tau, {\blds \theta}) \\
    & \qquad -\ep (1 - (1-\eta)^3) \hat{v}_{\eta \eta}(\eta, \tau, {\blds \theta}) \\
    & \qquad - \left(6\ep + \ep (1-\eta)^2 - \frac{\ep}{(1-\eta)} - \sin(\tau)(1 - (1-\eta)^3)\right) \hat{v}_{\eta}(\eta, \tau, {\blds \theta}) \\
    & \qquad + \left(3\ep + \frac{6\ep}{(1-\eta)} + 3 (1-\eta)^2 \sin(\tau)\right) \hat{v}(\eta, \tau, {\blds \theta}) - f,
\end{split}
\end{equation}
where $0 \leq \eta < 1 $ and $0 < \tau < \pi.$
However, for \(\pi \leq \tau \leq 2 \pi\), the introduction of the boundary layer element in \eqref{e:EPINN_scheme_circle01}, along with the inclusion of the regularizing term in \eqref{e:circle_C_term}, complicates the residual loss as follows:
\begin{equation}
\begin{split}
    & P_{\epsilon}\widetilde{v}(\eta,\tau, {\blds \theta}) - f \\
    & = -\ep \left( \frac{1}{(1-\eta)^2} - (1-\eta) - (1-\eta) \sin^3(\tau) \right) \hat{v}_{\tau \tau}(\eta, \tau, {\blds \theta}) \\
    & \quad + \left( 6 \ep (1-\eta) \sin^2(\tau) \cos(\tau) + \left( (1-\eta)^2 - \frac{1}{(1-\eta)} + (1-\eta)^2 \sin^3(\tau) \right) \cos(\tau) \right) \hat{v}_{\tau}(\eta, \tau, {\blds \theta}) \\
    & \quad -\ep \left( 1 - (1-\eta)^3 - (1-\eta)^3 \sin^3(\tau) \right) \hat{v}_{\eta \eta}(\eta, \tau, {\blds \theta}) \\
    & \quad - \left( 6\ep + \ep (1-\eta)^2 - \frac{\ep}{(1-\eta)} + 7\ep (1-\eta)^2 \sin^3(\tau) + (1-\eta)^3 \sin^3(\tau) - (1 - (1-\eta)^3) \sin(\tau) \right) \hat{v}_{\eta}(\eta, \tau, {\blds \theta}) \\
    & \quad + \left( 3\ep + \frac{6\ep}{(1-\eta)} + 6 \ep (1-\eta) \sin(\tau) + 6 (1-\eta)^2 \sin(\tau) \right) \hat{v}(\eta, \tau, {\blds \theta}) \\
    & \quad - \hat{v}(0, \tau, {\blds \theta}) \Big[ \ep \left( C(\eta, \tau) \delta_{\eta \eta} - \left( (1-\eta) C(\eta, \tau) - 2 C_{\eta}(\eta, \tau) \right) \delta_{\eta} \right. 
     \quad \left.  -\left( (1-\eta) C_{\eta}(\eta, \tau) - C_{\eta \eta}(\eta, \tau) \right) \delta \right) \\
    & \quad - C(\eta, \tau) \sin(\tau) \delta_{\eta} + C_{\eta}(\eta, \tau) \sin(\tau) \delta \Big] \exp\left( \frac{\sin \tau}{\ep} \eta \right) \\
    & \quad + \frac{\delta}{(1-\eta)^2} \Big[ (\ep \hat{v}_{\tau \tau}(0, \tau, {\blds \theta}) + 1+\eta) \cos(\tau) \hat{v}_{\tau}(0, \tau, {\blds \theta}) - \eta \sin(\tau) \hat{v}(0, \tau, {\blds \theta})) C(\eta, \tau) \\
    & \quad + \left( 2\ep \hat{v}_{\tau}(0, \tau, {\blds \theta}) + (1+\eta)) \cos(\tau) \hat{v}(0, \tau, {\blds \theta}) \right) C_{\tau}(\eta, \tau) + \ep \hat{v}(0, \tau, {\blds \theta}) C_{\tau \tau}(\eta, \tau) \Big] \exp\left( \frac{\sin \tau}{\ep} \eta \right) \\
    & \quad + C(\eta, \tau) \hat{v}(0, \tau, {\blds \theta}) \frac{\delta(\eta)}{(1-\eta)^2} \left( \frac{\eta \cos^2(\tau)}{\ep} \right) \exp\left( \frac{\sin \tau}{\ep} \eta \right) - f,
\end{split} \label{eq:circle:complex}
\end{equation}
where $0\leq \eta < 1$ and $\pi \leq \tau \leq 2\pi$.

To make the computation in \eqref{eq:circle:complex} feasible, we extract the largest order term in $\ep$, which includes $O(\ep^{-1})$ such that
\begin{equation}
    \psi(\eta,\tau,{\blds \theta}) := (1-(1-\eta)^3-(((1-\eta)\sin\tau)^3)\hat{v}(0,\tau,{\blds \theta})\frac{\delta(\eta)}{(1-\eta)^2} \left(\frac{\eta}{\ep}\cos^2\tau \right) \exp\left(\frac{\sin \tau}{\ep} \eta \right){\chi_{[\pi, 2\pi]}(\tau)}.
\end{equation}
When $p=1$, the loss in \eqref{e:circle_loss_first} becomes, by the triangular inequality, 
\begin{equation}\label{e:bound00}
 Loss  \leq \frac{1}{N}\sum_{i=0}^{N}|\psi (\eta_{i},\tau_{i},{\blds \theta})|
 +  \frac{1}{N}\sum_{i=0}^{N} |P_{\epsilon}\widetilde{v}((\eta_{i},\tau_{i},{\blds \theta}))-\psi (\eta_{i},\tau_{i},{\blds \theta})-f|.
\end{equation}
The rightmost term does not include a large order term such as \(\ep^{\alpha}\) (where \(\alpha < 0\)) since \(\psi\) contains large terms like \(\ep^{-1}\). Hence, this part can be computed using conventional \(L^1\) or \(L^2\) loss. 
{Assume that \(\boldsymbol{\theta}\) is fixed and that we select a sufficiently large number of sampling points \(N\),
\begin{equation}\label{e:normal_bound01}
 \frac{1}{N}\sum_{i=0}^{N} \left|P_{\epsilon}\widetilde{v}(\eta_{i}, \tau_{i}, \boldsymbol{\theta}) - \psi(\eta_{i}, \tau_{i}, \boldsymbol{\theta}) - f\right| \approx \frac{1}{\pi}\int_{0}^{2\pi} \int_{0}^{1} \left|P_{\epsilon}\widetilde{v}(\eta, \tau, \boldsymbol{\theta}) - \psi(\eta, \tau, \boldsymbol{\theta}) - f\right| \, d\eta \, d\tau < \infty.
\end{equation}}
Given that \(\psi\) includes \(\mathcal{O}(\epsilon^{-1})\), we handle \(\psi\) using the \(L^1\) loss with the bound:

\[ \label{known_upper_bound}
    \left | \frac{1}{\epsilon} \exp\left( \frac{\sin \tau}{\epsilon} \eta \right){\chi_{[\pi, 2\pi]}(\tau)}\right| _{L^{1}(D)} < C,
\]
where \(C\) is a generic constant. Assume we choose a sufficiently large number of sampling points \(N\) such that
\begin{equation}\label{e:bound01}
\begin{split}
     & \frac{1}{N}\sum_{i=0}^{N}|\psi(\eta_{i},\tau_{i},{\blds \theta})| \approx \frac{1}{\pi}\int^{2\pi}_{\pi} \int^{1}_{0} |\psi(\eta,\tau,{\blds \theta})| d\eta d\tau < C.
\end{split}
\end{equation}
Hence, we obtain that
\begin{equation}\label{e:bound03}
\begin{split}
     \frac{1}{\pi}\int^{2\pi}_{\pi} \int^{1}_{0} | \psi(\eta,\tau,{\blds \theta})| d\eta d\tau
     & < C + C\ep e^{-\frac{1}{2\ep}}.
\end{split}
\end{equation}
Therefore, the loss \eqref{e:circle_loss_first} becomes nearly constant due to \eqref{e:bound03} and \eqref{e:normal_bound01} when \(\epsilon\) is sufficiently small. This simplification makes our computations feasible.

\begin{figure}[H]
     \centering
     \begin{subfigure}[b]{0.4\textwidth}
         \centering
         \includegraphics[width=\textwidth]{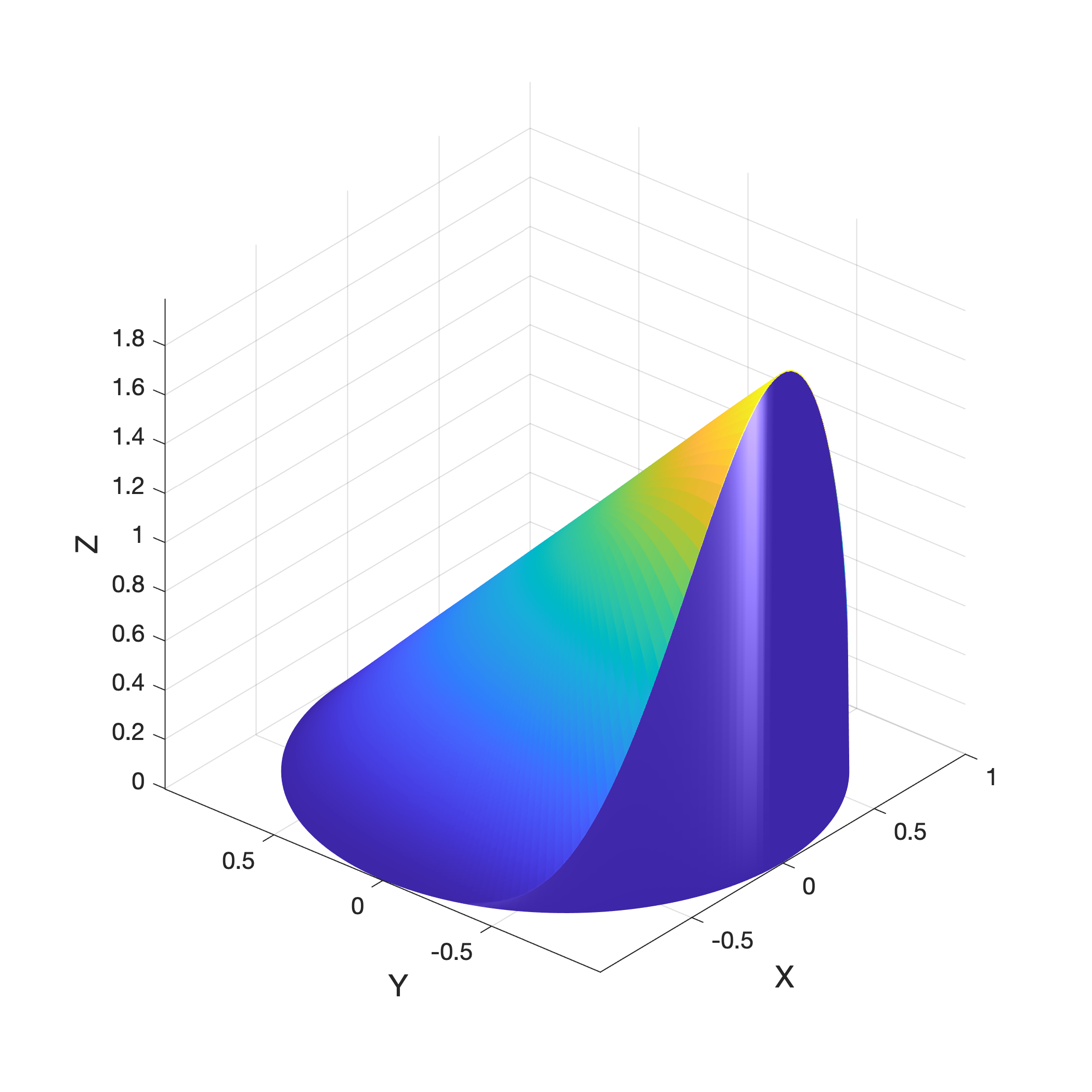}
         \caption{True solution}
     \end{subfigure}
     \begin{subfigure}[b]{0.4\textwidth}
         \centering
         \includegraphics[width=\textwidth]{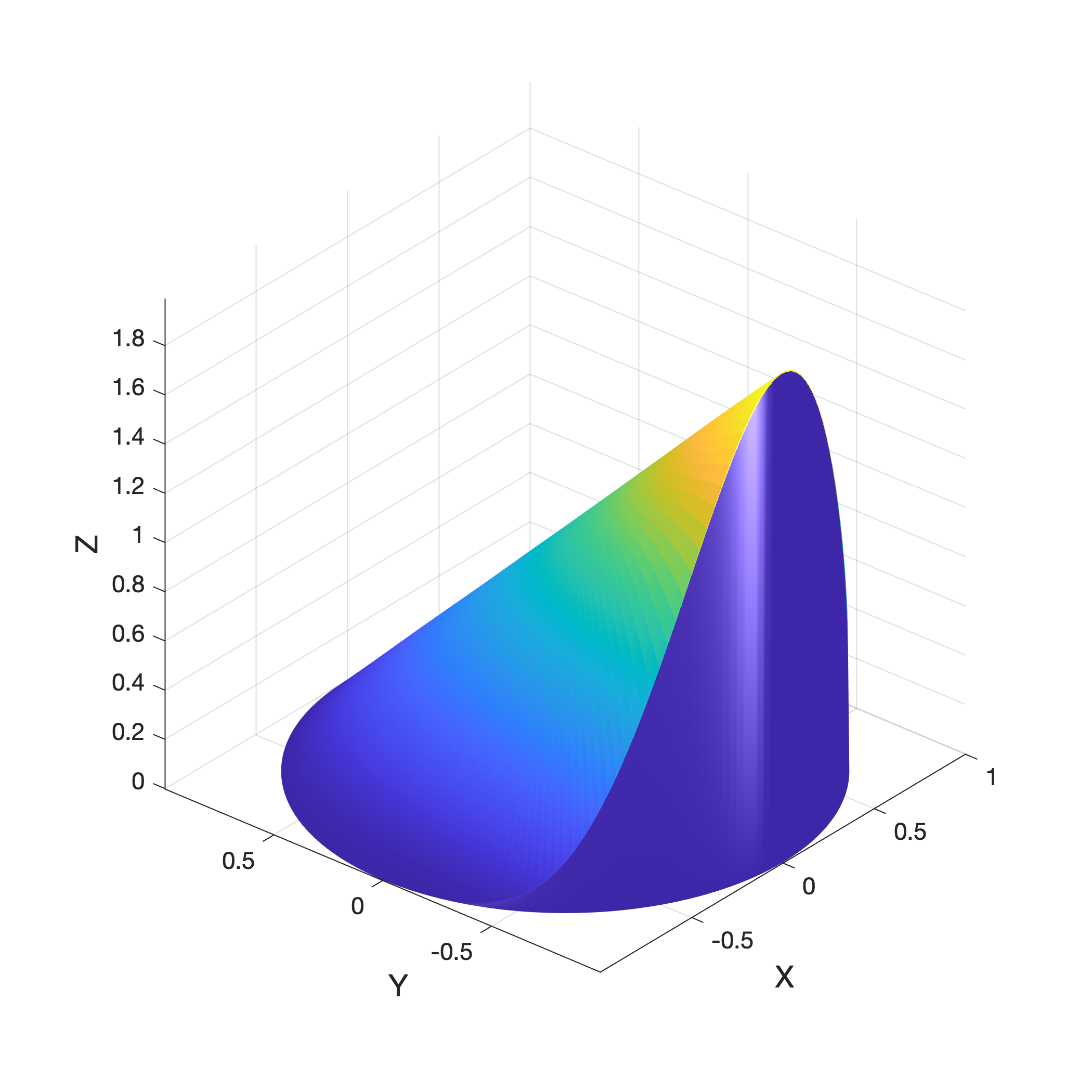}
         \caption{SL-PINN using $L^2$ training}
     \end{subfigure}
     \begin{subfigure}[b]{0.4\textwidth}
         \centering
         \includegraphics[width=\textwidth]{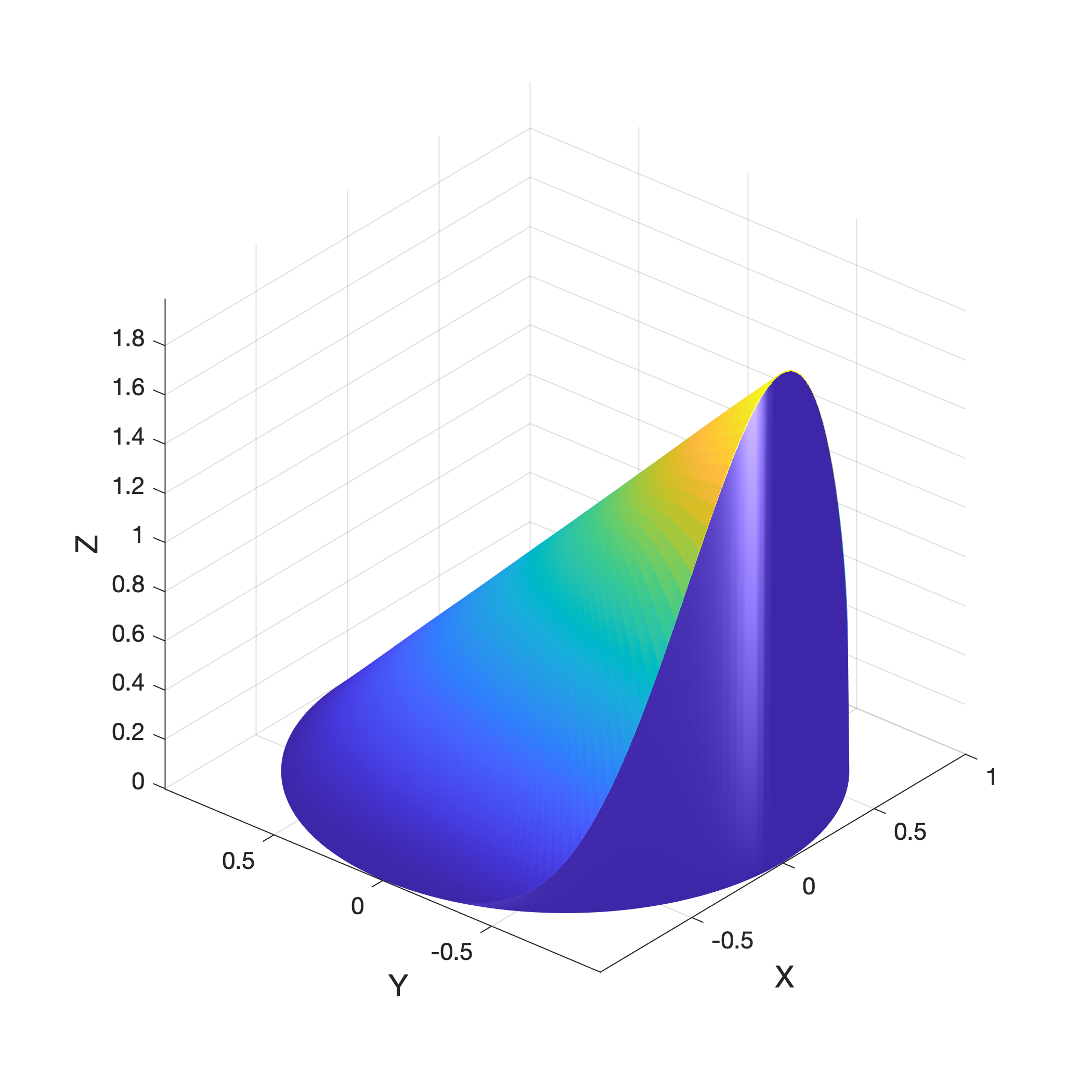}
         \caption{SL-PINN using $L^1$ training}
     \end{subfigure}
     \begin{subfigure}[b]{0.4\textwidth}
         \centering
         \includegraphics[width=\textwidth]{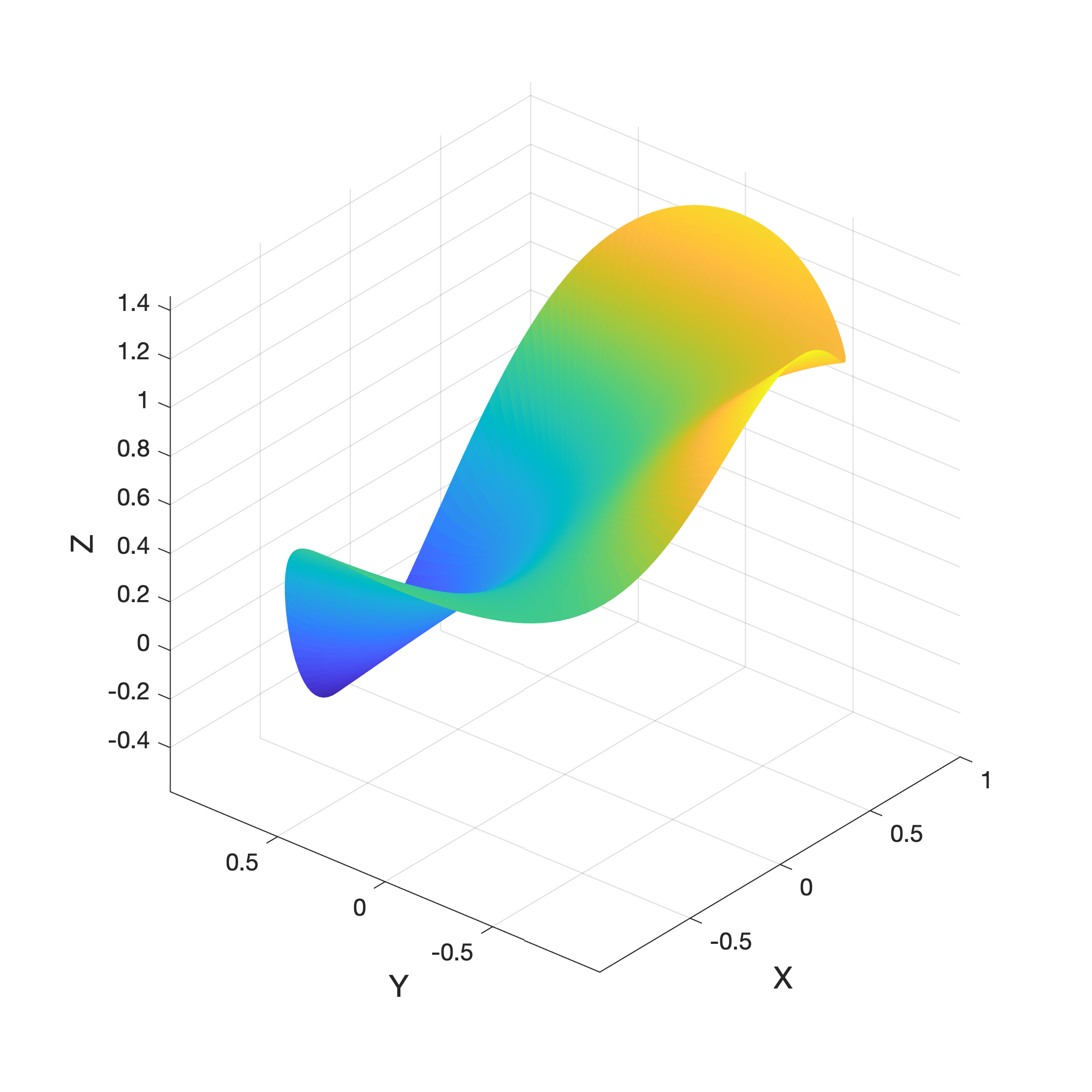}
         \caption{Conventional PINN}
     \end{subfigure}
        \caption{Numerical prediction of \eqref{e:circle_eq} with $\ep = 10^{-6}$ and $f=(1-x^2)^2$. For our simulations, we select a uniform grid of $50$ discretized points in each of the $\eta$ and $\tau$ directions.}\label{fig2}
\end{figure}

The performance of our new approach is demonstrated through a series of numerical experiments, as shown in Figures (\ref{fig2})-(\ref{fig4}). For these experiments, we use 50 uniformly discretized grid points in both the \(\tau\) and \(\eta\) directions.
We compare the performance of the standard five-layer PINN approximation with our SL-PINN approximation. Figure \ref{fig2} displays the numerical solutions of \eqref{e:circle_eq} with \(\epsilon = 10^{-6}\) when \(f = (1-x^2)^2\). Figure \ref{fig2} clearly illustrates the failure of the conventional PINN method in approximating the solution to the singular perturbation problem. In contrast, our SL-PINN significantly outperforms the conventional PINN method.
The numerical results depicted in Figure \ref{fig2} and Table \ref{tab:1} provide compelling evidence that the semi-analytic SL-PINN significantly outperforms the conventional PINN method. This improvement is attributed to the incorporation of a corrector function within the scheme. Our SL-PINN, enhanced with the corrector, consistently generates stable and accurate approximate solutions, regardless of the small parameter \(\ep\).
\begin{figure}
\centering
\includegraphics[scale=0.4]{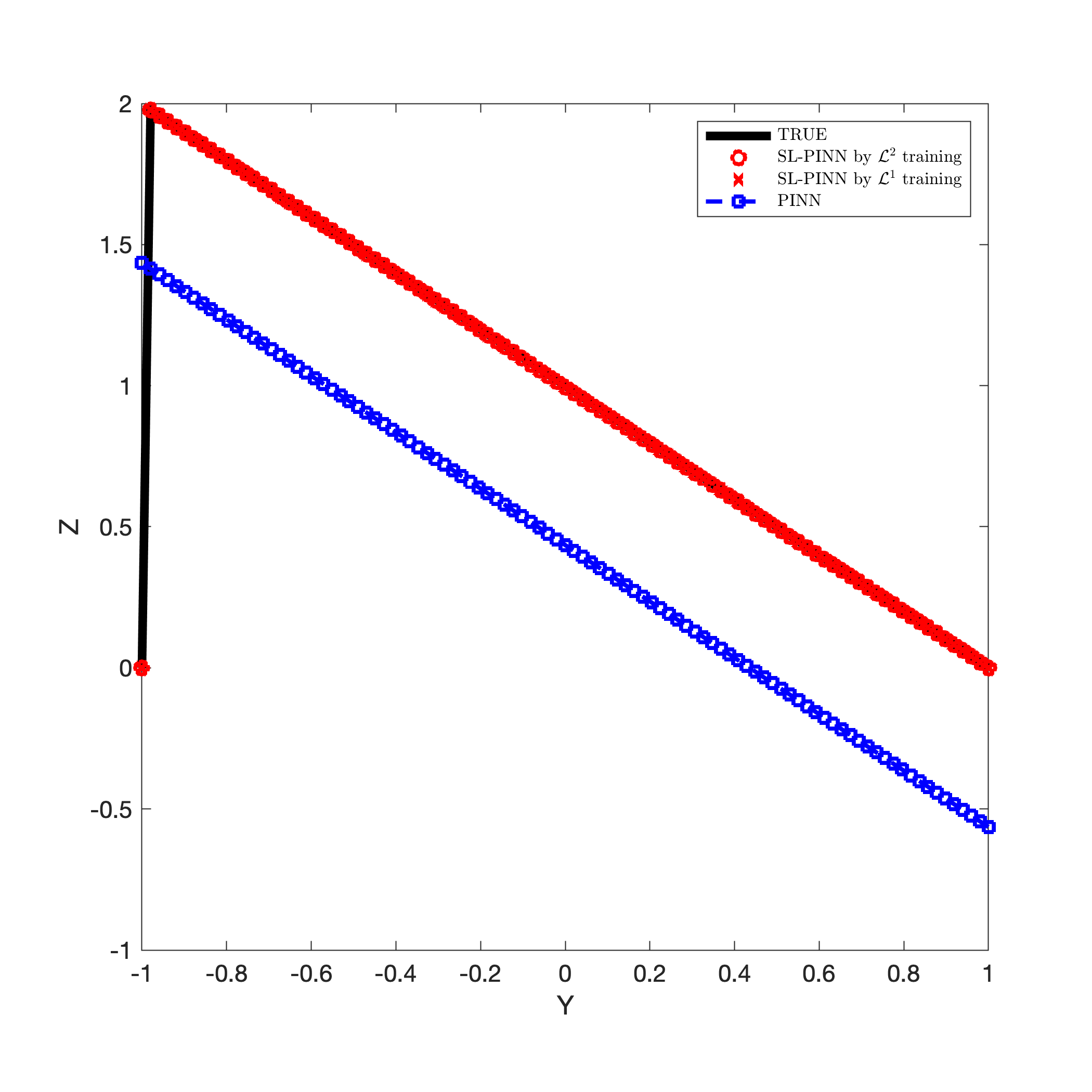}
\caption{The one-dimensional profile of predicted solutions along the line $\tau = \pi/2$.} \label{fig2-2}
\end{figure}
For a more detailed analysis, Figure \ref{fig2-2} provides a closer examination of the one-dimensional profile of predicted solutions at \(\tau = \pi/2\), enabling a clear and direct comparison. Figure \ref{fig2-2} demonstrates that both the \(L^1\) and \(L^2\) training approaches yield highly accurate approximations. However, the conventional PINN method fails to capture the sharp transition near the boundary layer.

To delve deeper into the subject, we consider a non-compatible case where \(f=1\), as shown in Figure \ref{fig3}. As mentioned earlier, this non-compatible case introduces an additional singularity when \(f\) does not vanish at the characteristic points \((\pm 1, 0)\). Analyzing the boundary layer in this scenario requires careful treatment, as discussed in \cite{JT13}. However, our SL-PINN approach effectively resolves the boundary layer behavior, overcoming the limitations posed by theoretical singularities.
We compare the performance of the standard five-layer PINN approximation with our novel SL-PINN approximation. Figure \ref{fig3} clearly demonstrates the failure of the conventional PINN method in accurately approximating the solution to the singular perturbation problem. In contrast, our SL-PINN, utilizing a two-layer neural network with a small number of neurons, significantly outperforms the conventional PINN method. The numerical results presented in Figure \ref{fig3} and Table \ref{tab:1} provide compelling evidence of the superior performance of our SL-PINN approach over the conventional PINN method.
For a closer examination, Figure \ref{fig3-3} provides a detailed one-dimensional profile of the predicted solutions at \(\tau=\pi/2\), enabling a clear and direct comparison. Figure \ref{fig3-3} shows the highly accurate approximations obtained through both the \(L^1\) and \(L^2\) training approaches. However, the conventional PINN method falls short in capturing the sharp transition near the boundary layer.

\begin{figure}[H]
     \centering
     \begin{subfigure}[b]{0.4\textwidth}
         \centering
         \includegraphics[width=\textwidth]{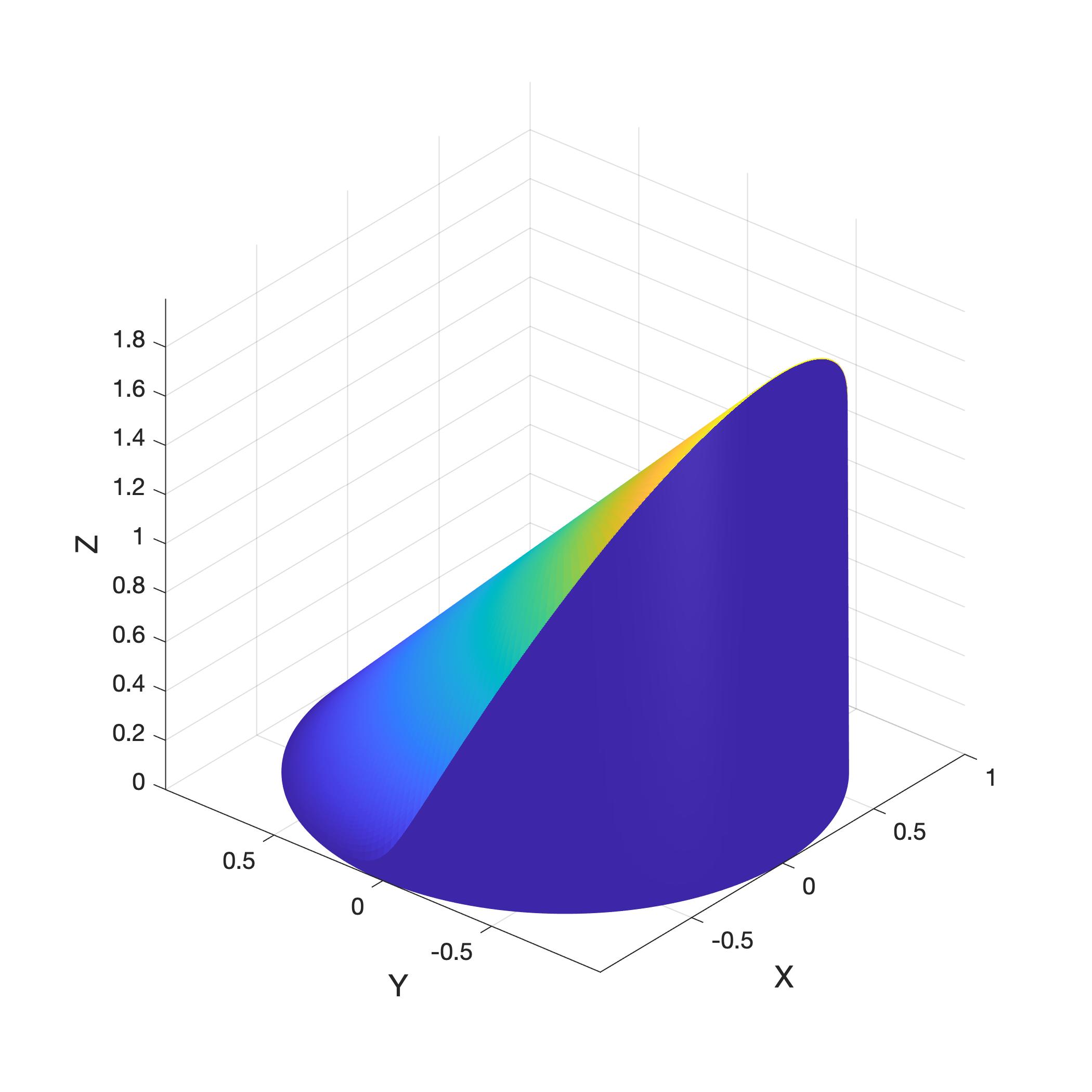}
         \caption{True solution}
     \end{subfigure}
     \begin{subfigure}[b]{0.4\textwidth}
         \centering
         \includegraphics[width=\textwidth]{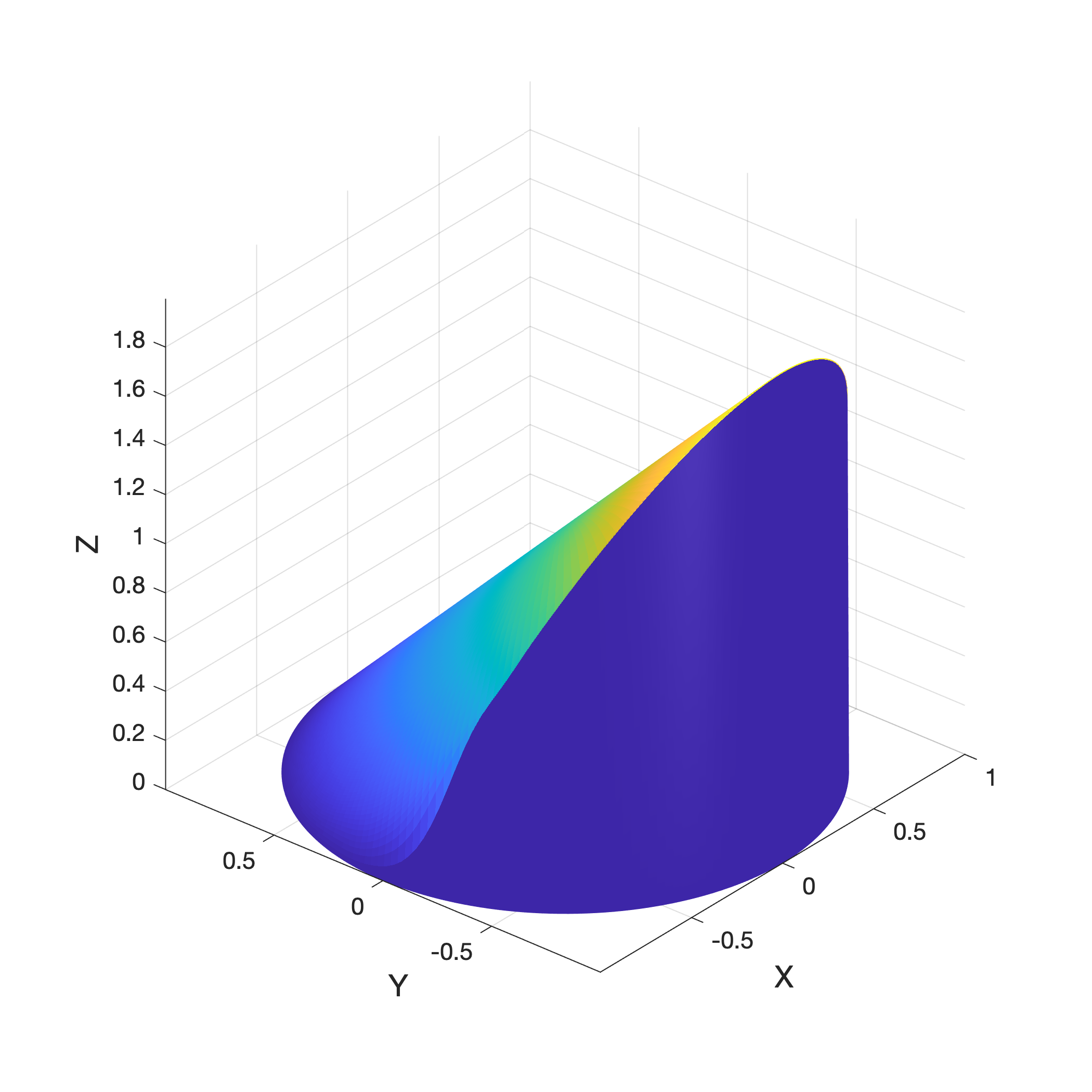}
         \caption{SL-PINN using $L^2$ training}
     \end{subfigure}
     \begin{subfigure}[b]{0.4\textwidth}
         \centering
         \includegraphics[width=\textwidth]{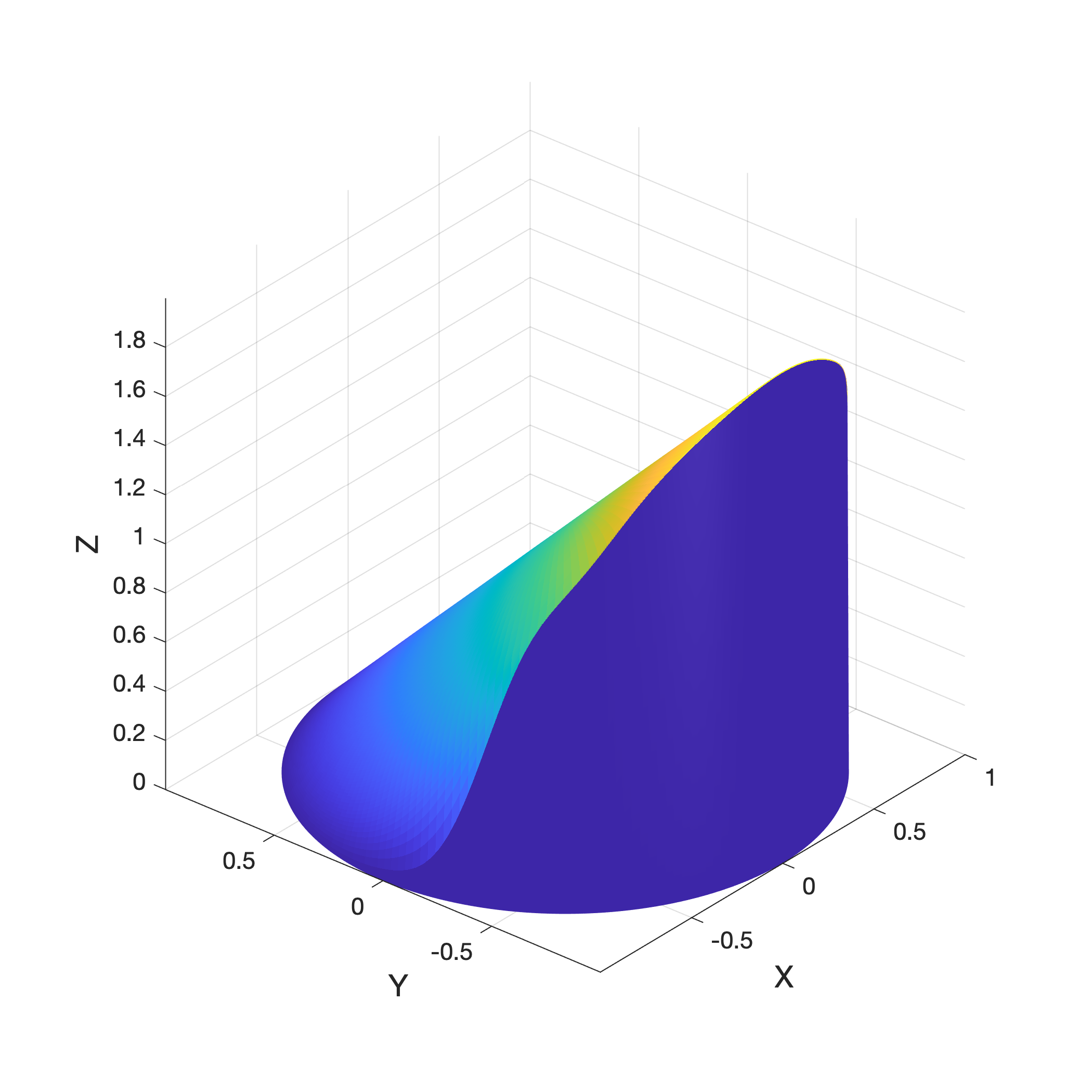}
         \caption{SL-PINN using $L^1$ training}
     \end{subfigure}
     \begin{subfigure}[b]{0.4\textwidth}
         \centering
         \includegraphics[width=\textwidth]{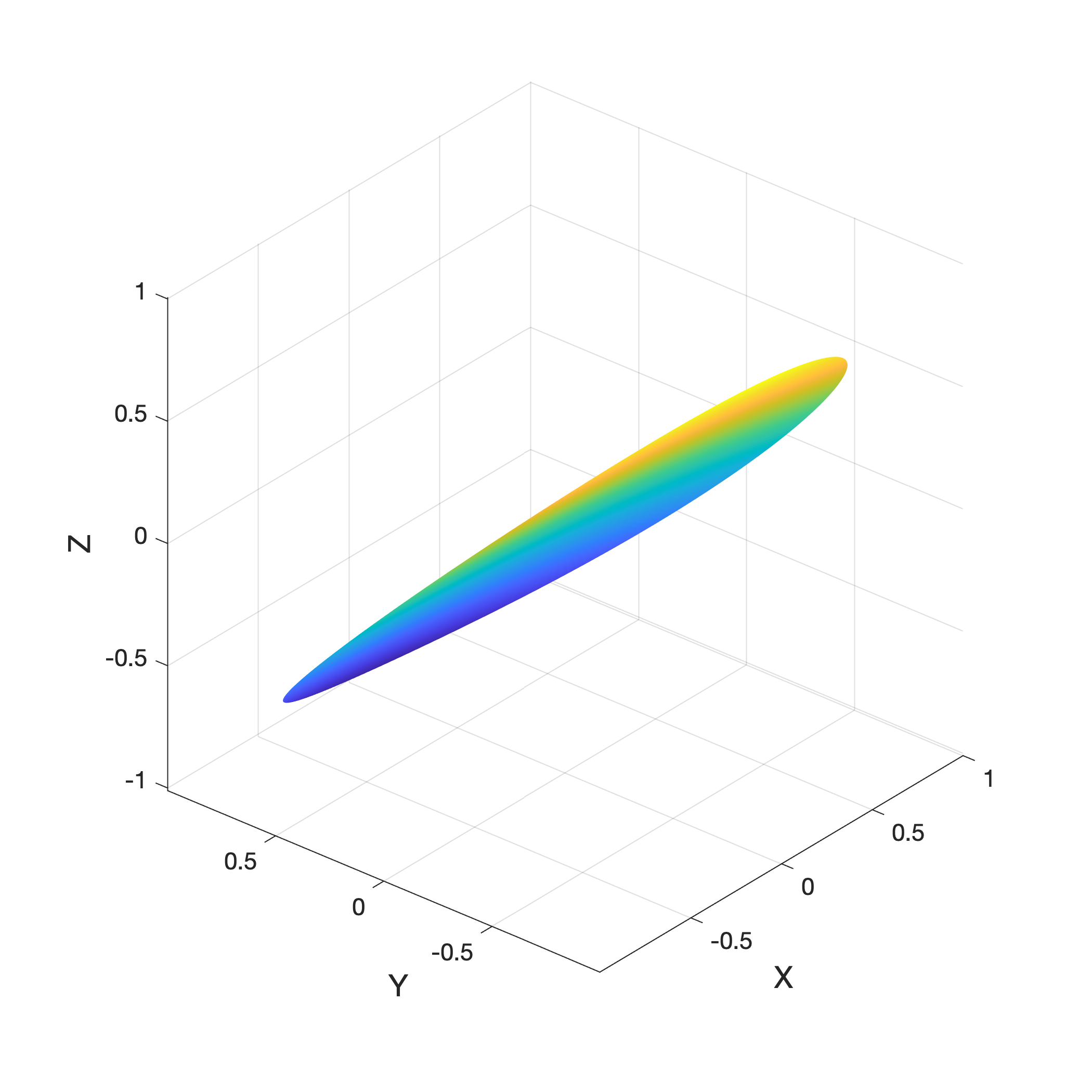}
         \caption{Conventional PINN}
     \end{subfigure}
        \caption{Numerical prediction of \eqref{e:circle_eq} with $\ep = 10^{-6}$ and $f=1$ (non-compatible case). For our simulations, we select a uniform grid of $50$ discretized points in each of the $\eta$ and $\tau$ directions.}\label{fig3}
\end{figure}

As an illustrative example, we apply our SL-PINN method to tackle a complex solution profile:
\be \label{e:f_oscillation}
    u^{\ep}=-A(x)(y+\sqrt{1-x^2})(1-e^{\frac{-2A(x)\sqrt{1-x^2}}{\epsilon}})+2A(x)(\sqrt{1-x^2})(1-e^{\frac{-A(x)(y+\sqrt{1-x^2})}{\epsilon}})
\ee 
In this case, the corresponding forcing term can be obtained through direct computation. The exact solution exhibits multiple humps, leading to an oscillatory numerical solution, as shown in Figure \ref{fig4}. Despite this complexity, our SL-PINN method effectively captures both the solution profile and the sharp transition near the boundary layer. We have excluded numerical experiments with the conventional PINN approach, as it consistently fails to produce an accurate solution profile.

\begin{figure}[H]
\centering
\includegraphics[scale=0.4]{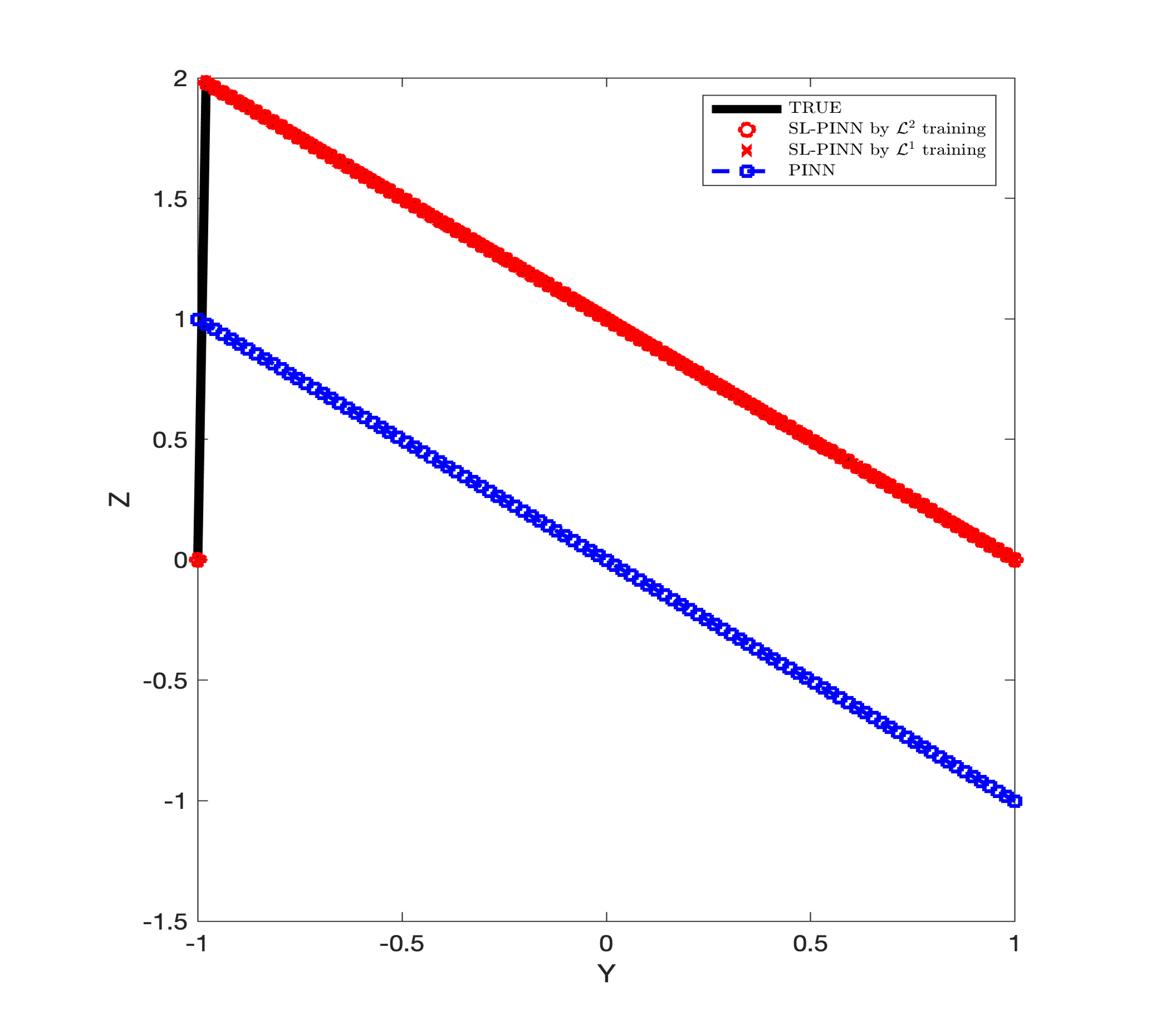}
\caption{The one-dimensional profile of predicted solutions along the line $\tau = \pi/2$.} \label{fig3-3}
\end{figure}

\begin{figure}[H]
    \centering
    \begin{subfigure}[b]{0.4\textwidth}
        \centering
        \includegraphics[width=\textwidth]{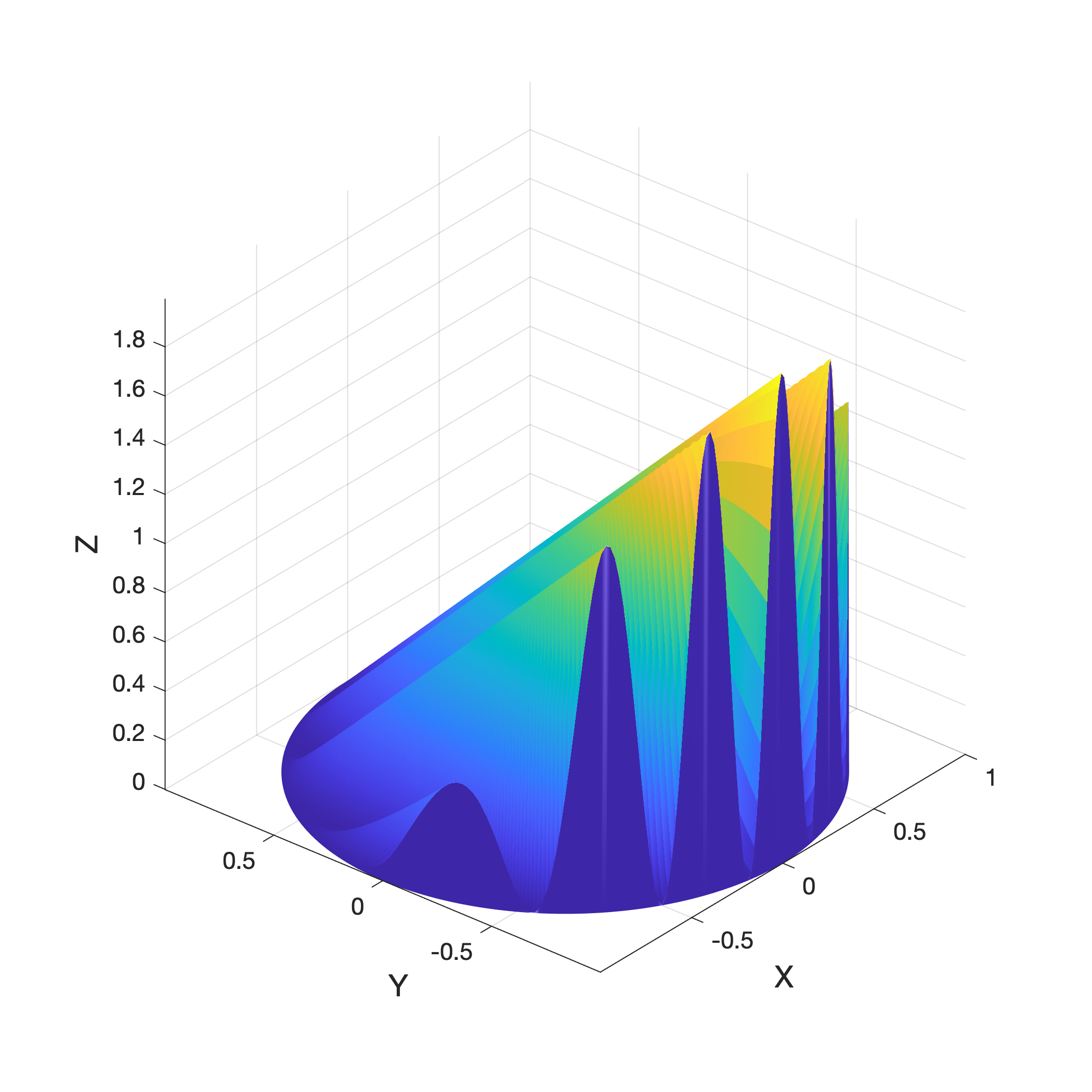}
        \caption{True solution}
    \end{subfigure}
    \begin{subfigure}[b]{0.4\textwidth}
        \centering
        \includegraphics[width=\textwidth]{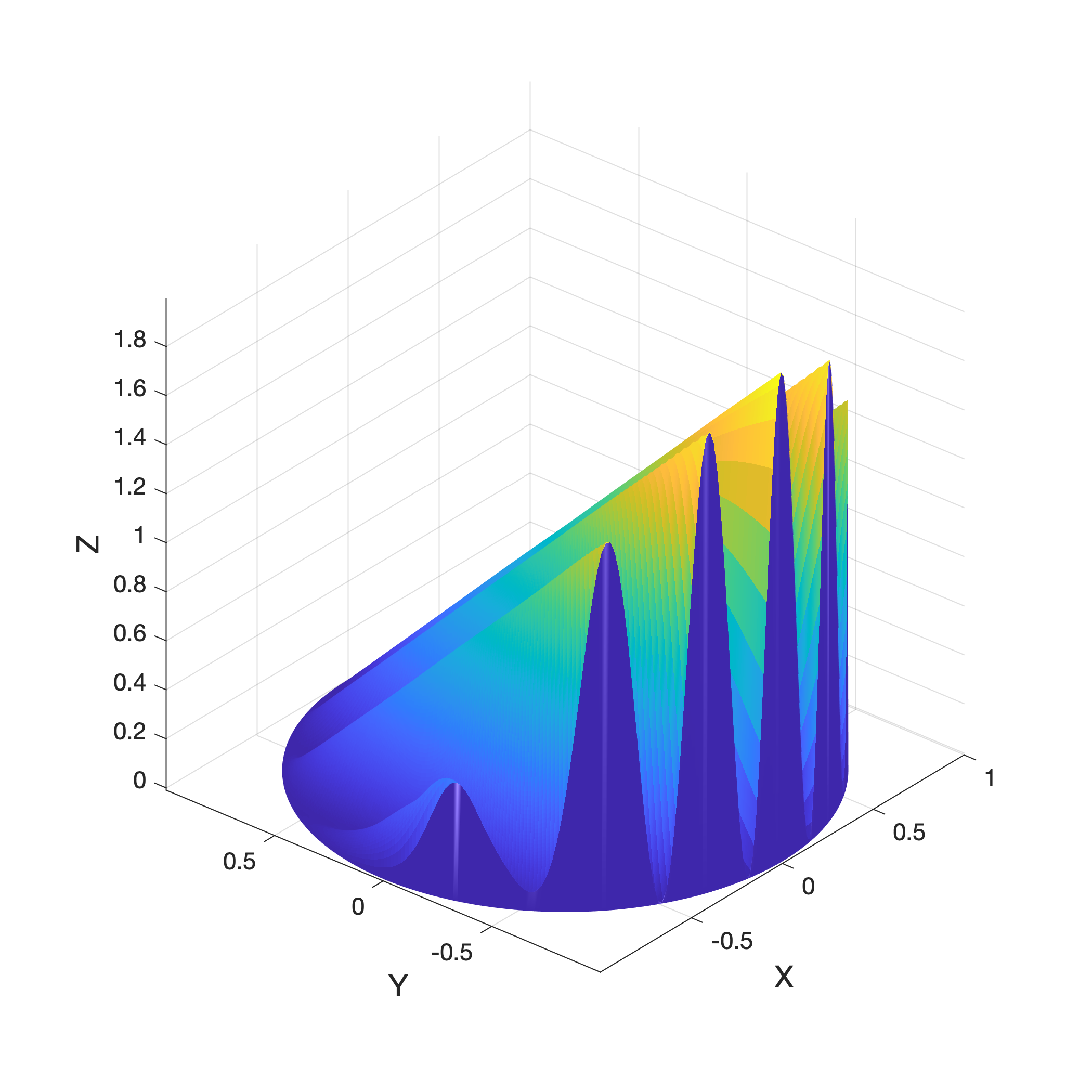}
        \caption{SL-PINN using $L^2$ training}
    \end{subfigure}
        \caption{Numerical prediction of \eqref{e:circle_eq} with $\ep = 10^{-6}$ and $f$ specified in \eqref{e:f_oscillation}. For our simulations, we select a uniform grid of $50$ discretized points in each of the $\eta$ and $\tau$ directions.}\label{fig4}
\end{figure}

\subsection{Numerical Simulations for Time-depedent problem } \label{subsec:time}

We extend our approach to {a} time dependent problem:
\begin{equation} \label{e:time_circle}
\begin{split}
   L_{\epsilon}u^{\ep} :u^{\ep}_{t}- \epsilon \Delta u^{\ep} - u^{\ep}_y & =f, \quad \text{ for } (x,y) \in \Omega, \quad t \in (0,T),\\
     u^{\ep}(x,y,t)  & = 0, \quad \text{  for  } (x,y) \in \partial \Omega, \quad t \in (0,T),\\
     u^{\ep}(x,y,0) & =0, \quad \text{ for } (x,y) \in \Omega, \quad \text{ at } t=0.\\
\end{split}    
\end{equation}
The analysis for time-dependent problems is similar to that for time-independent problems. For more detailed theoretical results, refer to the work by \cite{tspcircle}.
To facilitate the development of the boundary layer analysis, we employ boundary-fitted coordinates expressed as \(x = (1 - \eta) \cos \tau\) and \(y = (1 - \eta) \sin \tau\), as in the time-independent case.
By introducing the transformation $u^{\ep}(x,y,t)=v^{\ep}(\eta,\tau,t)$, we obtain that
\begin{equation}
\begin{aligned}
 P_{\epsilon}v^{\epsilon} &  :=  \ v^{\epsilon}_{t} + \frac{1}{(1-\eta)^2} ( -\epsilon ( v^{\epsilon}_{\tau \tau} 
- (1-\eta) v^{\epsilon}_{\eta} + (1-\eta)^2 v^{\epsilon}_{\eta \eta} ) \\
& \qquad + (1-\eta)^2 (\sin \tau) v^{\epsilon}_{\eta} - (1-\eta) (\cos \tau) v^{\epsilon}_{\tau} ) \\
& = f, \quad  \text{in } (\eta, \tau) \in D, \ t \in (0,T), \\
 v^{\epsilon}(0, \tau, t)  & =  \ 0, \quad  \text{at } 0 \leq \tau \leq 2\pi, \ t \in (0,T), \\
 v^{\epsilon}(\eta, \tau, 0)  & =  \ 0, \quad  \text{in } (\eta, \tau) \in D.
\end{aligned}
\end{equation}
As stated in the boundary layer analysis conducted by \cite{hong2015singular}, the sharp transition occurs solely in the spatial direction. Consequently, the solution does not exhibit a boundary layer in the temporal direction. Therefore, the time-independent problem yields the same form of the corrector equation as in the case without time dependence. Thus, we can derive the corrector equation as described by \cite{hong2015singular}:
For $t>0$,
\be
\begin{split}
     - \ep \varphi^0_{\eta \eta} +(\sin \tau) \varphi^0_{\eta}  & = 0, \quad  {\text{  for  } 0 < \eta < 1, \quad \pi < \tau < 2 \pi,}  \\
     \varphi^0  & = {- u^0(t,\cos \tau, \sin \tau)}, \quad {\text{  at  } \eta=0 },\\
     \varphi^0  & \rightarrow 0 \quad \text{ as } \eta \rightarrow 1.
\end{split}
\ee
Thus, we can find an explicit solution to the corrector equation:
\be
\varphi^0 = -u^0(t,\cos\tau, \sin\tau)\exp\left(\frac{\sin \tau}{\ep} \eta \right) \chi_{[\pi, 2\pi]}(\tau),
\ee
where $\chi$ stands for the characteristic function.
For our numerical scheme, we introduce an approximate form,
\be
\bar \varphi^0 = -u^0(t,\cos\tau, \sin\tau)\exp\left(\frac{\sin \tau}{\ep} \eta \right) \chi_{[\pi, 2\pi]}(\tau) \delta(\eta),
\ee
where $\delta(\eta)$ is a smooth cut-off function such that $\delta(\eta) = 1$ for $\eta \in [0, 1/2]$ and $=0$ for $\eta \in [3/4, 1]$.

We establish the {\it SL-PINN} based on boundary layer analysis, incorporating the profile of the corrector such that
\begin{equation}\label{e:EPINN_scheme_time}
    \widetilde{v}(t, \eta,\tau ; \, {\blds \theta}) 
            =(e^t-1)
            \left(\hat{v}(t, \eta,\tau,{\blds \theta})-\hat{v}(t, 0,\tau,{\blds \theta}){\hat{ \varphi}^0(\eta,\tau)}\right)C(\eta,\tau),
\end{equation}
{
where $\hat{ \varphi}^0$ is given by 
\begin{equation}
     \hat{ \varphi}^0=\exp\left(\frac{\sin \tau}{\ep} \eta \right) \chi_{[\pi, 2\pi]}(\tau) \delta(\eta),
\end{equation}}
and 
\begin{equation}\label{e:circle_time_C_term}
\begin{split}
    C(\eta,\tau)=\begin{cases}
    1-(1-\eta)^3 , \text{ if } 0 \leq \tau \leq \pi\\
    1-(1-\eta)^3-((1-\eta)\sin\tau)^3, \text{ if } \pi < \tau <2\pi,
    \end{cases}
\end{split}
\end{equation}
and $\hat{v}(t,\eta,\tau,{\blds \theta}) = \hat{u}(t,x,y,{\blds \theta}) = W_{4}{\sigma}(W_{1}x+W_{2}y+W_{3}t+b)$.
Then, the residual loss is defined by
\begin{equation}\label{e:circle_loss_time}
\begin{split}
    Loss=\left(\frac{1}{N}\sum_{i=0}^{N} |P_{\epsilon}\widetilde{v}((t_i,\eta_{i},\tau_{i},{\blds \theta}))-f|^{p} \right)^{1/p}  \quad  \text{ for }(t_i,\eta_{i},\tau_{i}) \in  [0,T] \times D,
\end{split}    
\end{equation}
where $p=1,2$.
Due to the boundary layer behavior near the lower semi-circle, we split the residual loss \eqref{e:circle_loss_time} into two sections: $0 < \tau < \pi$ and $\pi \leq \tau \leq 2 \pi$. For the $0 < \tau < \pi$, calculating the residual loss is simple:
\begin{equation}
\begin{split}\label{e:EPINN_circle_direct_calculation_upper}
    & P_{\epsilon}\widetilde{v}((t,1-\eta,\tau,{\blds \theta})) - f \\
    & = \left( e^{t} \hat{v} + (e^{t} - 1) \hat{v}_{t}(t,\eta,\tau,{\blds \theta}) \right) (1 - (1-\eta)^3) \\
    & \quad + (e^{t} - 1) \left( -\ep \left( \frac{1}{(1-\eta)^2} - (1-\eta) \right) \hat{v}_{\tau \tau}(t,\eta,\tau,{\blds \theta}) 
    + \left( (1-\eta)^2 - \frac{1}{(1-\eta)} \right) \cos(\tau) \hat{v}_{\tau}(t,\eta,\tau,{\blds \theta}) \right. \\
    & \quad - \ep (1 - (1-\eta)^3) \hat{v}_{\eta \eta}(t,\eta,\tau,{\blds \theta}) 
    - \left( 6\ep + \ep (1-\eta)^2 - \frac{\ep}{(1-\eta)} - \sin(\tau) (1 - (1-\eta)^3) \right) \hat{v}_{\eta}(t,\eta,\tau,{\blds \theta}) \\
    & \quad + \left. \left( 3\ep + \frac{6\ep}{(1-\eta)} + 3 (1-\eta)^2 \sin(\tau) \right) \hat{v}(t,\eta,\tau,{\blds \theta}) \right) - f 
\end{split}
\end{equation}
for $0 \leq \eta < 1 ,\quad 0 < \tau < \pi$ and$\quad 0 \leq t \leq T$.
Introducing the boundary layer element in equation \eqref{e:EPINN_scheme_time} and adding the regularizing term in equation \eqref{e:circle_time_C_term} makes the residual loss calculation more complex when 
$\tau$ ranges from $\pi$ to $2\pi$ such that

\begin{equation}\label{e:EPINN_circle_direct_calculation}
\begin{split}
    & P_{\epsilon}\widetilde{v}(t,\eta,\tau,{\blds \theta}) - f \\
    & = \left( e^{t} \hat{v}(t,\eta,\tau,{\blds \theta}) + (e^{t} - 1) \hat{v}_{t}(t,\eta,\tau,{\blds \theta}) 
    - \left( e^{t} \hat{v}(t,0,\tau,{\blds \theta}) + (e^{t} - 1) \hat{v}_{t}(t,0,\tau,{\blds \theta}) \right) \right) \\
    & \quad \times \exp\left( \frac{\sin \tau}{\ep} \eta \right) (1 - (1-\eta)^3 - (1-\eta)^3 \sin^3(\tau)) \\
    & \quad + (e^{t} - 1) \left( -\ep \left( \frac{1}{(1-\eta)^2} - (1-\eta) - (1-\eta) \sin^3(\tau) \right) \hat{v}_{\tau \tau}(t,\eta,\tau,{\blds \theta}) \right. \\
    & \quad + \left( 6 \ep (1-\eta) \sin^2(\tau) \cos(\tau) + \left( (1-\eta)^2 - \frac{1}{(1-\eta)} + (1-\eta)^2 \sin^3(\tau) \right) \cos(\tau) \right) \hat{v}_{\tau}(t,\eta,\tau,{\blds \theta}) \\
    & \quad -\ep \left( 1 - (1-\eta)^3 - (1-\eta)^3 \sin^3(\tau) \right) \hat{v}_{\eta \eta}(t,\eta,\tau,{\blds \theta}) \\
    & \quad + \left( 6\ep + \ep (1-\eta)^2 - \frac{\ep}{(1-\eta)} + 7\ep (1-\eta)^2 \sin^3(\tau) + (1-\eta)^3 \sin^3(\tau) - (1 - (1-\eta)^3) \sin(\tau) \right) \hat{v}_{\eta}(t,\eta,\tau,{\blds \theta}) \\
    & \quad + \left( 3\ep + \frac{6\ep}{(1-\eta)} + 6 \ep (1-\eta) \sin(\tau) + 6 (1-\eta)^2 \sin(\tau) \right) \hat{v}(t,\eta,\tau,{\blds \theta}) \\
    & \quad + \hat{v}(t,0,\tau,{\blds \theta}) \big[ \ep (C(\eta,\tau) \delta_{\eta \eta} + ((1-\eta) C(\eta,\tau) + 2 C_{\eta}(\eta,\tau)) \delta_{\eta} \\
    & \quad - ((1-\eta) C_{\eta}(\eta,\tau) + C_{\eta \eta}(\eta,\tau)) \delta ) - C(\eta,\tau) \sin(\tau) \delta_{\eta} + C_{\eta}(\eta,\tau) \sin(\tau) \delta \big] \exp\left( \frac{\sin \tau}{\ep} \eta \right) \\
    & \quad + \frac{\delta}{(1-\eta)^2} \big[ (\ep \hat{v}_{\tau \tau}(t,0,\tau,{\blds \theta}) + (2+\eta) \cos(\tau) \hat{v}_{\tau}(t,0,\tau,{\blds \theta}) - \eta \sin(\tau) \hat{v}(t,0,\tau,{\blds \theta})) C(\eta,\tau) \\
    & \quad + (2\ep \hat{v}_{\tau}(t,0,\tau,{\blds \theta}) + (1+\eta) \cos(\tau) \hat{v}(t,0,\tau,{\blds \theta})) C_{\tau}(\eta,\tau) + \ep \hat{v}(t,0,\tau,{\blds \theta}) C_{\tau \tau}(\eta,\tau) \big] \exp\left( \frac{\sin \tau}{\ep} \eta \right) \\
    & \quad + (e^{t} - 1) C(\eta,\tau) \hat{v}(t,0,\tau,{\blds \theta}) \frac{\delta(\eta)}{(1-\eta)^2} \left( \frac{\eta \cos^2(\tau)}{\ep} \right) \exp\left( \frac{\sin \tau}{\ep} \eta \right) - f,
\end{split}
\end{equation}

where $0\leq \eta < 1$, $\pi \leq \tau \leq 2\pi$, $0 \leq t \leq T$. 
To make computation in \eqref{e:EPINN_circle_direct_calculation} feasible, we extract the largest order term in $\ep$, which includes $O(\ep^{-1})$ such that
\begin{equation}
    \psi(t,\eta,\tau,{\blds \theta})=(e^{t}-1)(1-(1-\eta)^3-(((1-\eta)\sin\tau)^3)\hat{v}(t,0,\tau,{\blds \theta})\frac{\delta(\eta)}{(1-\eta)^2}(\frac{\eta}{\ep}\cos^2\tau)\exp\left(\frac{\sin \tau}{\ep} \eta \right) \chi_{[\pi, 2\pi]}(\tau).
\end{equation}
When $p=1$, the loss in \eqref{e:circle_loss_time} becomes, by the  triangular inequality,
\begin{equation} \label{e:bound_time_0}
 Loss  \leq \frac{1}{N}\sum_{i=0}^{N}|\psi (t_{i},\eta_{i},\tau_{i},{\blds \theta})|
 +  \frac{1}{N}\sum_{i=0}^{N} |P_{\epsilon}\widetilde{v}((t_{i},\eta_{i},\tau_{i},{\blds \theta}))-\psi (t_{i},\eta_{i},\tau_{i},{\blds \theta})-f|.
\end{equation}
The rightmost term does not include a large order term such as \(\ep^{\alpha}\) (where \(\alpha < 0\)) since \(\psi\) contains large terms like \(\ep^{-1}\). Hence, this part can be computed using conventional \(L^1\) or \(L^2\) loss. Assume that \(\boldsymbol{\theta}\) is fixed and that we select a sufficiently large number of sampling points \(N\) such that
\begin{equation}\label{e:normal_bound_time}
\begin{aligned}
    \frac{1}{N}\sum_{i=0}^{N} &\left|P_{\epsilon}\widetilde{v}(t_{i}, \eta_{i}, \tau_{i}, \boldsymbol{\theta}) - \psi(t_{i}, \eta_{i}, \tau_{i}, \boldsymbol{\theta}) - f\right| \\
    &\approx \frac{1}{T\pi}\int_{0}^{T}\int_{0}^{2\pi} \int_{0}^{1} \left|P_{\epsilon}\widetilde{v}(t, \eta, \tau, \boldsymbol{\theta}) - \psi(t, \eta, \tau, \boldsymbol{\theta}) - f\right| \, d\eta \, d\tau \, dt < \infty.
\end{aligned}
\end{equation}
Given that \(\psi\) includes \(\mathcal{O}(\epsilon^{-1})\), we handle \(\psi\) using the \(L^1\) loss with the bound \eqref{known_upper_bound}.
Assume we choose a sufficiently large number of sampling points \(N\):
\begin{equation}\label{e:bound00 time}
\begin{split}
     & \frac{1}{N}\sum_{i=0}^{N}|\psi(t_{i},\eta_{i},\tau_{i},{\blds \theta})| \approx \frac{1}{T \pi}\int^{T}_{0}\int^{2\pi}_{\pi} \int^{1}_{0} |\psi(t,\eta,\tau,{\blds \theta})|d \eta d \tau d t.\\
\end{split}
\end{equation}
Note that the right-hand side of equation (\ref{e:bound00 time}) is bounded by \(C\). Using calculations analogous to those in equations \eqref{e:bound01} through \eqref{e:bound03}, we can deduce the following result:
\begin{equation} \label{e:bound_time_1}
\begin{split}
     \frac{1}{T \pi}\int^{T}_{0}\int^{2\pi}_{\pi} \int^{1}_{0} |\psi(t,\eta,\tau,{\blds \theta})|d \eta d \tau d t
     < C+C\ep e^{-\frac{1}{2\ep}}.
\end{split}
\end{equation}
Thus, the loss \eqref{e:circle_loss_time} becomes nearly constant due to \eqref{e:bound00 time} and \eqref{e:normal_bound_time} when \(\epsilon\) is sufficiently small. This simplification allows for feasible computations in the SL-PINN approach.

The effectiveness of our new approach is demonstrated through a series of numerical experiments, as shown in Figure \ref{fig5}. For these experiments, we use a grid discretization with 50 uniformly spaced points in each direction of \(\tau\), \(\eta\), and \(t\). The function \(f\) employed in these experiments satisfies the compatibility condition mentioned in \cite{hong2015singular}:
\be \label{e:time_force}
f=(1-x^2)^2 t+(1-x^2)^{2}(-y+\sqrt{1-x^2}+\frac{\epsilon(y+\sqrt{1-x^2})}{(1-x^2)^{3/2}})+O(\epsilon).
\ee
Also, the corresponding $u$ is given by
\begin{equation}
\begin{split}
    u(t,x,y)=\begin{cases}
    t(1-x^{2})^{2}(-y+\sqrt{1-x^2}+\epsilon \frac{(y+\sqrt{1-x^2})}{(1-x^2)^{3/2}} , \text{ in } D\\
    0, \text{ on }\partial D.
    \end{cases}
\end{split}
\end{equation}
Figure \ref{fig5} presents the numerical solution obtained using SL-PINN with \(\epsilon = 10^{-6}\). The conventional PINNs are not included in the figure due to their significant deviation from the true solution. Enhanced with the corrector function, the SL-PINN consistently produces stable and accurate approximate solutions, regardless of the small parameter \(\epsilon\); see e.g, Table \ref{tab:1}.
Figure \ref{fig5}(C) shows the one-dimensional predicted solution profiles at \(\tau = \pi/2\) for various time instances, including \(t = 0, 0.25, 0.5, 0.75\), and \(1.0\), alongside the true solution. This example demonstrates that our SL-PINN method delivers accurate predictions throughout the entire time span.

\begin{figure}
    \centering
    \begin{subfigure}[b]{0.3\textwidth}
        \centering
        \includegraphics[width=\textwidth]{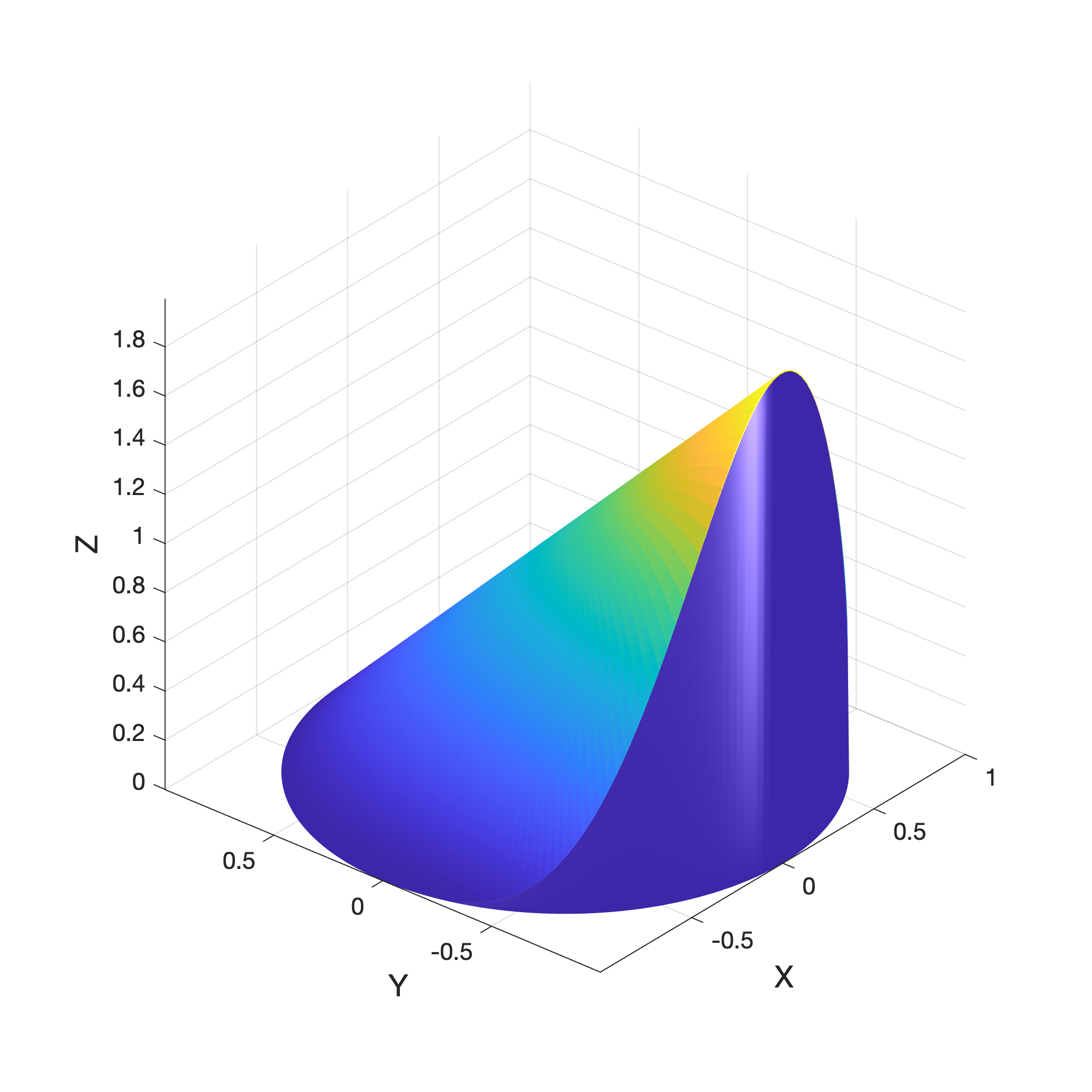}
        \caption{True soltuion at t=1}
    \end{subfigure}
    \begin{subfigure}[b]{0.3\textwidth}
        \centering
        \includegraphics[width=\textwidth]{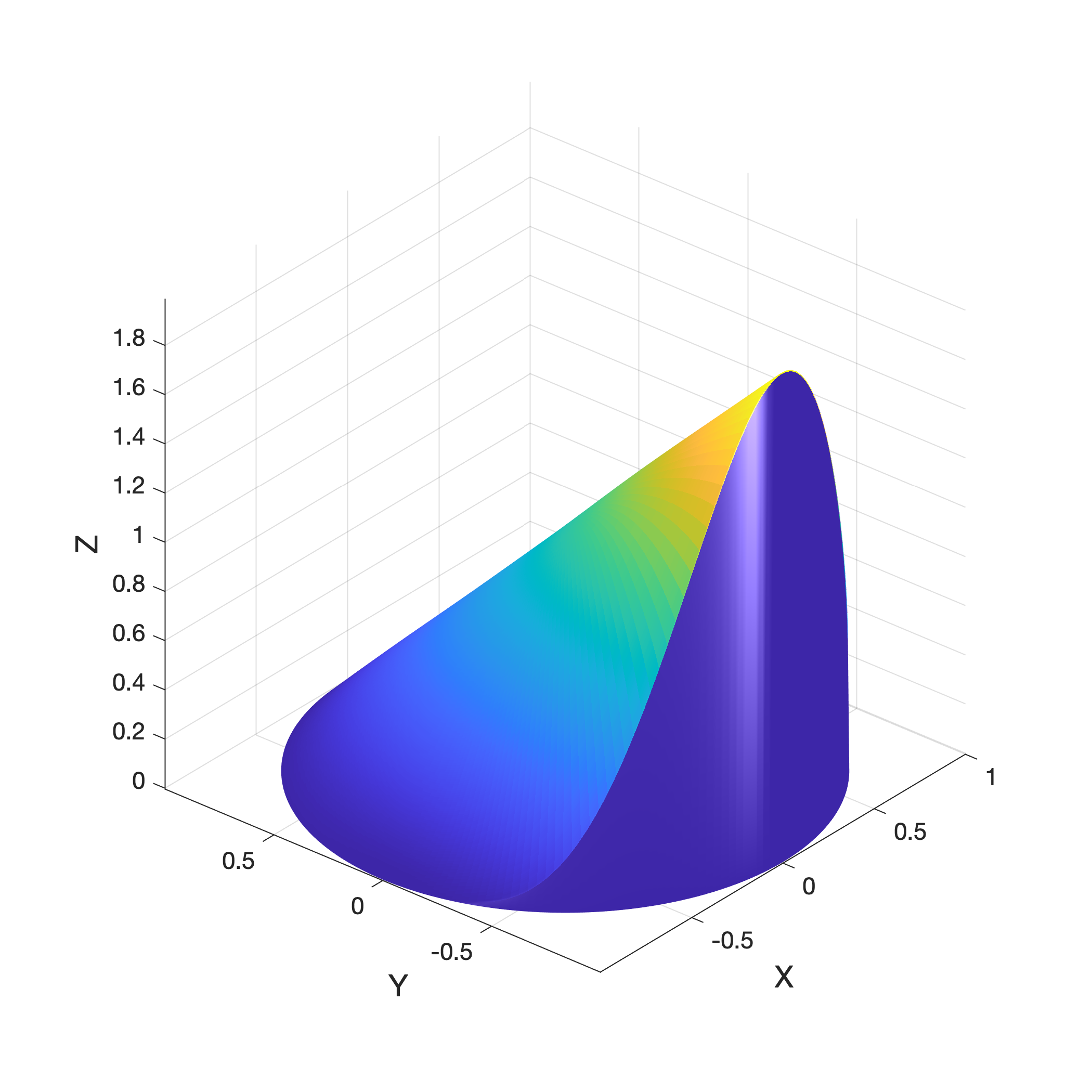}
        \caption{SL-PINN at t=1}
    \end{subfigure}
    \begin{subfigure}[b]{0.3\textwidth}
        \centering
        \includegraphics[width=\textwidth]{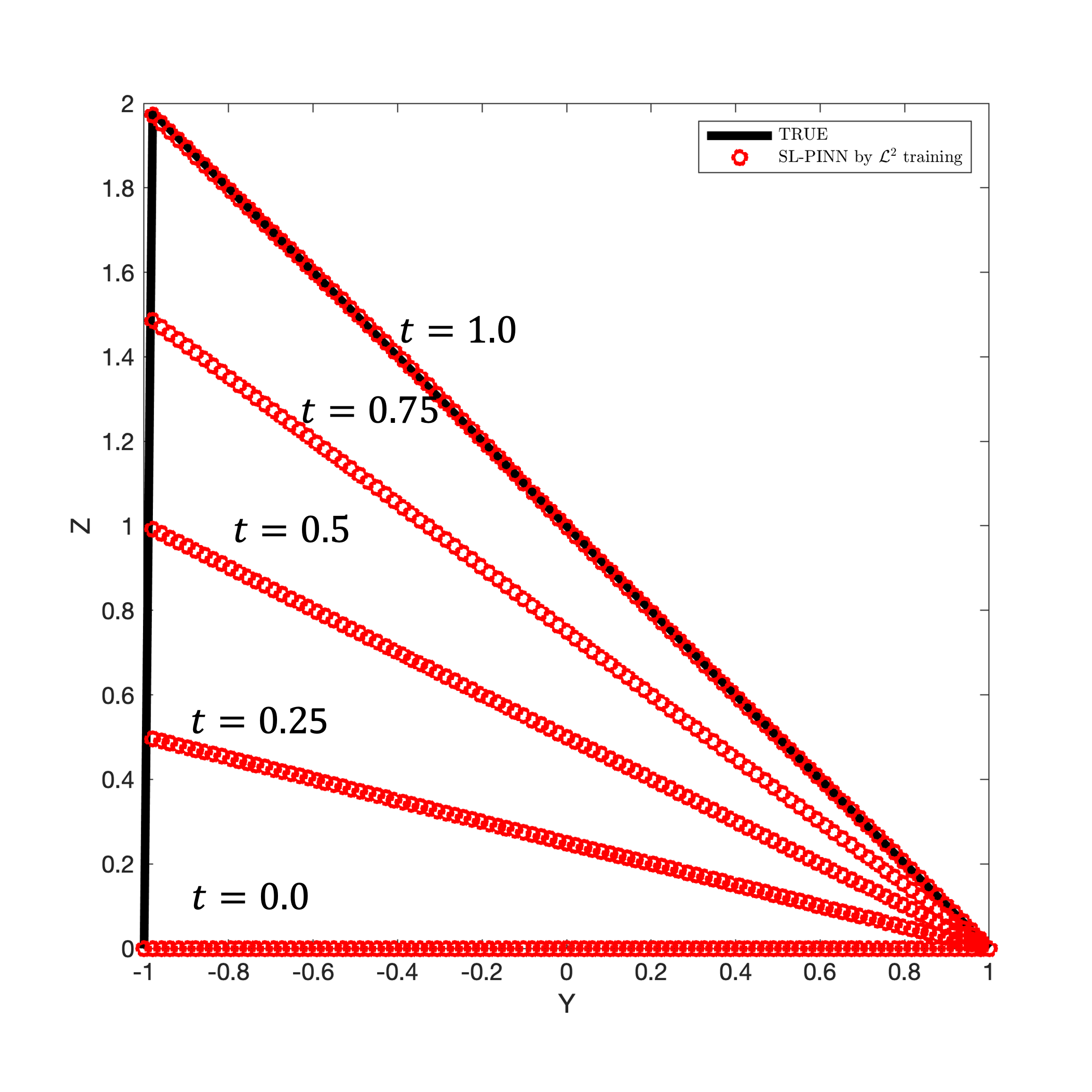}
        \caption{1D solution profiles}
    \end{subfigure}
        \caption{Numerical prediction of \eqref{e:time_circle} with $\ep = 10^{-6}$ and $f$ specified in \eqref{e:time_force}. For our simulations, we select a uniform grid of $50$ discretized points in each of the $r$ and $\tau$ directions. In the panel (c), the one-dimensional predicted solution profiles are displayed at $\tau = \pi/2$ for various time instances, including $t =0, 0.25, 0.5, 0.75$, and $1.0$.}\label{fig5}
\end{figure}

\subsection{Numerical Simulations for Non-linear problem}
To demonstrate the flexibility of our approach, we focus on the nonlinear convection-diffusion equation in a circular domain
{
\begin{equation}
\begin{split}
   L_{\epsilon}u^{\ep} {:=}- \epsilon \Delta u^{\ep} - u^{\ep}_y+(u^{\ep})^{3} & =f, \quad \text{  in  } \Omega=\{(x,y) | x^{2}+y^{2}< 1\}\\
    u^{\ep}  & = 0, \quad \text{  at  } \partial \Omega.
\end{split}    
\end{equation}
}As in Sections \ref{subsec:plain} and \ref{subsec:time}, we employ boundary-fitted coordinates \(x = (1 - \eta) \cos \tau\) and \(y = (1 - \eta) \sin \tau\). By substituting \(u^{\ep}(x, y) = v^{\ep}(\eta, \tau)\), the transformed equation becomes:
\begin{equation}
\begin{split}
P_{\epsilon} v^{\epsilon} 
    := \ \frac{1}{(1-\eta)^2} (
    -\epsilon ( v^{\epsilon}_{\tau \tau} - (1-\eta) v^{\epsilon}_{\eta} + (1-\eta)^{2} v^{\epsilon}_{\eta \eta} ) + (1-\eta)^{2} \sin \tau v^{\epsilon}_{\eta} - (1-\eta) \cos \tau v^{\epsilon}_{\tau} &+  (1-\eta)^{2} (v^{\epsilon})^3
    ) \\
    & = f, \quad \text{in } D, \\
v^{\epsilon}(0, \tau)  = 0, \quad &\text{at } 0 \leq \tau \leq 2\pi.
\end{split}
\end{equation}
When performing boundary layer analysis, the dominant terms are the same as in the linear case described in Section \ref{subsec:plain}. Therefore, the nonlinear terms do not affect the boundary layer analysis. Consequently, we employ the same corrector equation used in the linear case without further modification
\begin{equation}\label{e:nonlinr_cor}
\begin{aligned}
    - \ep \varphi^0_{\eta \eta} +(\sin \tau) \varphi^0_{\eta}  & = 0, \quad  {\text{  for  } 0 < \eta < 1, \quad \pi < \tau < 2 \pi,}  \\
     \varphi^0  & = - u^0(\cos \tau, \sin \tau), \quad {\text{  at  } \eta=0 },\\
     \varphi^0  & \rightarrow 0 \quad \text{ as } \eta \rightarrow 1.
\end{aligned}
\end{equation}
The explicit solution of \eqref{e:nonlinr_cor} is found as
\be
\varphi^0 = -u^0(\cos\tau, \sin\tau)\exp\left(\frac{\sin \tau}{\ep} \eta \right) \chi_{[\pi, 2\pi]}(\tau),
\ee
where $\chi$ stands for the characteristic function.
To match the boundary condition in our numerical scheme, we introduce a cut-off function to derive an approximate form 
\be
\bar \varphi^0 = -u^0(\cos\tau, \sin\tau)\exp\left(\frac{\sin \tau}{\ep} \eta \right) \chi_{[\pi, 2\pi]}(\tau) \delta(\eta),
\ee
where $\delta(\eta)$ is a smooth cut-off function such that $\delta(\eta) = 1$ for $\eta \in [0, 1/2]$ and $=0$ for $r \in [3/4, 1]$.

We now establish the semi-analytic SL-PINN method as
\begin{equation}
    \widetilde{v}(\eta,\tau ; \, {\blds \theta}) 
            =(\hat{v}(\eta,\tau, {\blds \theta})-\hat{v}(0,\tau, {\blds \theta}) {\hat{ \varphi}^0})C(\eta,\tau),
\end{equation}
where $\hat{ \varphi}^0$ is given by 
\begin{equation}
     \hat{ \varphi}^0=\exp\left(\frac{\sin \tau}{\ep} \eta \right) \chi_{[\pi, 2\pi]}(\tau) \delta(\eta),
\end{equation}
and
\begin{equation} 
\begin{split}
    C(\eta,\tau)=\begin{cases}
    1-(1-\eta)^3 , \text{ if } 0 < \tau < \pi,\\
    1-(1-\eta)^3-((1-\eta)\sin\tau)^3, \text{ if } \pi \leq \tau \leq 2\pi,
    \end{cases}
\end{split}
\end{equation}
with $\hat{v}(\eta,\tau , {\blds \theta}) = \hat{u}(x,y , {\blds \theta})$.
Then, the residual loss is defined by
\begin{equation}
\begin{split}
    Loss= \left( \frac{1}{N} \sum_{i=0}^{N} \left| P_{\epsilon}\widetilde{v}(\eta_{i}, \tau_{i}; \boldsymbol{\theta}) - f \right|^{p} \right)^{1/p} \quad \text{for} \ (\eta_{i}, \tau_{i}) \in D,
\end{split}    
\end{equation}
The residual loss was calculated similarly to \eqref{e:DD1c} to \eqref{eq:circle:complex}. Therefore, we omit the detailed derivation here.
Figure \ref{fig non-li} presents the numerical solution obtained using SL-PINN with \(\epsilon = 10^{-6}\); see, e.g., Table \ref{tab:1}. Unlike conventional PINNs, which struggle to capture the boundary layer, the SL-PINN, enhanced with the corrector function, consistently produces stable and accurate approximate solutions, regardless of the small parameter \(\epsilon\).

\begin{figure}[H]
    \centering
    \begin{subfigure}[b]{0.4\textwidth}
        \centering
    \includegraphics[width=\textwidth]{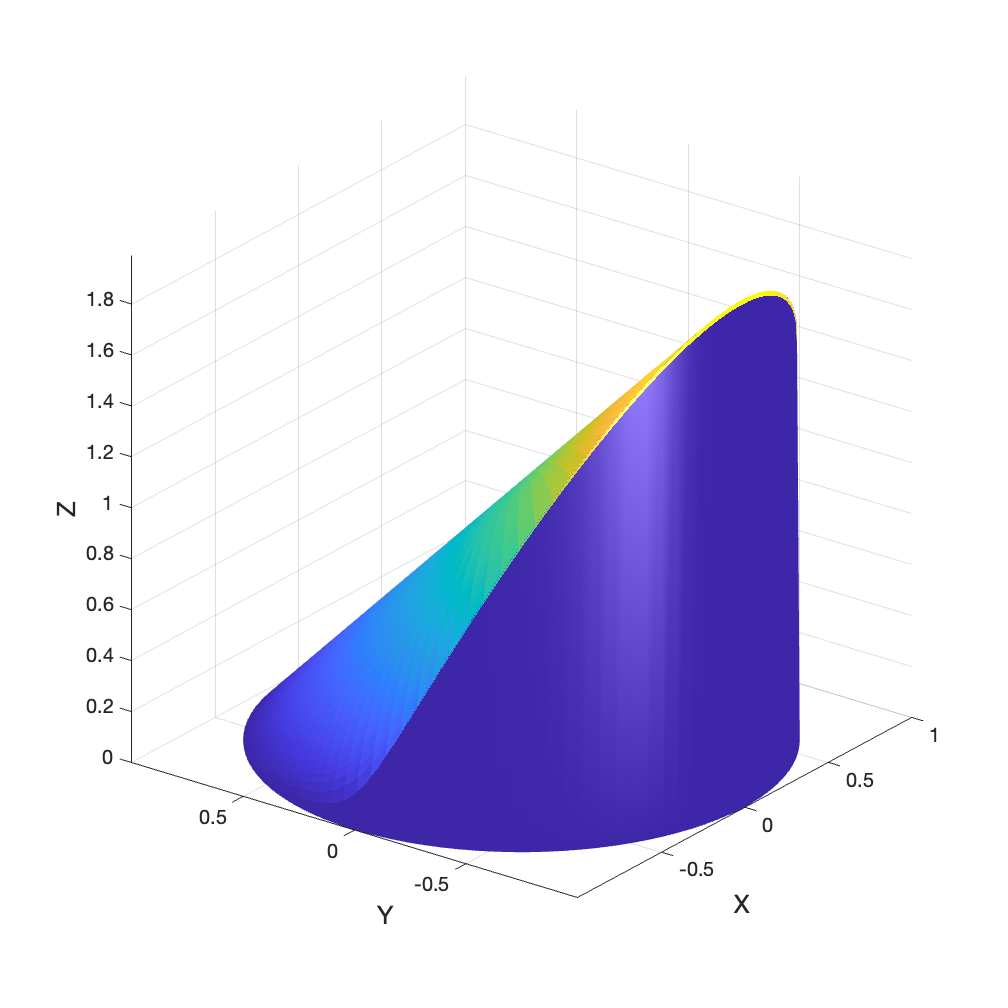}
        \caption{True solution}
    \end{subfigure}
    \begin{subfigure}[b]{0.4\textwidth}
    \centering
    \includegraphics[width=\textwidth]{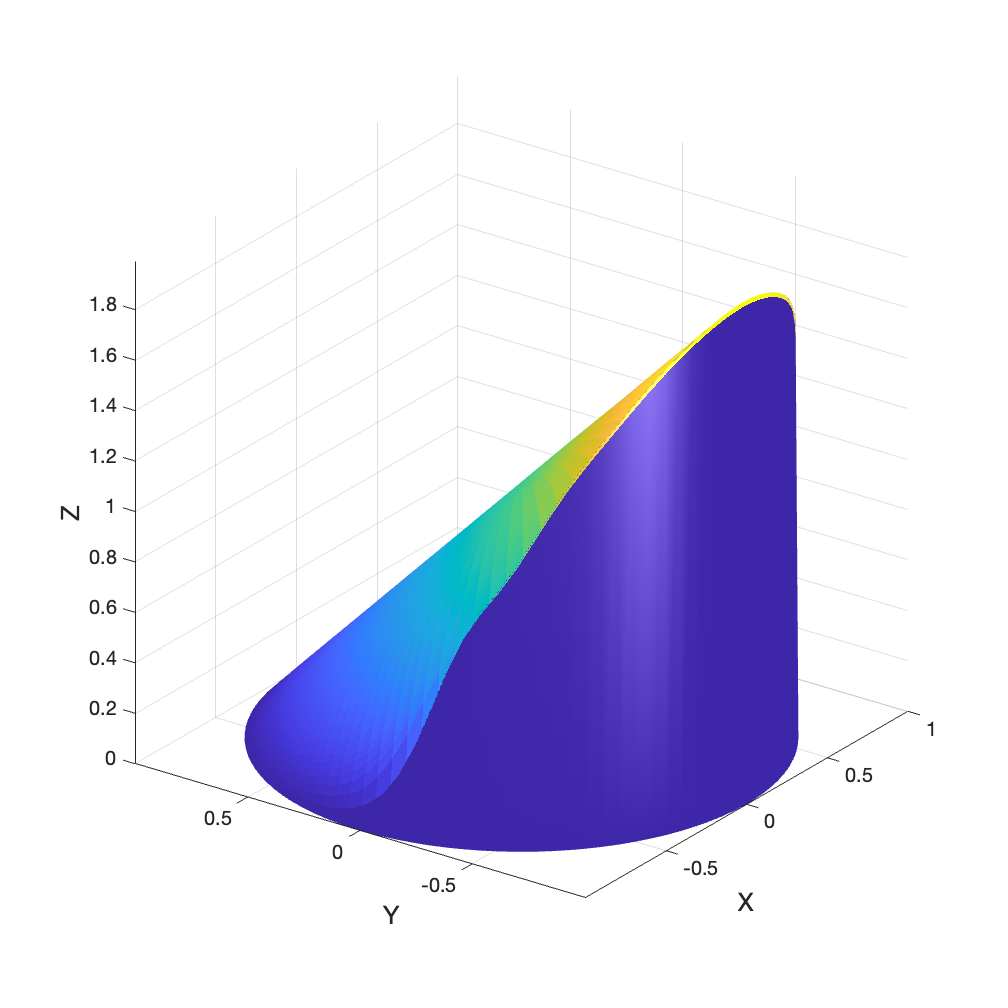}
        \caption{SL-PINN using $L^2$ training}
    \end{subfigure}
    \begin{subfigure}[b]{0.4\textwidth}
        \centering
    \includegraphics[width=\textwidth]{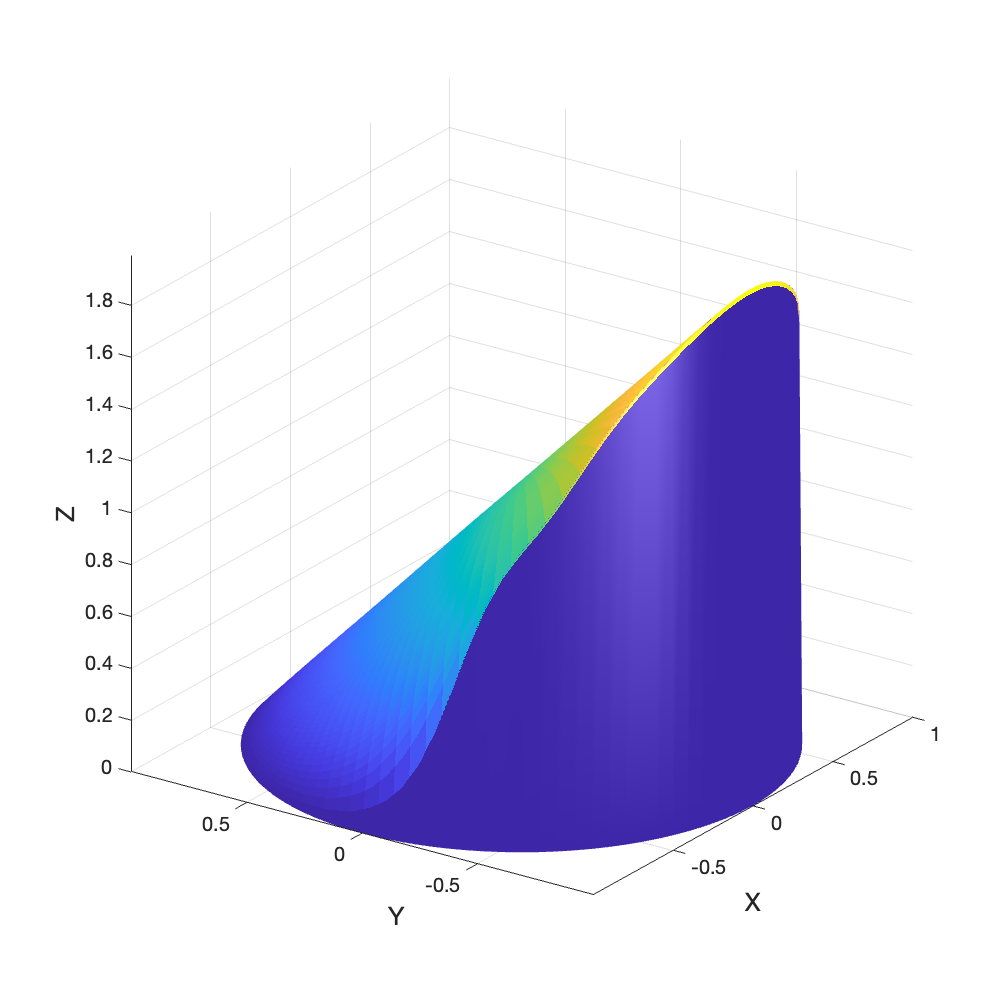}
        \caption{SL-PINN using $L^1$ training}
    \end{subfigure}
    \begin{subfigure}[b]{0.4\textwidth}
        \centering
    \includegraphics[width=\textwidth]{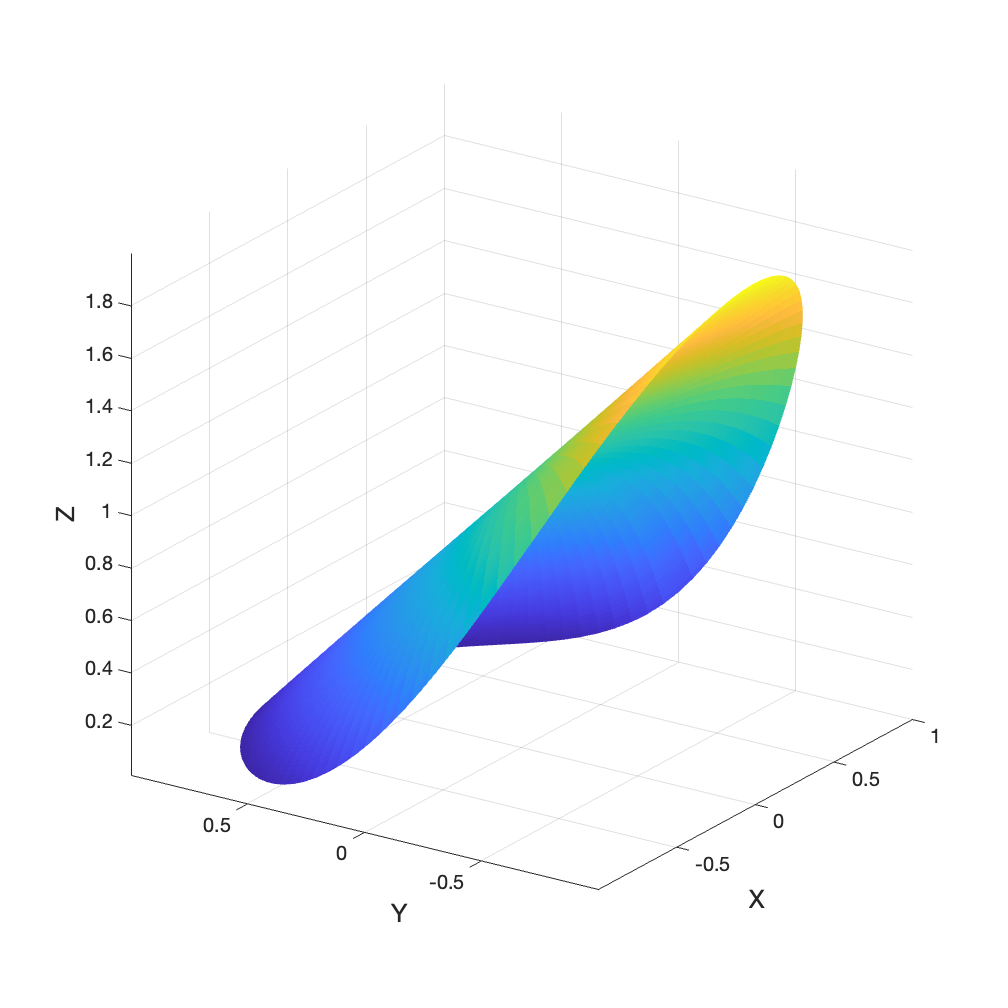}
        \caption{Conventional PINN}
    \end{subfigure}
        \caption{Numerical prediction of \eqref{e:circle_eq} with $\ep = 10^{-6}$. For our simulations, we select a uniform grid of $50$ discretized points in each of the $\eta$ and $\tau$ directions.}\label{fig non-li}
\end{figure}

\section{Elliptical Domain}
In this section, we investigate the boundary layer behavior of the singularly perturbed convection-diffusion problem (\ref{eq:main}) in an elliptical domain. As \(\epsilon\) becomes small, the solution to (\ref{eq:main}) exhibits rapid transitions near the lower semi-ellipse. Additionally, unlike the channel domain, the elliptical domain introduces singularities at the junctions of the upper and lower semi-ellipses. To address these issues, we perform a \textit{new and complete} boundary layer analysis and validate our approach through numerical simulations.

\subsection{Boundary Layer Analysis}
We consider the convection-diffusion equations on an elliptical domain, which represents a natural yet non-trivial extension of the problem:
\begin{equation} \label{e:ellipse}
\begin{split}
   L_{\epsilon}u^{\ep} :=- \epsilon \Delta u^{\ep} - u^{\ep}_y & =f, \quad \text{  in  } \Omega=\left\{(x,y) | {\frac{x^{2}}{A^{2}}}+\frac{y^{2}}{B^{2}}< 1\right \}\\
     u^{\ep}  & = 0, \quad \text{  at  } \partial \Omega.
\end{split}    
\end{equation}
Within the elliptical domain, we utilize the elliptic coordinate system \cite{churchill1961mathematical} and examine two separate scenarios; \lowercase\expandafter{\romannumeral1}) $A > B$, and \lowercase\expandafter{\romannumeral2}) $A < B$.
We start by introducing the expansion \( u = u^0 + \varphi^0 \), where \( u^0 \) represents the outer expansion and \(\varphi^0\) is the inner expansion that captures the behavior within the boundary layer. To derive the corrector function from (\ref{e:ellipse}), we begin by examining the limit problem, setting \(\epsilon\) to zero in the governing equation, that is,
\begin{equation} \label{e:ellipse_limit}
\begin{split}
L_{0}u^{0} :=-u^{0}_y & = f \quad \text{  in  } \Omega,\\
u^{0} & =0 \text{ at } \Gamma_{+}=\left\{(x,y) | {\frac{x^{2}}{A^{2}}}+\frac{y^{2}}{B^{2}}= 1, y > 0\right \}.
\end{split}
\end{equation}
Note that the boundary condition in \eqref{e:ellipse_limit} is imposed only on the upper semi-ellipse, where the fluid flows in. Consequently, we expect the boundary layer to form near the lower semi-ellipse, where the fluid flows out. The singularly perturbed problem described by equation \eqref{e:ellipse} requires careful treatment in boundary layer analysis due to the presence of a degenerate boundary layer near the characteristic points at \((\pm A, 0)\).
To begin with the case \(A > B\), we define the elliptic coordinates \((\eta, \tau)\) as follows:
\begin{equation}
\begin{split}
    x = a \cosh (R - \eta) \cos \tau, \quad y = a \sinh (R - \eta) \sin \tau,
\end{split}
\end{equation}
where \(0 \leq \eta < R\), \(\tau \in [0, 2 \pi]\), and \(a\) is a constant.
By setting $u^{\ep}(x,y)= v^{\ep}(\eta,\tau)$ with $a\cosh R =A$ and $a\sinh R =B$, we transform \eqref{e:ellipse} into the following form
\begin{equation}
\begin{split}
     P_{\epsilon} v^{\ep} {:=}-\epsilon(v^{\ep}_{\eta \eta}+v^{\ep}_{\tau \tau})+(a \cosh (R-\eta) \sin \tau)  v^{\ep}_{\eta}-(a\sinh (R-\eta) \cos \tau)v^{\ep}_{\tau} & = H f \quad \text{ in } D=[0,R)\times [0,2\pi]\\
      v^{\ep}(0,\tau)&=0 \text{ at } 0 \leq \tau \leq 2\pi,
\end{split}
\end{equation}
where $H=(a\sinh (R-\eta) \cos \tau )^{2}+(a \cosh (R-\eta) \sin \tau)^{2}$.
To derive the corrector equation for \(\varphi\), we use the following stretched variable such that $\overline{\eta}=\frac{\eta}{\epsilon^{\alpha}}$.
We then obtain the corrector equation with \(\alpha = 1\), which represents the thickness of the boundary layer such as

\begin{equation} \label{e:cor_elliptic_01}
\begin{split}
     -\varphi^0_{\overline{\eta \eta}}+ \left (a \left(\frac{e^{R-\epsilon\overline{\eta}}+e^{\epsilon\overline{\eta}-R}}{2} \right)\sin \tau \right) \varphi^0_{\overline{\eta}}&=0, \text{  for  } 0 < \eta < R,\quad \pi < \tau < 2 \pi,\\
     \varphi^0&=-u^{0}(B_{x}(\tau), B_{y}(\tau)) \text{ at } \overline{\eta}=0,\\
     \varphi^0 &\rightarrow 0, \quad \text{as } \overline{\eta} \rightarrow \infty,
\end{split}
\end{equation}
where
\begin{align}\notag
    B_{x}(\tau)=a\cosh R \cos \tau,\\
    B_{y}(\tau)=a\sinh R \sin \tau.
\end{align}
Rather than seeking an explicit solution for the corrector equation, we use the Taylor series expansion for computational convenience to obtain a solution profile for \eqref{e:cor_elliptic_01},
\begin{equation} \label{e:cor_elliptic_02}
\begin{split}
     -\varphi^0_{\overline{\eta \eta}}+&\left(a
     \left(\frac{e^{R}+e^{-R}}{2}+\epsilon
     \left(\frac{e^{-R}-e^{R}}{2}\right)
     \overline{\eta}+\frac{\epsilon^{2}}{2!}
     \left(\frac{e^{R}+e^{-R}}{2}\right)
     \overline{\eta}^{2}+..          \right)
     \sin \tau\right) \varphi^0_{\overline{\eta}}=0.\\
\end{split}
\end{equation}
By identifying the dominant terms in equation \eqref{e:cor_elliptic_02}, we derive the approximate corrector equations,
{
\begin{equation} 
\begin{split}
     -\varphi^0_{\overline{\eta \eta}}+ \left (a \left(\frac{e^{R}+e^{-R}}{2} \right)\sin \tau \right) \varphi^0_{\overline{\eta}}&=0, \text{  for  } 0 < \eta < R,\quad \pi < \tau < 2 \pi,\\
     \varphi^0&=-u^{0}( B_{x}(\tau), B_{y}(\tau)) \text{ at } \overline{\eta}=0,\\
     \varphi^0 &\rightarrow 0, \quad \text{as } \overline{\eta} \rightarrow \infty.
\end{split}
\end{equation}
}
An explicit solution can be calculated as
\be \label{e:elliptic_cor}
     \bar{\varphi}^0=-u^{0}( B_{x}(\tau), B_{y}(\tau)) \exp\left({\frac{A\sin\tau}{\varepsilon   }\eta}\right)\chi_{[\pi,2\pi]},
\ee
where $\chi$ stands for the characteristic function.
We introduce an approximate form to satisfy the boundary condition,
\be
\bar \varphi^0 =-u^{0}( B_{x}(\tau), B_{y}(\tau))\exp\left({\frac{A\sin\tau}{\varepsilon   }\eta}\right)\chi_{[\pi,2\pi]}\delta(\eta),
\ee
where $\delta(\eta)$ is a smooth cut-off function such that $\delta(\eta) = 1$ for $\eta \in [0, R/2]$ and $=0$ for $\eta \in [3R/4, 1]$.

A convergence analysis for the boundary layer problem is essential to ensure that the proposed corrector accurately represents boundary layer behavior, thereby justifying the construction of the SL-PINN scheme. 
We introduce the compatibility conditions and the convergence theory supporting our SL-PINN scheme. Detailed proofs for the following lemmas and theorems are provided in the appendix. In this article, we assume that the following compatibility conditions hold:
\begin{equation}\label{compat:ellpise}
\frac{\partial^{p_{1}+p_{2}} f}{\partial x^{p_{1}} \partial y^{p_{2}}} = 0 \quad \text{at} \quad (\pm a\cosh R, 0) \quad \text{for} \quad 0 \leq 2p_{1} + p_{2} \leq 2 , \quad p_{1}, p_{2} \geq 0,
\end{equation}
\begin{lemma}
There exists a positive constant \( k \) such that, for integers \( l, n, s \geq 0, m = 0, 1, 2 \), and for \( 1 \leq p \leq \infty \),
\begin{equation}
\left| (\sin \tau)^{-l} \left( \frac{\eta}{\varepsilon} \right)^n  \frac{\partial^{s+m} \varphi^{0}}{\partial \eta^s \partial \tau^m} \right|_{L_p(D)} \leq \kappa \sup_{\tau} |a_{0,h}(\tau)| \varepsilon^{\frac{1}{p}-s},
\end{equation}
where \( h = -s + m + l + n + 1 \).
The notation \( a_{0,q} \) follows the same convention as outlined in the \textit{Notation convection} of \cite{JUNG201188}.
\end{lemma}
\begin{lemma}
If \( q \leq 1 \) or if the compatibility conditions (\ref{e:append_comp_cond}) hold for \( 2 + 3n \geq -2 + q \geq 0 \), then \( a_{0,q}(\tau) \) is bounded for \( \tau \in [\pi, 2\pi] \).
\end{lemma}
\begin{theorem}\label{conver_thm_ellipse}
With the compatibility condition (\ref{compat:ellpise}), The following estimate holds:
\begin{equation}
|u^\varepsilon - u^0 - \bar{\varphi}^0|_{L^2(\Omega)} + \sqrt{\varepsilon} |u^\varepsilon - u^0 - \bar{\varphi}^0|_{H^1(\Omega)} \leq \kappa \sqrt{\varepsilon}.
\end{equation}
where \(u^\varepsilon\), \(u^0\) is the solution of (\ref{e:ellipse}), (\ref{e:ellipse_limit}), and \(\bar{\varphi}^0\) is the corrector in (\ref{e:cor_elliptic_01}).
\end{theorem}
Under the compatibility conditions, Theorem (\ref{conver_thm_ellipse}) demonstrates that the solution of \eqref{e:ellipse} converges to the limit solution \eqref{e:ellipse_limit} in the \(L^2\) norm as \(\epsilon \rightarrow 0\). Building on these results, we ensure that the corrector function effectively represents the boundary layer profile. Consequently, this corrector function can be seamlessly integrated into our neural network scheme, specifically the SL-PINN.

\subsection{Numerical Simulations for Major Axis Parallel to \texorpdfstring{$x$-Axis}{x-Axis}}
We now introduce our SL-PINN of the form,
\begin{equation}\label{e:EPINN_scheme_ellipse01}
    \widetilde{v}(\eta,\tau ; \, {\blds \theta}) 
            =
            \left(\hat{v}(\eta,\tau,{\blds \theta})-\hat{v}(0,\tau,{\blds \theta})
            {\hat{ \varphi}^0(\eta,\tau)}\right)C(\eta,\tau),
\end{equation}
where $\hat{ \varphi}^0$ is given by 
\begin{equation}
     \hat{ \varphi}^0=\exp\left(\frac{A\sin \tau}{\ep} \eta \right) \chi_{[\pi, 2\pi]}(\tau) \delta(\eta),
\end{equation}
and $C(\eta,\tau)$ is the regularizing term given by
\begin{equation} \label{e:ellipse_reg01}
\begin{split}
    C(\eta,\tau)=\begin{cases}
    1-\left(\frac{R-\eta}{R}\right)^{3} , \text{ if } 0 < \tau < \pi\\
    1-\left(\frac{R-\eta}{R}\right)^{3}-\left(\frac{R-\eta}{R}\sin\tau \right)^3, \text{ if } \pi \leq \tau \leq 2\pi.
    \end{cases}
\end{split}
\end{equation} 
Then, the residual loss is defined by
\begin{equation} \label{e:ellipse_loss_first}
\begin{split}
    Loss=\left(\frac{1}{N}\sum_{i=0}^{N}
    |P_{\epsilon}\widetilde{v}((\eta_{i},\tau_{i},{\blds \theta}))-f |^{p} \right)^{1/p} \quad  \text{ for }(\eta_{i},\tau_{i}) \in D,
\end{split}    
\end{equation} 
where $p=1,2$. 
To account for the boundary layer behavior occurring near the boundary where \(\pi \leq \tau \leq 2 \pi\), we divide the residual loss \eqref{e:ellipse_loss_first} into two sections: \(0 < \tau < \pi\) and \(\pi \leq \tau \leq 2 \pi\). Calculating the residual loss is straightforward for the range \(0 < \tau < \pi\)
\begin{equation} \label{x_upper_elipse_cal}
\begin{split}
    &P_{\epsilon}\widetilde{v}(\eta,\tau,{\blds \theta}) - f \\
    & = -\ep \left( \hat{v}_{\eta \eta}(\eta,\tau,{\blds \theta}) \left(1 - \left(\frac{R-\eta}{R}\right)^3\right) - 2\hat{v}_{\eta}(\eta,\tau,{\blds \theta}) \left(-3 \frac{(R-\eta)^2}{R^3}\right) + \hat{v}(\eta,\tau,{\blds \theta}) \left(-6 \frac{R-\eta}{R^3}\right) \right. \\
    & \quad + \left. \hat{v}_{\tau \tau}(\eta,\tau,{\blds \theta}) \left(1 - \left(\frac{R-\eta}{R}\right)^3\right) \right) \\
    & \quad - (a \cosh (R-\eta) \sin \tau) \left( \hat{v}_{\eta}(\eta,\tau,{\blds \theta}) \left(1 - \left(\frac{R-\eta}{R}\right)^3\right) + \hat{v}(\eta,\tau,{\blds \theta}) \left(-3 \frac{(R-\eta)^2}{R^3}\right) \right) \\
    & \quad - (a \sinh (R-\eta) \cos \tau) \hat{v}_{\tau}(\eta,\tau,{\blds \theta}) \left(1 - \left(\frac{R-\eta}{R}\right)^3\right) - f,
\end{split}
\end{equation}
for $0 < \eta \leq R \;, 0 < \tau < \pi$.
When considering \(\pi \leq \tau \leq 2 \pi\), incorporating the boundary layer element in \eqref{e:EPINN_scheme_ellipse01} and including the regularizing term in \eqref{e:ellipse_reg01} result in a more intricate form of the residual loss, as shown below
\begin{equation} \label{x_lower_elipse_cal}
\begin{split}
    &P_{\epsilon}\widetilde{v}(\eta,\tau,{\blds \theta}) - f = \\
    &-\ep \left( \hat{v}_{\eta \eta}(\eta,\tau,{\blds \theta})C + 2\hat{v}_{\eta}(\eta,\tau,{\blds \theta})C_{\eta} + \hat{v}(\eta,\tau,{\blds \theta})C_{\eta \eta} + \hat{v}_{\tau \tau}(\eta,\tau,{\blds \theta})C + 2\hat{v}_{\tau}(\eta,\tau,{\blds \theta})C_{\tau} + \hat{v}(\eta,\tau,{\blds \theta})C_{\tau \tau} \right) \\
    &+(a \cosh (R-\eta) \sin \tau) \left( \hat{v}_{\eta}(\eta,\tau,{\blds \theta})C + \hat{v}(\eta,\tau,{\blds \theta})C_{\eta} \right) \\
    &-(a \sinh (R-\eta) \cos \tau) \left( \hat{v}_{\tau}(\eta,\tau,{\blds \theta})C + \hat{v}(\eta,\tau,{\blds \theta})C_{\tau} \right) \\
    &+\ep \hat{v}(0,\tau,{\blds \theta}) \left[ \delta_{\eta \eta}C + 2 \delta_{\eta}C_{\eta} + \delta C_{\eta \eta} \right] \exp\left( \frac{A \sin \tau}{\ep} \eta \right) \\
    &+\hat{v}(0,\tau,{\blds \theta}) \left[ 2A \sin \tau \delta_{\eta}C + 2A \sin \tau \delta C_{\eta} \right] \exp\left( \frac{A \sin \tau}{\ep} \eta \right) \\
    &+a (\cosh (R-\eta) \sin \tau) \hat{v}(0,\tau,{\blds \theta}) \left[ \delta_{\eta} C + \delta C_{\eta} \right] \exp\left( \frac{A \sin \tau}{\ep} \eta \right) \\
    &+\ep \delta \left[ \hat{v}_{\tau \tau}(\eta,\tau,{\blds \theta})C + 2\hat{v}_{\tau}(\eta,\tau,{\blds \theta})C_{\tau} + \hat{v}(\eta,\tau,{\blds \theta})C_{\tau \tau} \right] \exp\left( \frac{A \sin \tau}{\ep} \eta \right) \\
    &-\delta \left[ A \eta \left( (\sin \tau \hat{v}(\eta,\tau,{\blds \theta}) - 2\cos \tau \hat{v}_{\tau}(\eta,\tau,{\blds \theta}))C - 2(\cos \tau \hat{v}(\eta,\tau,{\blds \theta}))C_{\tau} \right) \right] \exp\left( \frac{A \sin \tau}{\ep} \eta \right) \\
    &+\frac{\hat{v}(0,\tau,{\blds \theta})}{\ep} \delta C \left[ (A \sin \tau)^{2} - a \cosh (R-\eta) \sin^2 \tau A + (A\eta \cos \tau)^{2} + a \sinh (R-\eta) \cos^2 \tau A \eta \right]\\ 
    &\times \exp\left( \frac{A \sin \tau}{\ep} \eta \right) - f,
\end{split}
\end{equation}
for $0 < \eta \leq R$ and $\pi \leq \tau \leq 2\pi$.
Similar to the circular domain case, calculating the residual loss in \eqref{x_lower_elipse_cal} involves a large term, such as \(\epsilon^{-1}\), which can negatively impact the loss optimization process.

To make the computation in equation \eqref{x_lower_elipse_cal} feasible, we extract the dominant term in \(\epsilon\), specifically the term of \(\mathcal{O}(\epsilon^{-1})\). Hence, we set the dominant term $\psi(\eta,\tau,{\blds \theta})$ such that
\begin{equation}
\begin{split} 
    \psi(\eta,\tau,{\blds \theta}):=\frac{\hat{v}(0,\tau,{\blds \theta})}{\ep}\delta(\eta) (1-(\frac{R-\eta}{R})^{3}-(\frac{R-\eta}{R}\sin\tau)^3)[(A\sin\tau)^{2}-a\cosh (R-\eta)\sin^2\tau A\\
    +(A\eta\cos\tau)^{2}+a\sinh (R-\eta) \cos^2\tau A\eta]
    \exp\left( \frac{A\sin \tau}{\ep} \eta \right){\chi_{[\pi, 2\pi]}(\tau)}.
\end{split}
\end{equation}
When $p=1$, by the triangular inequality, the loss in \eqref{e:ellipse_loss_first} bounds
\begin{equation}\label{e:bound_ellipse00}
 Loss  \leq \frac{1}{N}\sum_{i=0}^{N}|\psi (\eta_{i},\tau_{i},{\blds \theta})|
 +  \frac{1}{N}\sum_{i=0}^{N} |P_{\epsilon}\widetilde{v}((\eta_{i},\tau_{i},{\blds \theta}))-\psi (\eta_{i},\tau_{i},{\blds \theta})-f|.
\end{equation}
The rightmost term does not involve a high-order term, such as \(\epsilon^{\alpha}\) where \(\alpha < 0\), because \(\psi\) already contains large terms like \(\epsilon^{-1}\). Therefore, this part can be computed using a conventional \(L^1\) or \(L^2\) loss. Assume that \(\boldsymbol{\theta}\) is fixed and that we select a sufficiently large number of sampling points \(N\) such as 
\begin{equation}\label{e:normal_bound_ellipse_x}
 \frac{1}{N}\sum_{i=0}^{N} \left|P_{\epsilon}\widetilde{v}(\eta_{i}, \tau_{i}, \boldsymbol{\theta}) - \psi(\eta_{i}, \tau_{i}, \boldsymbol{\theta}) - f\right| \approx \frac{1}{\pi}\int_{0}^{2\pi} \int_{0}^{1} \left|P_{\epsilon}\widetilde{v}(\eta, \tau, \boldsymbol{\theta}) - \psi(\eta, \tau, \boldsymbol{\theta}) - f\right| \, d\eta \, d\tau < \infty.
\end{equation}
Considering that \(\psi\) contains terms of \(\mathcal{O}(\epsilon^{-1})\), we utilize an \(L^1\) loss for handling \(\psi\).
For precise computations, we choose a sufficiently large number of sampling points, \(N\), such as
\begin{equation}\label{e:bound_ellipse01}
\begin{split}
     & \frac{1}{N}\sum_{i=0}^{N} \left|\psi(\eta_{i},\tau_{i},{\blds \theta}) \right| \approx \frac{1}{R\pi}\int^{2\pi}_{\pi} \int^{R}_{0} |\psi(\eta,\tau,{\blds \theta})|d \eta d \tau.
\end{split}
\end{equation}
The right-hand side of equation (\ref{e:bound_ellipse01}) is bounded by \(C\), as shown in Section \ref{subsec:plain}. Specifically, we obtain that
\begin{equation}\label{e:bound_ellipse02}
\begin{split}
    \frac{1}{R\pi}\int^{2\pi}_{\pi}\int^{R}_{0} |P_{\epsilon}\psi(\eta,\tau,{\blds \theta})|d \eta d \tau
     \leq C.
\end{split}
\end{equation}
Therefore, the loss \eqref{e:ellipse_loss_first} becomes nearly constant due to \eqref{e:normal_bound_ellipse_x} and \eqref{e:bound_ellipse02} when \(\epsilon\) is sufficiently small. This simplification facilitates feasible computations in the SL-PINN approach.

\begin{figure}[htbp]
     \centering
     \begin{subfigure}[b]{0.4\textwidth}
         \centering
         \includegraphics[width=\textwidth]{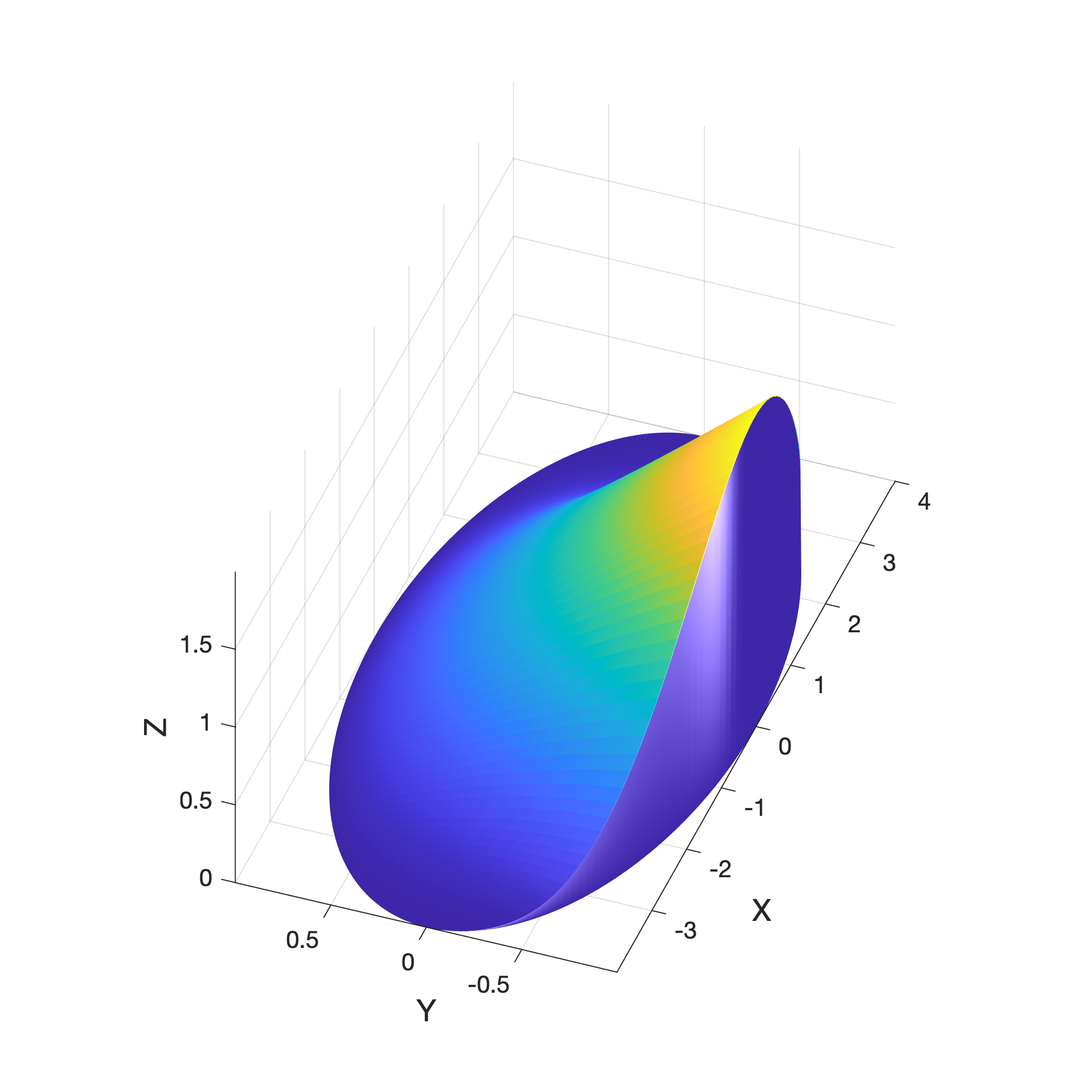}
         \caption{True solution}
     \end{subfigure}
     \begin{subfigure}[b]{0.4\textwidth}
         \centering
         \includegraphics[width=\textwidth]{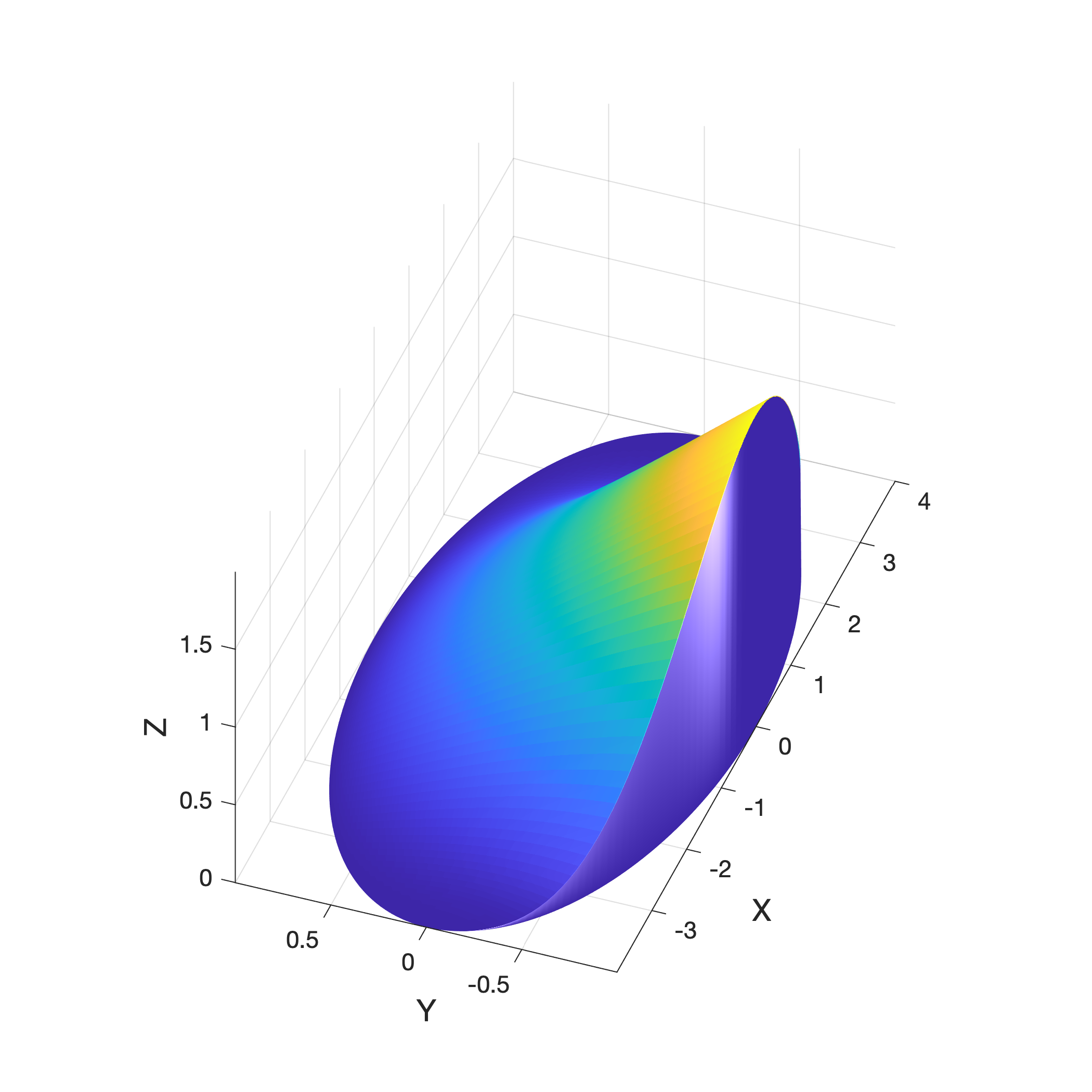}
         \caption{SL-PINN using $L^2$ training}
     \end{subfigure}
     \begin{subfigure}[b]{0.4\textwidth}
         \centering
         \includegraphics[width=\textwidth]{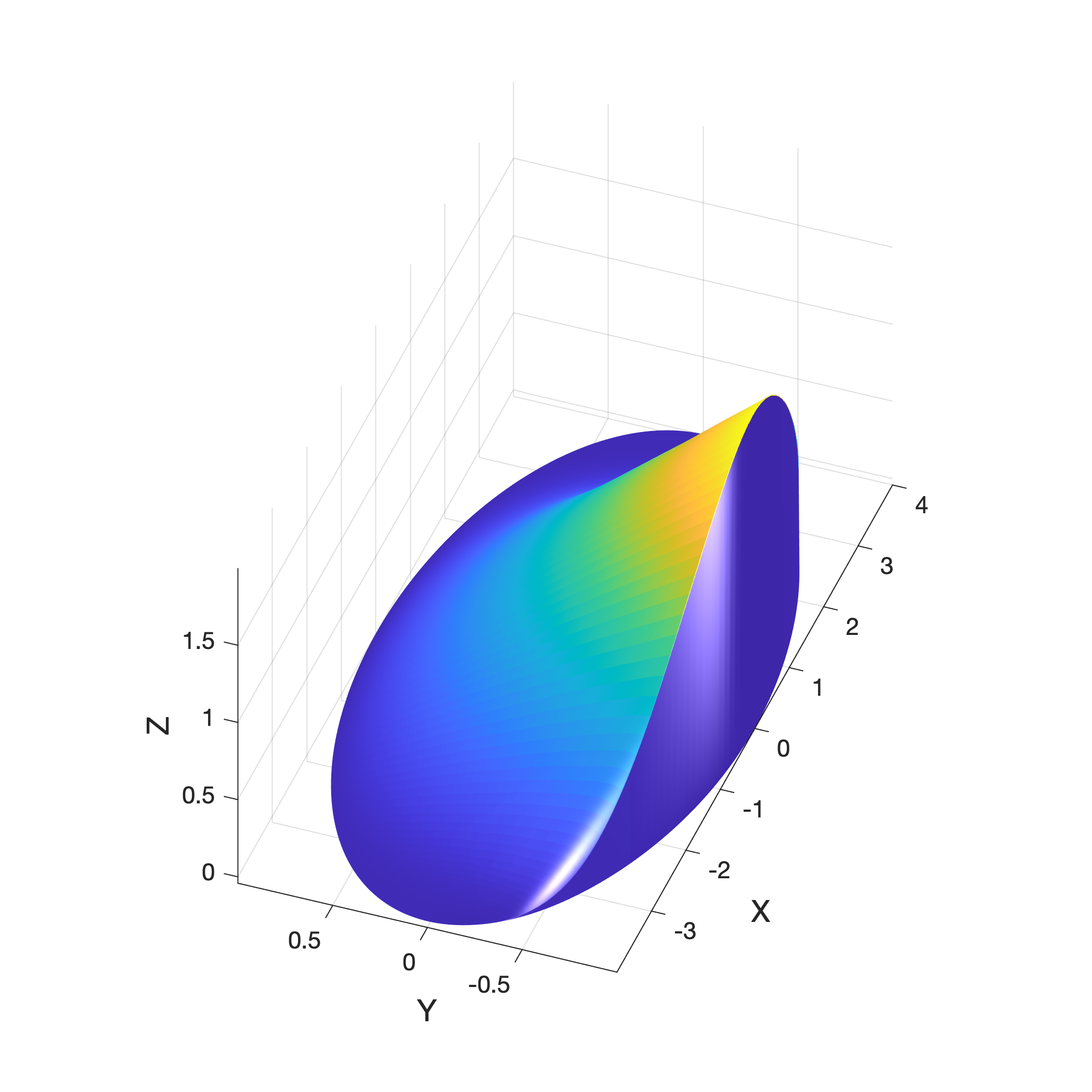}
         \caption{SL-PINN using $L^1$ training}
     \end{subfigure}
     \begin{subfigure}[b]{0.4\textwidth}
         \centering
         \includegraphics[width=\textwidth]{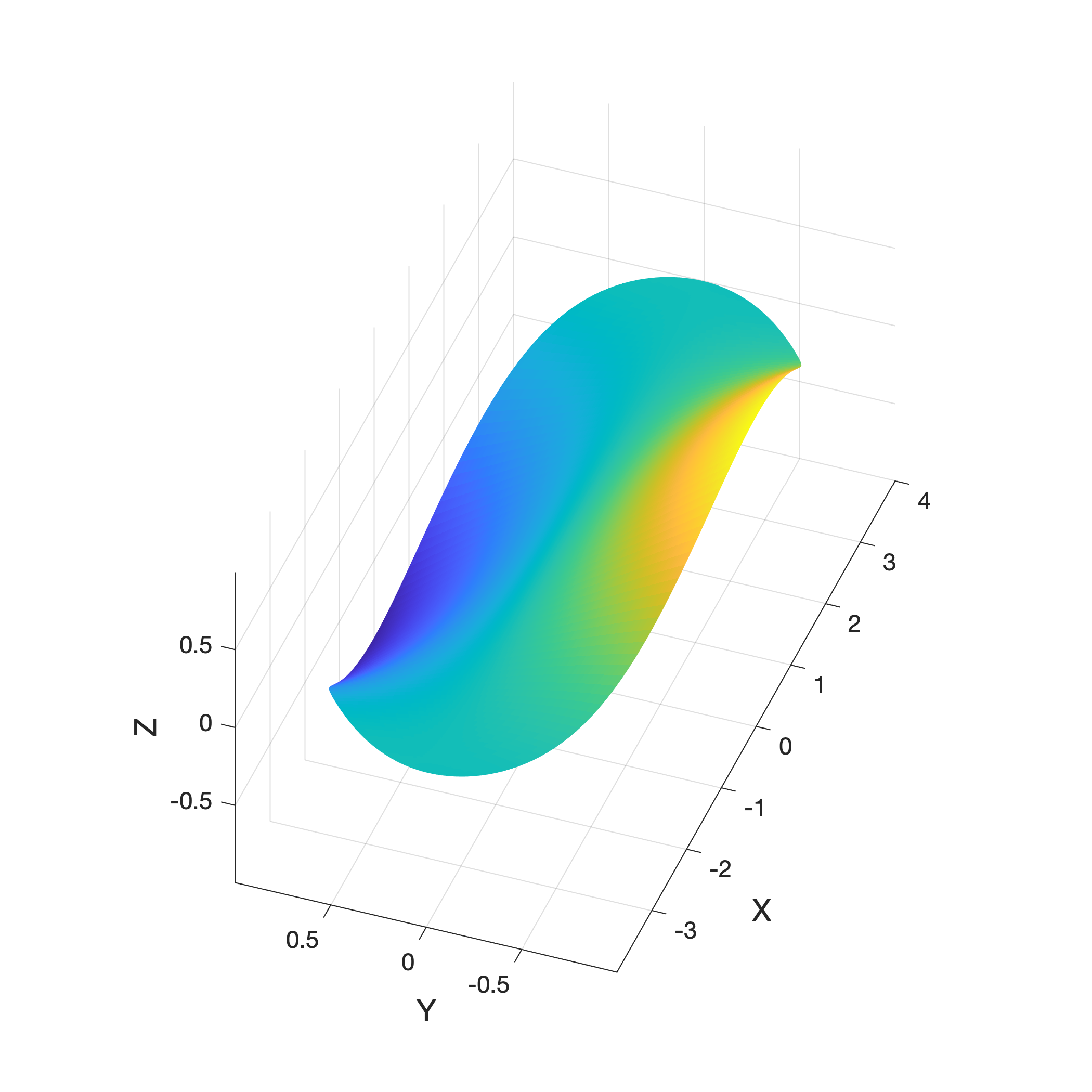}
         \caption{Conventional PINN}
     \end{subfigure}
     \caption{Numerical prediction of \eqref{e:ellipse} with $(A,B)=(4,1)$ and $\ep = 10^{-6}$ when $f={(1-({x}/{A})^2)^2}/{B}$. For our simulations, we select a uniform grid of $50$ discretized points in each of the $r$ and $\tau$ directions.} \label{fig6}
\end{figure}

\begin{figure}[htbp]
    \centering
    \begin{subfigure}[b]{0.45\textwidth}
        \centering
        \includegraphics[width=\textwidth]{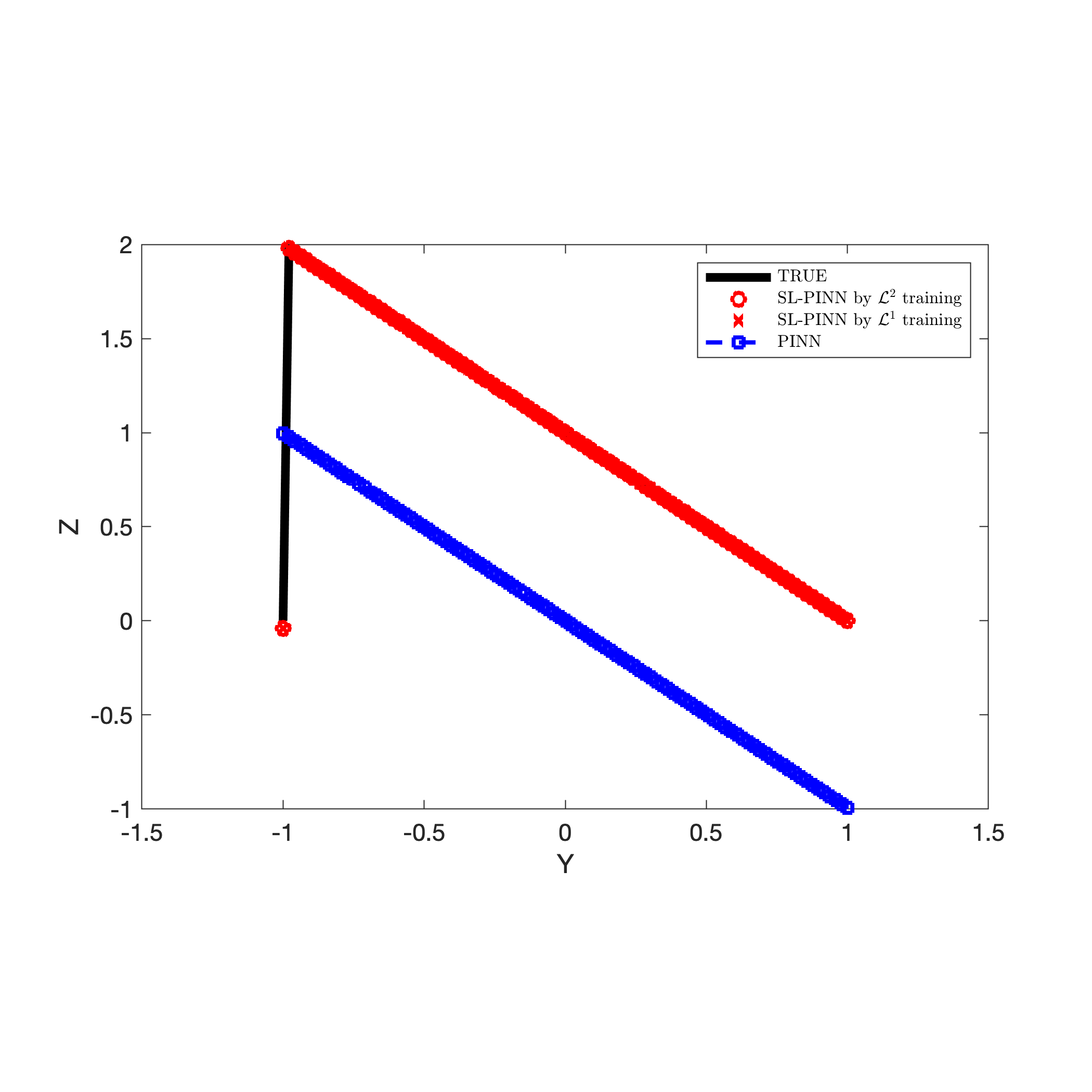}
        \caption{True soltuion at t=1}
    \end{subfigure}
    \begin{subfigure}[b]{0.45\textwidth}
        \centering
        \includegraphics[width=\textwidth]{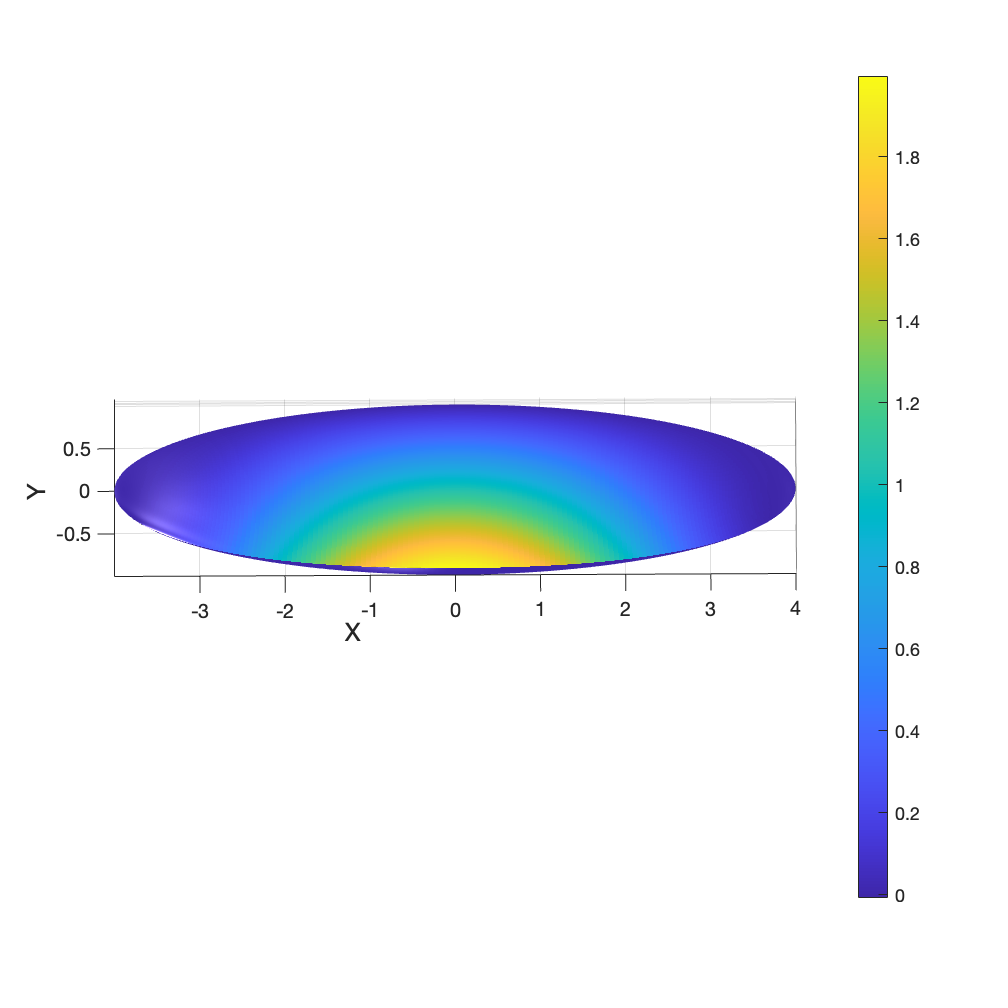}
        \caption{SL-PINN}
    \end{subfigure}
    \caption{The one-dimensional profile of predicted solutions in Figure \ref{fig6} along the line $\tau = \pi/2$ and solution surface projection profile.}\label{fig6-6}
\end{figure}

The effectiveness of our new approach is demonstrated through a series of numerical experiments presented in Figures \ref{fig6} and \ref{fig6-6}. In these experiments, we utilize a grid discretization consisting of 50 uniformly spaced points in both the \(\tau\) and \(r\) directions. 
Similar to the circular domain, when \(f\) does not vanish at the characteristic points \((\pm A, 0)\), the function \(f(x, y)\) becomes incompatible due to the appearance of singularities in its derivatives at these points. To maintain simplicity in this paper, we focus on the first compatibility condition.
We compare the performance of the standard five-layer PINN approximation with our SL-PINN approximation. Figure \ref{fig6} showcases the numerical solutions of \eqref{e:ellipse} with \((A,B) = (4,1)\) and \(\epsilon = 10^{-6}\) when \(f = \frac{(1-(x/A)^2)^2}{B}\). Figure \ref{fig6} clearly illustrates the conventional PINN method's failure to accurately approximate the solution to the singular perturbation problem.
In contrast, our new scheme exhibits significant superiority over the conventional PINN method. The numerical results presented in Figure \ref{fig6} and Table \ref{tab:1} provide compelling evidence of the substantial performance improvement achieved by our SL-PINN approach. Our semi-analytic SL-PINN, enriched with the corrector, consistently produces stable and accurate approximate solutions, regardless of the small parameter \(\epsilon\).
Figure \ref{fig6-6} offers a closer examination of the one-dimensional profile of predicted solutions at \(\tau = \pi/2\), facilitating a clear and direct comparison. Figure \ref{fig6-6} demonstrates the high accuracy of both the \(L^1\) and \(L^2\) training approaches in generating approximations. However, the conventional PINN method falls short in capturing the sharp transition near the boundary layer.

\subsection{Numerical Simulations for Major Axis Parallel to \texorpdfstring{$y$-Axis}{y-Axis}}
To begin with the case \(A < B\), we define the elliptic coordinates as $(\eta, \tau)$ as follows:
\begin{equation}
\begin{split}
    x=a\sinh (R-\eta) \cos \tau ,y=a\cosh (R-\eta) \sin \tau,
\end{split}
\end{equation}
where  \(0 \leq \eta < R\) and $\tau \in [0, 2 \pi]$.
By setting $u^{\ep}(x,y)= v^{\ep}(\eta,\tau)$ with  $a\sinh R =A$ and $a\cosh R =B$, we transform \eqref{e:ellipse} into the following form
\begin{equation}
\begin{split}
     P_{\epsilon} v^{\ep} :=-\epsilon(v^{\ep}_{\eta \eta}+v^{\ep}_{\tau \tau})+(a \sinh (R-\eta) \sin \tau)  v^{\ep}_{\eta}-(a\cosh (R-\eta) \cos \tau)v^{\ep}_{\tau} & = H f \quad \text{ in } D=[0,R)\times [0,2\pi]\\
      v^{\ep}(0,\tau)&=0 \text{ at } 0 \leq \tau \leq 2\pi,
\end{split}
\end{equation}
where $H=(a\sinh (R-\eta) \sin \tau )^{2}+(a \cosh (R-\eta) \cos \tau)^{2}$.
To derive the corrector equation of $\varphi$, we use the following stretched variable,
\begin{equation}
\begin{split}
    \overline{\eta}=&\frac{\eta}{\epsilon^{\alpha}}.
\end{split}
\end{equation}
We then obtain the corrector equation with $\alpha = 1$, which represents the thickness of the boundary layer, 
\begin{equation}\label{e:cor_elliptic}
\begin{split}
     -\varphi^0_{\overline{\eta \eta}}+\left (a \left(\frac{e^{R-\epsilon\overline{\eta}}-e^{\epsilon\overline{\eta}-R}}{2} \right)\sin \tau \right) \varphi^0_{\overline{\eta}}&=0, \text{  for  } 0 < \eta < R,\quad \pi < \tau < 2 \pi,\\
     \varphi^0&=-u^{0}(B_{x}(\tau),B_{y}(\tau)) \text{ at } \overline{\eta}=0,\\
     \varphi^0 &\rightarrow 0, \quad \text{as } \overline{\eta} \rightarrow \infty,
\end{split}
\end{equation}
where
\begin{align}\notag
    B_{x}(\tau)=a\sinh R \cos \tau,\\
    B_{y}(\tau)=a\cosh R \sin \tau.
\end{align}
Instead of seeking an explicit solution for the corrector equation, we use a Taylor series expansion for computational convenience to obtain a solution profile for 
\eqref{e:cor_elliptic},
\begin{equation}\label{e:cor_elliptic_02_y}
\begin{split}
     -\varphi^0_{\overline{\eta \eta}}+&\left(a
     \left(\frac{e^{R}-e^{-R}}{2}-\epsilon
     \left(\frac{e^{-R}+e^{R}}{2}\right)
     \overline{\eta}+\frac{\epsilon^{2}}{2!}
     \left(\frac{e^{R}-e^{-R}}{2}\right)
     \overline{\eta}^{2}-..          \right)
     \sin \tau\right) \varphi^0_{\overline{\eta}}=0.\\
\end{split}
\end{equation}
By identifying the dominant terms in equation \eqref{e:cor_elliptic_02_y}, we derive the approximate corrector equations,
\begin{equation}\label{e:ellipse_y_approximation}
\begin{split}
     -\varphi^0_{\overline{\eta \eta}}+\left (a \left(\frac{e^{R}-e^{-R}}{2} \right)\sin \tau \right) \varphi^0_{\overline{\eta}}&=0, \text{  for  } 0 < \eta < R,\quad \pi < \tau < 2 \pi,\\
     \varphi^0&=-u^{0}(B_{x}(\tau),B_{y}(\tau))\text{ at } \overline{\eta}=0,\\
     \varphi^0 &\rightarrow 0, \quad \text{as } \overline{\eta} \rightarrow \infty.
\end{split}
\end{equation}
The explicit solution of \eqref{e:ellipse_y_approximation} is found as
\be
\bar \varphi^0 =-u^{0}( B_{x}(\tau), B_{y}(\tau))\exp\left({\frac{A\sin\tau}{\varepsilon   }\eta}\right)\chi_{[\pi,2\pi]}\delta(\eta).
\ee
where $\chi$ stands for the characteristic function.
To match the boundary condition in our numerical scheme, we introduce a cut-off function to derive an approximate form of
\be
\bar \varphi^0 =-u^{0}( B_{x}(\tau), B_{y}(\tau))\exp\left(\frac{\sin \tau}{\ep} \eta \right) \chi_{[\pi, 2\pi]}(\tau) \delta(\eta),
\ee
where $\delta(\eta)$ is a smooth cut-off function such that $\delta(\eta) = 1$ for $\eta \in [0, R/2]$ and $=0$ for $\eta \in [3R/4, 1]$.

We now introduce our {\it SL-PINN} of the form,
\begin{equation}\label{e:EPINN_scheme_ellipse02}
    \widetilde{v}(\eta,\tau ; \, {\blds \theta}) 
            =
            \left(\hat{v}(\eta,\tau,{\blds \theta})-\hat{v}(0,\tau,{\blds \theta})
            {\hat {\varphi}^0}\right)C(\eta,\tau),
\end{equation}
where $\hat{ \varphi}^0$ is given by 
\begin{equation}
     \hat{ \varphi}^0=\exp\left(\frac{A\sin \tau}{\ep} \eta \right) \chi_{[\pi, 2\pi]}(\tau) \delta(\eta),
\end{equation}
and $C(\eta,\tau)$ is the boundary regularizing term given by
\begin{equation} \label{e:ellipse_reg02}
\begin{split}
    C(\eta,\tau)=\begin{cases}
    1-(\frac{R-\eta}{R})^{3} , \text{ if } 0 \leq \tau \leq \pi,\\
    1-(\frac{R-\eta}{R})^{3}-(\frac{R-\eta}{R}\sin\tau)^3, \text{ if } \pi < \tau <2\pi.
    \end{cases}
\end{split}
\end{equation}
Then, the residual loss is defined by
\begin{equation}\label{e:ellipse_loss_second}
\begin{split}
    Loss=\left(\frac{1}{N}\sum_{i=0}^{N}|P_{\epsilon}\widetilde{v}((\eta_{i},\tau_{i},{\blds \theta}))-f |^{p}\right)^{1/p} \quad  \text{ for }(\eta_{i},\tau_{i}) \in D,
\end{split}    
\end{equation}
where $p=1,2$. 
Considering the boundary layer behavior near $\pi \leq \tau \leq 2 \pi$, we split the residual loss  \eqref{e:ellipse_loss_second} into two sections: $0 < \tau < \pi$ and $\pi \leq \tau \leq 2 \pi$. 
The residual loss calculation for $0 < \tau < \pi$ is relatively simple:

\begin{equation}\label{y-upper-elipse}
\begin{split}
    &P_{\epsilon}\widetilde{v}((R-\eta,\tau,{\blds \theta})) - f = \\
    &-\ep \left( \hat{v}_{\eta \eta}(\eta,\tau,{\blds \theta}) \left(1 - \left(\frac{R-\eta}{R}\right)^3\right) - 2\hat{v}_{\eta}(\eta,\tau,{\blds \theta}) \left(-3 \frac{(R-\eta)^2}{R^3}\right) + \hat{v}(\eta,\tau,{\blds \theta}) \left(-6 \frac{R-\eta}{R^3}\right) \right. \\
    & \quad + \left. \hat{v}_{\tau \tau}(\eta,\tau,{\blds \theta}) \left(1 - \left(\frac{R-\eta}{R}\right)^3\right) \right) \\
    &+(a \sinh (R-\eta) \sin \tau) \left( \hat{v}_{\eta}(\eta,\tau,{\blds \theta}) \left(1 - \left(\frac{R-\eta}{R}\right)^3\right) + \hat{v}(\eta,\tau,{\blds \theta}) \left(3 \frac{(R-\eta)^2}{R^3}\right) \right) \\
    &-(a \cosh (R-\eta) \cos \tau) \hat{v}_{\tau}(\eta,\tau,{\blds \theta}) \left(1 - \left(\frac{R-\eta}{R}\right)^3\right) - f,
\end{split}
\end{equation}
for $0 \leq \eta < R \;, 0 \leq \tau \leq \pi$.

However, when considering $\pi < \tau < 2 \pi$, incorporating the boundary layer element in \eqref{e:EPINN_scheme_ellipse02} and including the regularizing term in \eqref{e:ellipse_reg02} result in a more intricate form of the residual loss, as shown below:

\begin{equation}\label{y-lower-elipse}
\begin{split}
    &P_{\epsilon}\widetilde{v}((R-\eta,\tau,{\blds \theta})) - f = \\
    &-\ep \left( \hat{v}_{\eta \eta}(\eta,\tau,{\blds \theta})C + 2\hat{v}_{\eta}(\eta,\tau,{\blds \theta})C_{\eta} + \hat{v}(\eta,\tau,{\blds \theta})C_{\eta \eta} + \hat{v}_{\tau \tau}(\eta,\tau,{\blds \theta})C + 2\hat{v}_{\tau}(\eta,\tau,{\blds \theta})C_{\tau} + \hat{v}(\eta,\tau,{\blds \theta})C_{\tau \tau} \right) \\
    &+(a \sinh (R-\eta) \sin \tau) \left( \hat{v}_{\eta}(\eta,\tau,{\blds \theta})C + \hat{v}(\eta,\tau,{\blds \theta})C_{\eta} \right) \\
    &-(a \cosh (R-\eta) \cos \tau) \left( \hat{v}_{\tau}(\eta,\tau,{\blds \theta})C + \hat{v}(\eta,\tau,{\blds \theta})C_{\tau} \right) \\
    &+\ep \hat{v}(0,\tau,{\blds \theta}) \left[ \delta_{\eta \eta}C + 2 \delta_{\eta}C_{\eta} + \delta C_{\eta \eta} \right] \exp\left( \frac{A\sin \tau}{\ep} \eta \right) \\
    &+\hat{v}(0,\tau,{\blds \theta}) \left[ 2A \sin \tau \delta_{\eta}C + 2A \sin \tau \delta C_{\eta} \right] \exp\left( \frac{A\sin \tau}{\ep} \eta \right) \\
    &-a (\cosh (R-\eta) \sin \tau) \hat{v}(0,\tau,{\blds \theta}) \left[ \delta_{\eta} C + \delta C_{\eta} \right] \exp\left( \frac{A \sin \tau}{\ep} \eta \right) \\
    &+\ep \delta \left[ \hat{v}_{\tau \tau}(\eta,\tau,{\blds \theta})C + 2\hat{v}_{\tau}(\eta,\tau,{\blds \theta})C_{\tau} + \hat{v}(\eta,\tau,{\blds \theta})C_{\tau \tau} \right] \exp\left( \frac{A \sin \tau}{\ep} \eta \right) \\
    &-\delta \left[ A \eta \left( (\sin \tau \hat{v}(\eta,\tau,{\blds \theta}) - 2\cos \tau \hat{v}_{\tau}(\eta,\tau,{\blds \theta}))C - 2(\cos \tau \hat{v}(\eta,\tau,{\blds \theta}))C_{\tau} \right) \right] \exp\left( \frac{A \sin \tau}{\ep} \eta \right) \\
    &+\frac{\hat{v}(0,\tau,{\blds \theta})}{\ep} \delta C \left[ (A \sin \tau)^{2} - a \sinh (R-\eta) \sin^2 \tau A + (A\eta \cos \tau)^{2} + a \cosh (R-\eta) \cos^2 \tau A\eta \right]\\
    &\times \exp\left( \frac{A \sin \tau}{\ep} \eta \right) - f,
\end{split}
\end{equation}
for $0 \leq \eta < R \;, \pi \leq \tau \leq
2\pi$.

Similar to the circular domain, the residual loss indicated in \eqref{x_lower_elipse_cal} encompasses a substantial term that may potentially misdirect our loss optimization process.
In order to make the computation in \eqref{x_lower_elipse_cal} feasible, we extract the dominant term in $\epsilon$, specifically the term of $\mathcal{O}(\epsilon^{-1})$.
Hence, we define
\begin{equation}
\begin{split} 
    \psi(\eta,\tau,{\blds \theta}):=\frac{\hat{v}(0,\tau,{\blds \theta})}{\ep}\delta(\eta) (1-(\frac{R-\eta}{R})^{3}-(\frac{R-\eta}{R}\sin\tau)^3)[(A\sin\tau)^{2}-a\sinh (R-\eta)\sin^2\tau A\\
    +(A\eta\cos\tau)^{2}+a\cosh (R-\eta) \cos^2\tau A\eta]
    \exp\left( \frac{A\sin \tau}{\ep} \eta \right)\chi_{[\pi, 2\pi]}(\tau).
\end{split}
\end{equation}
By the triangular inequality, the loss in \eqref{e:ellipse_loss_second} with $p=1$ bounds
\begin{equation}\label{e:bound_ellipse000}
 Loss  \leq \frac{1}{N}\sum_{i=0}^{N}|\psi (\eta_{i},\tau_{i},{\blds \theta})|
 +  \frac{1}{N}\sum_{i=0}^{N} |P_{\epsilon}\widetilde{v}((\eta_{i},\tau_{i},{\blds \theta}))-\psi (\eta_{i},\tau_{i},{\blds \theta})-f|.
\end{equation}
Since $\psi$ absorbs large terms such as $\epsilon^{-1}$, the rightmost term does not involve high-order terms such as ${\epsilon^{\alpha}}$ where $\alpha < 0$. 
Assume that \(\boldsymbol{\theta}\) is fixed and that we select a sufficiently large number of sampling points \(N\) such as
\begin{equation}\label{e:normal_bound_ellipse_y}
 \frac{1}{N}\sum_{i=0}^{N} \left|P_{\epsilon}\widetilde{v}(\eta_{i}, \tau_{i}, \boldsymbol{\theta}) - \psi(\eta_{i}, \tau_{i}, \boldsymbol{\theta}) - f\right| \approx \frac{1}{\pi}\int_{0}^{2\pi} \int_{0}^{1} \left|P_{\epsilon}\widetilde{v}(\eta, \tau, \boldsymbol{\theta}) - \psi(\eta, \tau, \boldsymbol{\theta}) - f\right| \, d\eta \, d\tau < \infty.
\end{equation}

Considering that \(\psi\) contains terms of \(\mathcal{O}(\epsilon^{-1})\), we utilize an \(L^1\) loss for handling \(\psi\).
For precise computations, we choose a sufficiently large number of sampling points, \(N\), such as
\begin{equation}\label{e:bound_ellipse_y_01}
\begin{split}
     & \frac{1}{N}\sum_{i=0}^{N}|\psi(\eta_{i},\tau_{i},{\blds \theta})| \approx \frac{1}{R\pi}\int^{2\pi}_{\pi} \int^{R}_{0} |\psi(\eta,\tau,{\blds \theta})|d \eta d \tau.
\end{split}
\end{equation}

The right-hand side of equation (\ref{e:bound_ellipse01}) is bounded by \(C\), as shown in Section \ref{subsec:plain}. Specifically, we obtain that
\begin{equation}\label{e:bound_ellipse020}
\begin{split}
     \frac{1}{R\pi}\int^{2\pi}_{\pi} \int^{R}_{0} |\psi(\eta,\tau,{\blds \theta})|d \eta d \tau
    \leq C.
\end{split}
\end{equation}

\begin{figure}[H]
     \centering
     \begin{subfigure}[b]{0.4\textwidth}
         \centering
         \includegraphics[width=\textwidth]{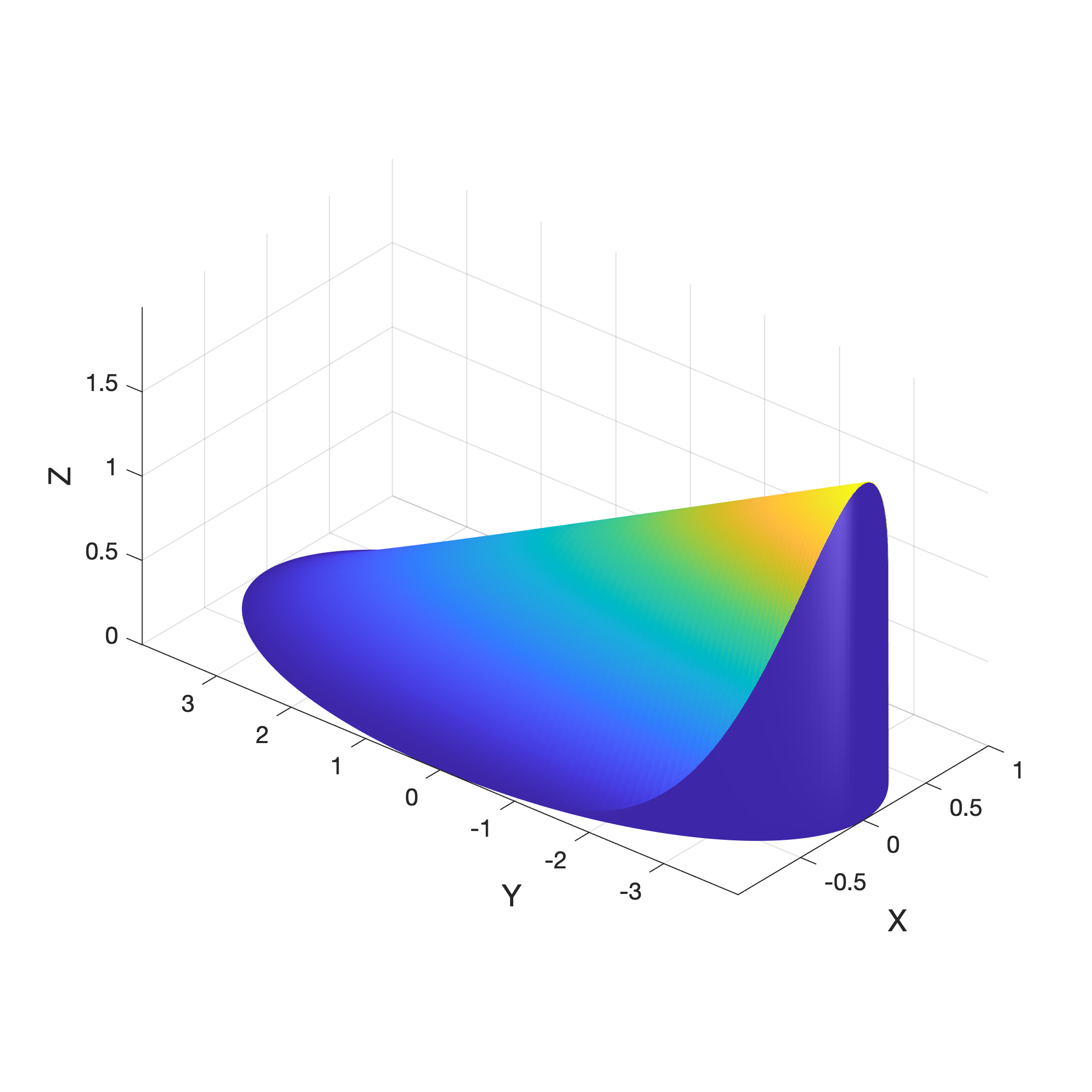}
         \caption{True solution}
     \end{subfigure}
     \begin{subfigure}[b]{0.4\textwidth}
         \centering
         \includegraphics[width=\textwidth]{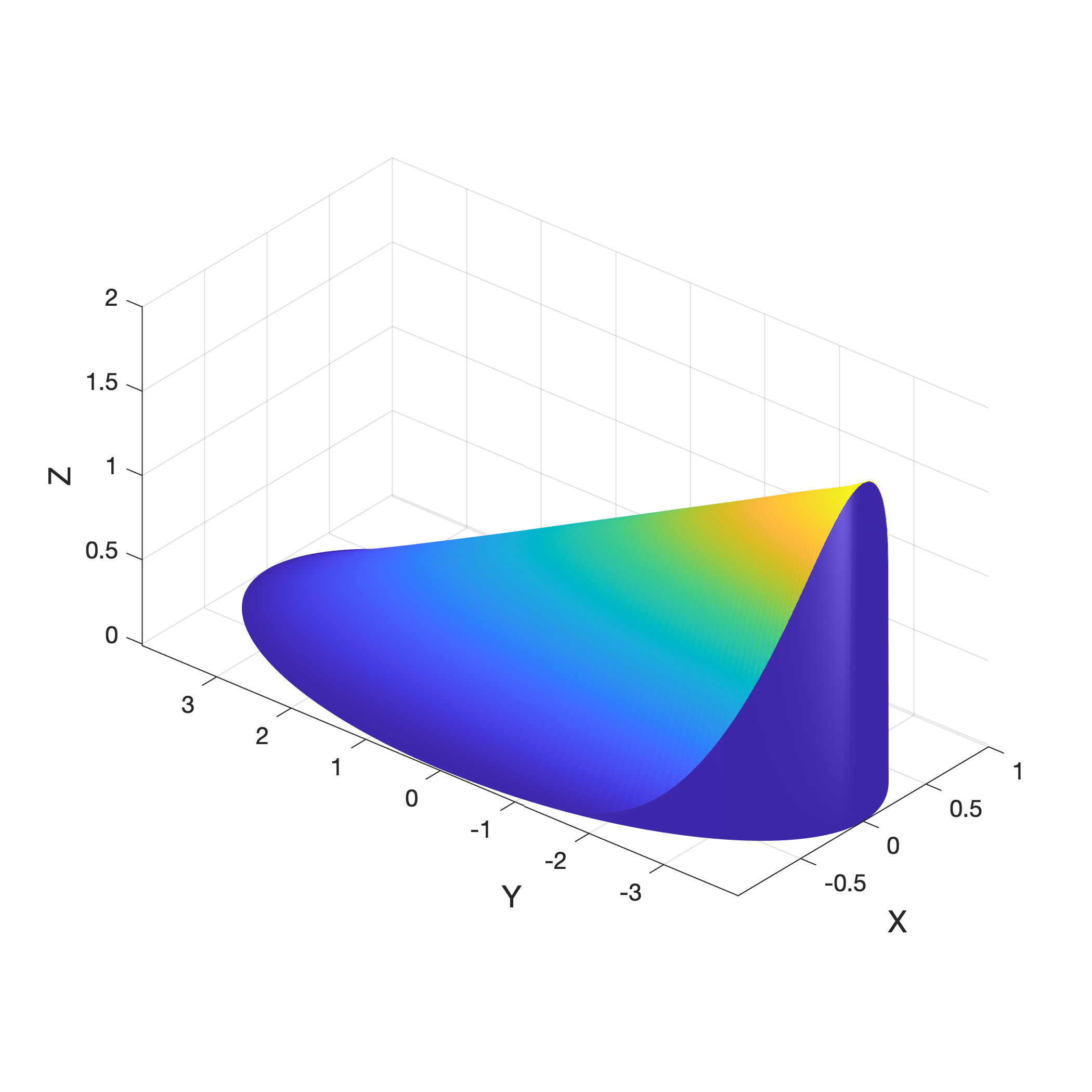}
         \caption{SL-PINN using $L^2$ training}
     \end{subfigure}
     \begin{subfigure}[b]{0.4\textwidth}
         \centering
         \includegraphics[width=\textwidth]{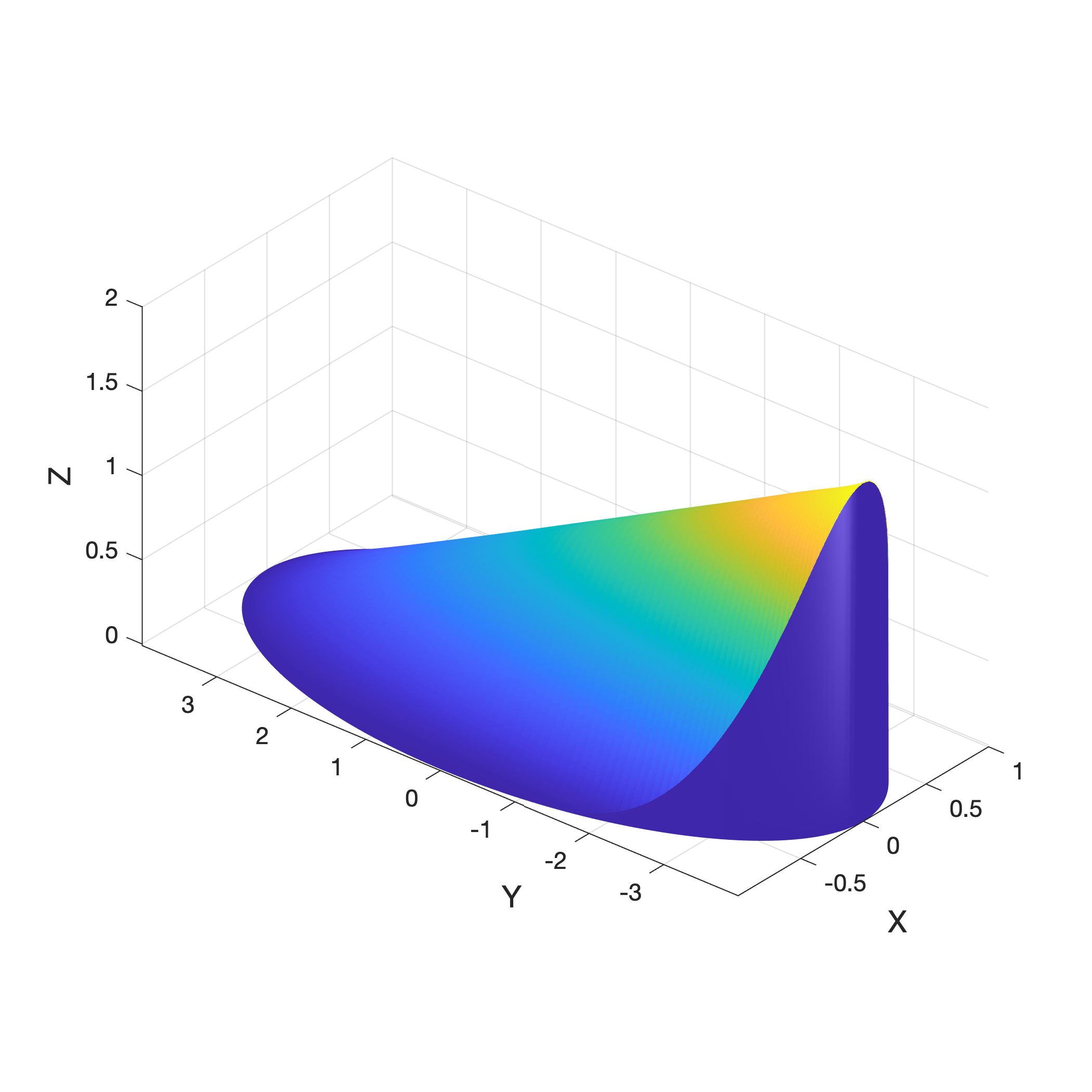}
         \caption{SL-PINN using $L^1$ training}
     \end{subfigure}
     \begin{subfigure}[b]{0.4\textwidth}
         \centering
         \includegraphics[width=\textwidth]{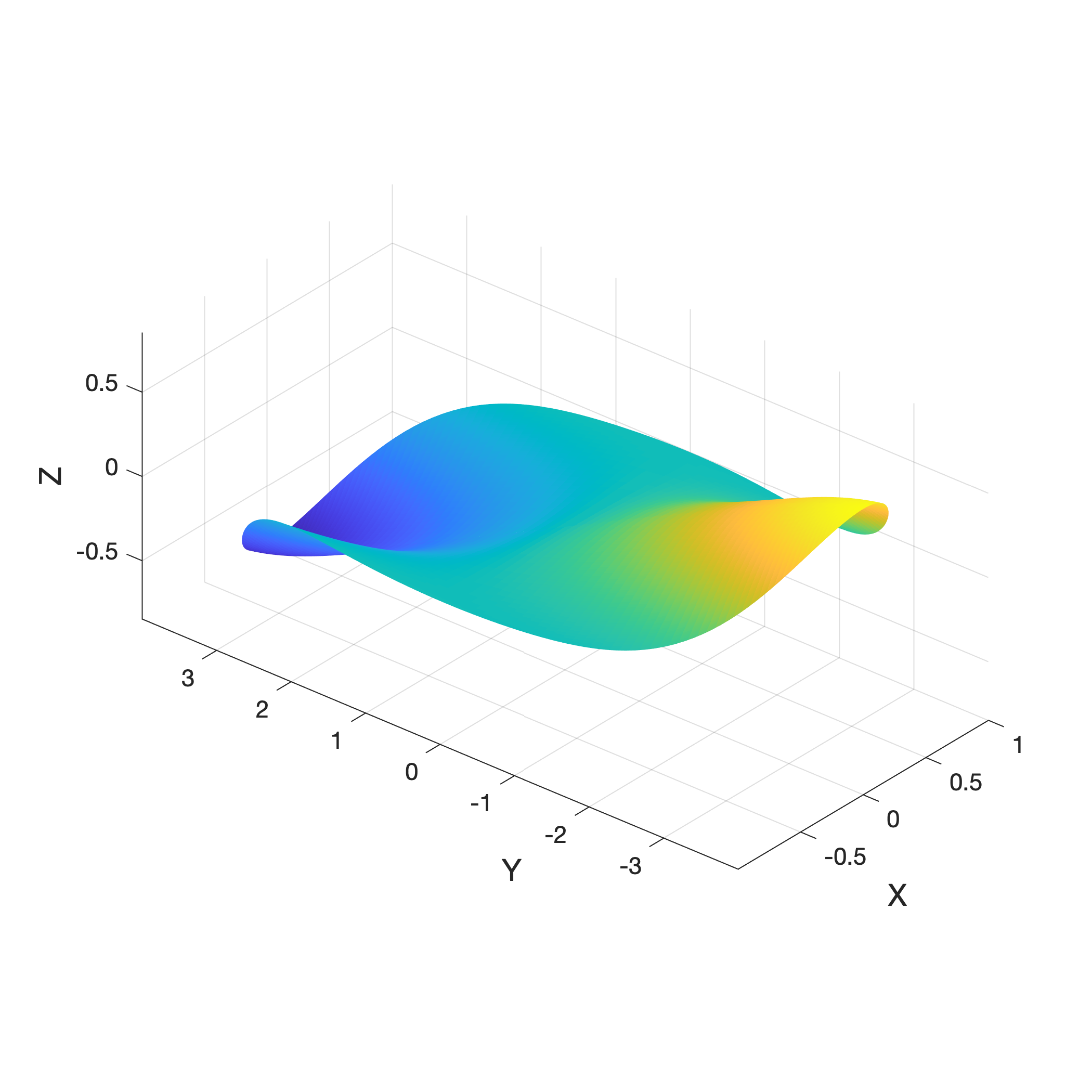}
         \caption{Conventional PINN}
     \end{subfigure}
     \caption{Numerical prediction of \eqref{e:ellipse} with $(A,B)=(1,4)$ and $\ep = 10^{-6}$ when $f={(1-({x}/{A})^2)^2}/{B}$. For our simulations, we select a uniform grid of $50$ discretized points in each of the $r$ and $\tau$ directions.}\label{fig7-0}
\end{figure}

\begin{figure}
    \centering
    \begin{subfigure}[b]{0.45\textwidth}
        \centering
        \includegraphics[width=\textwidth]{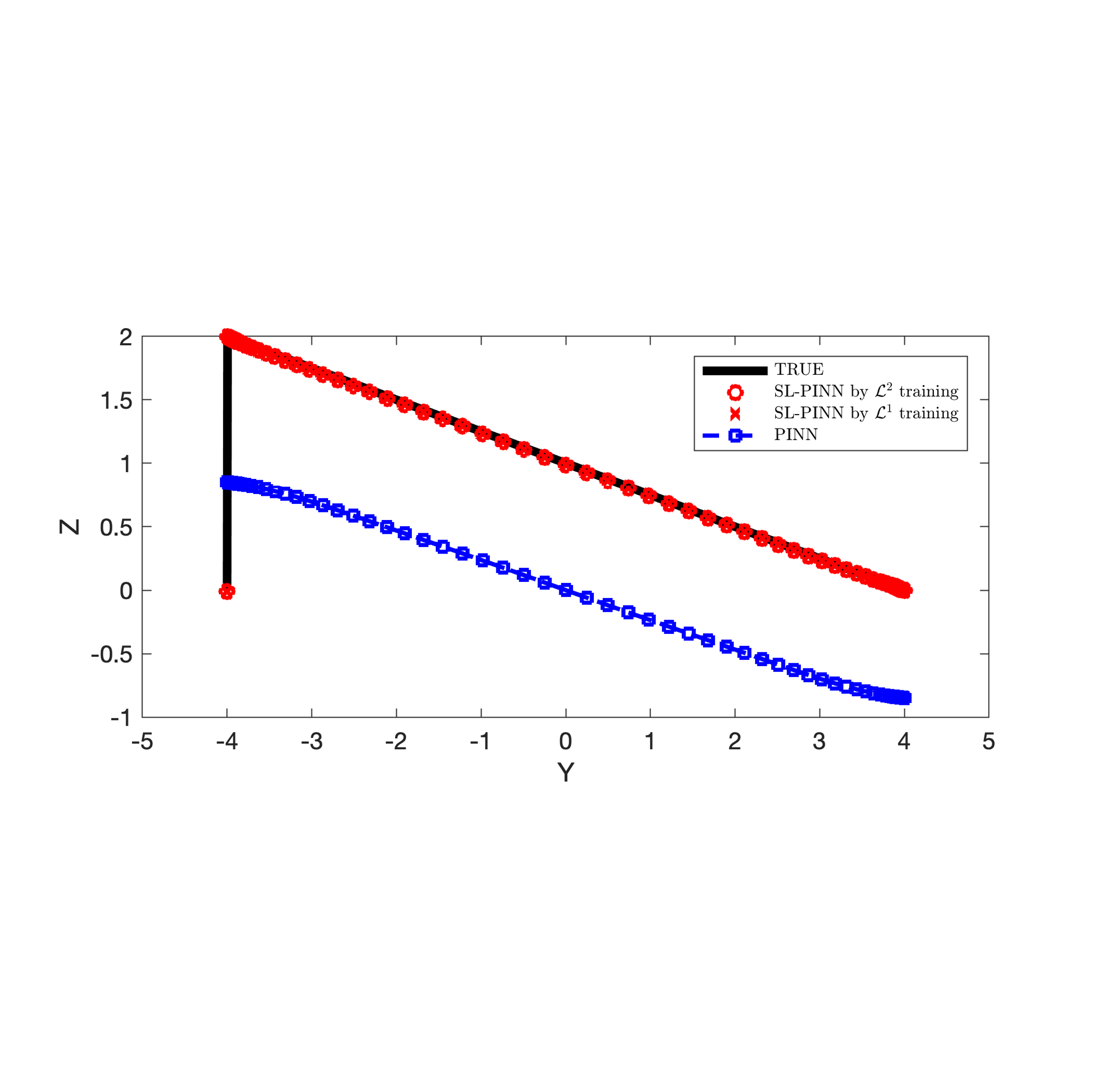}
        \caption{True soltuion at t=1}
    \end{subfigure}
    \begin{subfigure}[b]{0.45\textwidth}
        \centering
        \includegraphics[width=\textwidth]{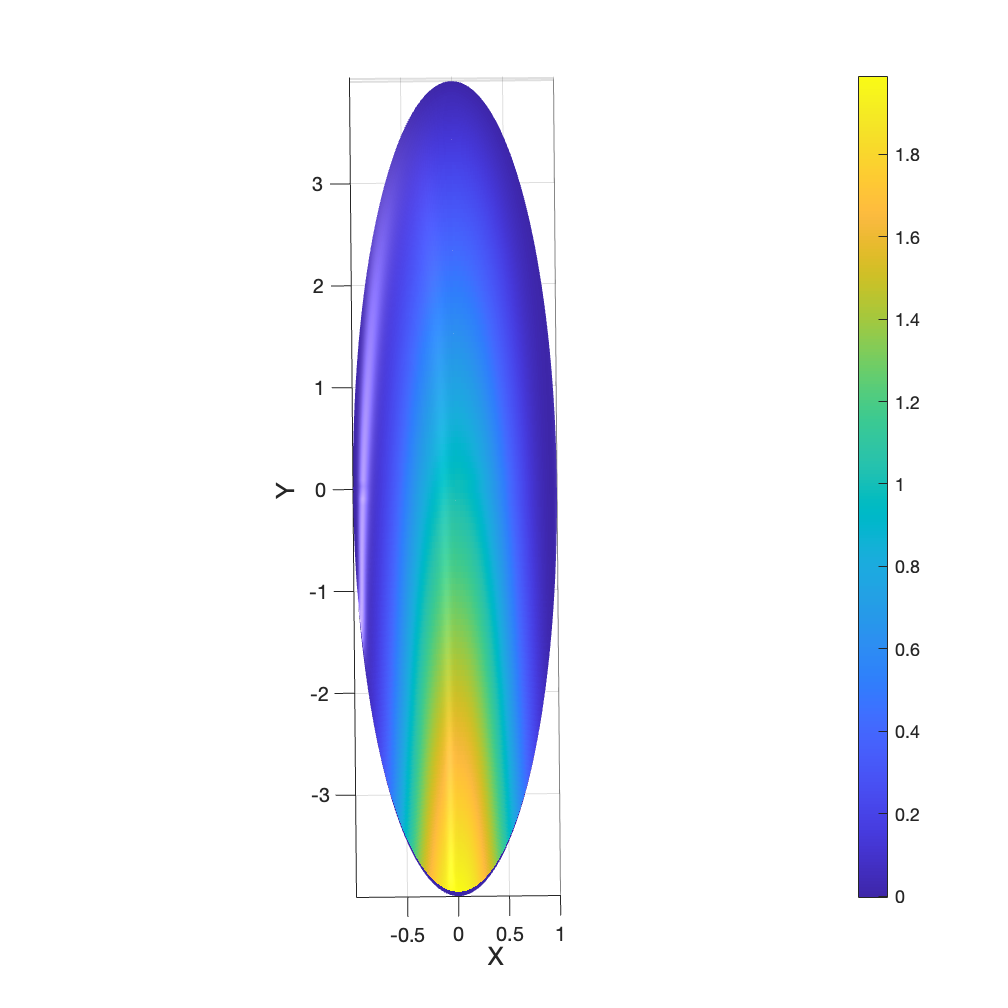}
        \caption{SL-PINN}
    \end{subfigure}
    \caption{The one-dimensional profile of predicted solutions in Figure \ref{fig7-0} along the line $\tau = \pi/2$ and solution surface projection profile.}\label{fig7}
\end{figure}

\begin{table}[htbp]

\begin{tabular}{ |p{1.8cm}|p{1.8cm}|p{1.8cm}|p{1.8cm}|p{1.8cm}|p{1.8cm}|p{1.8cm}|p{1.8cm}| }
 \hline
 \multicolumn{8}{|c|}{Relative $L^2$ error by the $L^2$ training (PINN)} \\
 \hline
    & square (Fig. \ref{fig1})& circle with non compatible (Fig. \ref{fig3}) & circle with compatible (Fig. \ref{fig2}) & ellipse 4:1 (Fig. \ref{fig6})& ellipse 1:4 (Fig. \ref{fig7-0})& oscillation (Fig. \ref{fig4}) & non-linear case (Fig. \ref{fig non-li}) \\
 \hline
 $\epsilon=10^{-4}$   & $9.55\times10^{-1}$    &$9.20\times10^{-1}$ &  $9.27\times10^{-1}$&  $8.55\times10^{-1}$&  $7.84\times10^{-1}$&  $8.76\times10^{-1}$&  $2.90\times10^{-1}$ \\
 $\epsilon=10^{-6}$&   $9.55\times10^{-1}$  & $9.21\times10^{-1}$   &$9.27\times10^{-1}$ &$8.65\times10^{-1}$&  $7.03\times10^{-1}$&  $9.12\times10^{-1}$&  $6.11\times10^{-1}$ \\
  $\epsilon=10^{-8}$ & $9.55\times10^{-1}$&$9.15\times10^{-1}$ &$9.20\times10^{-1}$ &  $8.62\times10^{-1}$ &  $7.85\times10^{-1}$ &  $9.23\times10^{-1}$ &  $4.13\times10^{-1}$ \\
 \hline
\end{tabular}

\begin{tabular}{ |p{1.8cm}|p{1.8cm}|p{1.8cm}|p{1.8cm}|p{1.8cm}|p{1.8cm}|p{1.8cm}|p{1.8cm}| }
 \hline
 \multicolumn{8}{|c|}{Relative $L^2$ error by $L^2$ training (SL-PINN)} \\
 \hline
 & square (Fig. \ref{fig1})& circle with non compatible (Fig. \ref{fig3}) & circle with compatible (Fig. \ref{fig2}) & ellipse 4:1 (Fig. \ref{fig6})& ellipse 1:4 (Fig. \ref{fig7-0})& oscillation (Fig. \ref{fig4})& non-linear case (Fig. \ref{fig non-li}) \\
 \hline
 $\epsilon=10^{-4}$   & $2.5\times10^{-3}$    &$3.2\times10^{-3}$ &  $1.6\times10^{-3}$&  $5.1\times10^{-3}$&  $2.6\times10^{-3}$&  $6.5\times10^{-3}$&  $4.2\times10^{-3}$ \\
 $\epsilon=10^{-6}$&   $2.0\times10^{-3}$  & $2.8\times10^{-3}$   &$1.5\times10^{-3}$ &$3.9\times10^{-3}$&  $1.7\times10^{-3}$&  $7.6\times10^{-3}$&  $3.5\times10^{-3}$ \\
  $\epsilon=10^{-8}$ & $2.0\times10^{-3}$   &$4.7\times10^{-3}$ &$2.3\times10^{-3}$ &  $5.8\times10^{-3}$ &  $2.8\times10^{-3}$ &  $9.5\times10^{-3}$&  $5.0\times10^{-3}$ \\
 \hline
\end{tabular}

\begin{tabular}{ |p{1.8cm}|p{1.8cm}|p{1.8cm}|p{1.8cm}|p{1.8cm}|p{1.8cm}|p{1.8cm}|p{1.8cm}|  }
 \hline
 \multicolumn{8}{|c|}{Relative $L^2$ error by $L^1$ training (SL-PINN)} \\
 \hline
 & square (Fig. \ref{fig1})& circle with non compatible (Fig. \ref{fig3}) & circle with compatible (Fig. \ref{fig2}) & ellipse 4:1 (Fig. \ref{fig6})& ellipse 1:4 (Fig. \ref{fig7-0})& oscillation (Fig. \ref{fig4})& non-linear case (Fig. \ref{fig non-li}) \\
 \hline
 $\epsilon=10^{-4}$   & $9.5\times10^{-5}$    &$4.2\times10^{-3}$ &  $2.6\times10^{-3}$&  $2.3\times10^{-3}$&  $3.4\times10^{-3}$&  $4.0\times10^{-2}$&  $3.5\times10^{-3}$ \\
 $\epsilon=10^{-6}$&   $1.4\times10^{-4}$  & $3.6\times10^{-3}$   &$4.6\times10^{-3}$ &$5.6\times10^{-3}$&  $4.2\times10^{-3}$&  $5.2\times10^{-2}$&  $4.0\times10^{-3}$ \\
  $\epsilon=10^{-8}$ & $3.9\times10^{-4}$   &$4.3\times10^{-3}$ &$1.8\times10^{-3}$ &  $3.4\times10^{-3}$ &  $4.8\times10^{-3}$ &  $6.2\times10^{-2}$ &  $5.0\times10^{-3}$ \\
 \hline
\end{tabular}

\caption{Comparison of relative $L^2$ errors for equations \eqref{e:square}, \eqref{e:circle_eq}, and \eqref{e:ellipse} utilizing standard 5-layer PINN, SL-PINN employing $L^2$ training, and SL-PINN employing $L^1$ training.} \label{tab:1}
\end{table}
Therefore, the loss \eqref{e:ellipse_loss_second} becomes nearly constant due to \eqref{e:normal_bound_ellipse_y} and \eqref{e:bound_ellipse020} when \(\epsilon\) is sufficiently small. This simplification facilitates feasible computations in the SL-PINN approach. The effectiveness of our new approach is demonstrated through a series of numerical experiments presented in Figures \ref{fig7-0} and \ref{fig7}.
We employed a grid comprising 50 uniformly spaced points for these experiments in both the $\tau$ and $\eta$ directions.
We compare the standard five-layer PINN with our SL-PINN. 
Figure \ref{fig7-0} shows the numerical solutions of \eqref{e:ellipse} with $(A,B)=(1,4)$ and $\ep = 10^{-6}$ when $f={(1-({x}/{A})^2)^2}/{B}$. 
The figure illustrates the inaccurate predicted solution computed by the conventional PINN for the singular perturbation problem. In contrast, our novel approach demonstrates superior performance; see e.g., Table \ref{tab:1} for more details.
Figure \ref{fig7} offers a detailed view of the one-dimensional profile of predicted solutions at $\tau=\pi/2$. This demonstrates the high accuracy of both the $L^1$ and $L^2$ training approaches in predicting solutions. However, the traditional PINN method falls short of capturing the sharp transition near the boundary layer.

\section{Conclusion}

In this study, we have introduced a semi-analytic approach to enhance the numerical performance of PINNs in addressing a range of singularly perturbed boundary value problems and convection-dominated equations on rectangular, circular, and elliptical domains. For each singular perturbation problem examined, we derived an analytic approximation known as the corrector function, which captures the behavior of the fast (stiff) component of the solution within the boundary layer. Additionally, we conducted a detailed boundary layer analysis to better understand the dynamics within these regions. 
By integrating these corrector functions into a 2-layer PINN framework with hard constraints, we successfully addressed the stiffness inherent in the approximate solutions, resulting in our novel semi-analytic SL-PINNs enriched with the correctors. Through the incorporation of these corrector functions, we effectively overcame the stiffness challenge associated with approximate solutions, leading to the development of our improved SL-PINNs.

For future work, we aim to extend our approach to more complex geometries, which present significant challenges. Addressing these complexities may require a greater reliance on machine learning predictions to accurately capture the intricate boundary layer dynamics. Additionally, we plan to apply our semi-analytic SL-PINN methodology to Burgers' equations in one and two dimensions. This extension will allow us to further test and refine our approach, demonstrating its versatility and robustness in handling a broader range of nonlinear convection-diffusion problems.


\section{Appendix}
We consider the convection-diffusion equations in an elliptical domain, 
\begin{equation}\label{original pde}
\left\{
\begin{aligned}
L_\varepsilon u^\varepsilon &:= -\varepsilon \Delta u^\varepsilon - u^\varepsilon_y = f(x, y) \text{ in } \Omega, \\
u^\varepsilon &= 0 \text{ on } \partial \Omega,
\end{aligned}
\right.
\end{equation}
where \(0 < \varepsilon \ll 1\). 
The domain \(\Omega\) is an ellipse defined as \( \Omega = \{(x, y) \mid \frac{x^2}{A^2} + \frac{y^2}{B^2} < 1\} \), where \(A = a \cosh R\) and \(B = a \sinh R\). The function \(f\) is assumed to be sufficiently smooth. For the domain \(\Omega\), we particularly focus on the ellipse with its major axis parallel to the \(x\)-axis. The limit problem (i.e., when \(\varepsilon = 0\)) is defined by
\begin{equation}\label{limit equation}
\left\{
\begin{aligned}
-u^0_y &= f(x, y) \text{ in } \Omega, \\
u^0 &= 0 \text{ on } \Gamma_u,
\end{aligned}
\right.
\end{equation}
where \(\Gamma_u = \{(x, y) \mid \frac{x^2}{A^{2}} + \frac{y^{2}}{B^{2}} = 1, y > 0\}\). The explicit solution of equation (\ref{limit equation}) is
\begin{equation}\label{limit solution}
u^0(x, y) = \int^{C_u(x)}_{y} f(x, s) ds, \quad (x, y) \in \Omega,
\end{equation}
which \(C_u(x) = B\sqrt{1 - \frac{x^{2}}{A^{2}}}\) is upper halves of the ellipse. We define boundary-fitted coordinates for an elliptic coordinate system such that 
\begin{equation}
x = a \cosh (R-\eta) \cos \tau, \quad y = a \sinh (R-\eta) \sin \tau,
\end{equation}
where the elliptic domain \( \Omega \) changed to \( D = \{ (\eta, \tau) \in (0, R) \times (0, 2\pi) \}\), and \( R \) is the maximal radial extent.
We also define the domain \(D_{3R/4}\) as \(D_{3R/4} = \{ (\eta, \tau) \in D : r \leq \frac{3R}{4} \}\).
In these coordinates, the transformed differential operator for equation (\ref{original pde}) is given by
\begin{equation}\label{Transforemd pde}
\begin{aligned}
L_\epsilon u^\epsilon &= -\epsilon \Delta u^\epsilon - u^\epsilon_y \\
&= \frac{1}{H} \left(-\epsilon (v^{\epsilon}_{\eta \eta} + v^{\epsilon}_{\tau \tau}) + (a \cosh (R-\eta) \sin \tau) v^{\epsilon}_{\eta} - (a \sinh (R-\eta) \cos \tau) v^{\epsilon}_{\tau} \right),
\end{aligned}
\end{equation}
where \(H = (a\sinh (R-\eta) \cos \tau)^{2} + (a \cosh (R-\eta) \sin \tau)^{2}\).
Considering the stretched variable \(\bar{\eta} = \frac{\eta}{\epsilon}\), we transform \eqref{Transforemd pde} to
\begin{equation}\label{enriched transforemd pde}
\begin{aligned}
L_\epsilon u^\epsilon &= -\epsilon \Delta u^\epsilon - u^\epsilon_y \\
&= \frac{1}{H} \left(-\frac{1}{\epsilon} v^{\epsilon}_{\bar{\eta}\bar{\eta}} -\epsilon v^{\epsilon}_{\tau \tau} + \frac{1}{\epsilon}(a \cosh (R- \bar{\eta}) \sin \tau) v^{\epsilon}_{\bar{\eta}} - (a \sinh (R-\bar{\eta}) \cos \tau) v^{\epsilon}_{\tau} \right).
\end{aligned}
\end{equation}
To eliminate the leading terms of \eqref{enriched transforemd pde}, we derive the following equation for the first corrector \(\theta^0\):
\begin{equation}\label{corrector equation}
\left\{
\begin{aligned}
-\frac{\partial^2 \varphi^{0}}{\partial \bar{\eta}^2} + a\cosh( R-\epsilon \bar{\eta}) \sin \tau \frac{\partial \varphi^{0}}{\partial \bar{\eta}} &= 0, & \text{for } 0 < \bar{\eta} < \infty, \pi < \tau < 2\pi, \\
\varphi^{0} &= -u^0 (a \cosh R \cos \tau, a\sinh R \sin \tau) & \text{at } \bar{\eta} = 0, \\
\varphi^{0} &\rightarrow 0 & \text{as } \bar{\eta} \rightarrow \infty.
\end{aligned}
\right.
\end{equation}
Instead of directly finding the explicit solution in \eqref{corrector equation}, we seek an approximation of the solution in \eqref{Approximated corrector equation} using the following form
\begin{equation}\label{Approximated corrector equation}
\left\{
\begin{aligned}
-\frac{\partial^2 \varphi^0}{\partial \bar{\eta}^2} + a\cosh R \sin \tau \frac{\partial \varphi^0}{\partial \bar{\eta}} &= 0, & \text{for } 0 < \bar{\eta} < \infty, \pi < \tau < 2\pi, \\
\varphi^0 &= -u^0 (a \cosh R \cos \tau, a\sinh R \sin \tau) & \text{at } \bar{\eta} = 0, \\
\varphi0 &\rightarrow 0 & \text{as } \bar{\eta} \rightarrow \infty.
\end{aligned}
\right.
\end{equation}
An explicit solution is directly obtained by
\begin{equation}
\varphi^0 = -u^0 (a \cosh R \cos \tau, a \sinh R \sin \tau) \exp \left( \frac{A \sin \tau }{\epsilon}\eta \right) \chi_{[\pi, 2\pi]}(\tau),
\end{equation}
where \(\chi_A\) is the characteristic function and \(A = a \cosh R\). 
Introducing a cut-off function to satisfy the boundary condition, we write an approximate form of \(\varphi^0\)
\begin{equation}\label{delta corrector}
\bar{\varphi}^0 = -u^0 (a \cosh R \cos \tau, a \sinh R \sin \tau) \exp \left( \frac{A \sin \tau }{\epsilon}\eta \right) \delta(\eta) \chi_{[\pi, 2\pi]}(\tau),
\end{equation}
where \(\delta(\eta)\) is a smooth cut-off function such that \(\delta(\eta) = 1\) for \(\eta \in [0, R/2]\) and \(= 0\) for \(\eta \in [3R/4, R]\).
For notational convenience, the following expressions are defined:
\begin{equation}\notag
    B_{x}(\tau) = a \cosh R \cos \tau,
\end{equation}
\begin{equation}\notag
     B_{y}(\tau) = a \sinh R \sin \tau,
\end{equation}
\begin{equation}
     \bar{\varphi}^0 = -u^0 ( B_{x}(\tau), B_{y}(\tau)) \exp \left( \frac{A \sin \tau }{\epsilon} \eta \right) \delta(\eta) \chi_{[\pi, 2\pi]}(\tau).
\end{equation}
Since \(\varphi^0\) vanishes \(u^0\) like at \(\tau = \pi, 2\pi\), \(\varphi^0\) is continuous and piecewise smooth on \(\overline{\Omega}\), and thus we conclude that \(\varphi^0, \bar{\varphi}^0 \in H^1(\Omega)\). 
We then obtain that 
\begin{equation}
u^0 + \bar{\varphi}^0 \in H^1_0(\Omega).
\end{equation}
Note that if \( f = f(x, y) \) does not vanish at the characteristic points \((\pm a \cosh R, 0)\), \( f(x, y) \) is considered incompatible, resulting in singularities in the derivatives of \( u^0 \) at these points (see, e.g., \cite{JUNG201188}). Although the study of non-compatible cases can be pursued elsewhere, this paper assumes the compatibility condition to concentrate on the numerical analysis.
\begin{lemma}\label{compatibility_lemma}
We assume that
\begin{equation}
\frac{\partial^{\alpha+\beta} f(x, y)}{\partial x^\alpha \partial y^\beta} = 0 \quad \text{at } (\pm a\cosh R, 0), \quad 0 \leq 2\alpha + \beta \leq \gamma - 1, \quad \gamma \geq 1, \quad \alpha, \beta \geq 0,
\end{equation}
where \( f(x, y) \) belongs to \( C^\gamma(D) \), and \( D \) is the ellipse as in (\ref{original pde}). Then the following function
\begin{equation}
\frac{f(x, C_u(x))}{C^\gamma_u(x)}
\end{equation}
is bounded for all \( x \in (-a\cosh R, a\cosh R) \).    
\end{lemma}
We omit the proof as it is straightforward and follows a similar approach to that in \cite{JUNG201188}.

In this paper, we assume the following compatibility condition,
\begin{equation} \label{e:append_comp_cond}
\frac{\partial^{p_{1}+p_{2}} f}{\partial x^{p_{1}} \partial y^{p_{2}}} = 0 \quad \text{at} \quad (\pm a\cosh R, 0) \quad \text{for} \quad 0 \leq 2p_{1} + p_{2} \leq 2 , \quad p_{1}, p_{2} \geq 0,
\end{equation}
With the compatibility condition, we can bound the corrector function.
\begin{lemma}\label{b:corrector_estimate}
There exists a positive constant \( k \) such that, for integers \( l, n, s \geq 0, m = 0, 1, 2 \), and for \( 1 \leq p \leq \infty \),
\begin{equation}
\left| (\sin \tau)^{-l} \left( \frac{\eta}{\varepsilon} \right)^n  \frac{\partial^{s+m} \varphi^{0}}{\partial \eta^s \partial \tau^m} \right|_{L_p(D)} \leq \kappa \sup_{\tau} |a_{0,h}(\tau)| \varepsilon^{\frac{1}{p}-s},
\end{equation}
where \( h = -s + m + l + n + 1 \), and
\begin{equation}
a_{0,q}(\tau) = \sum_{\substack{m+r \leq q \\ m,r \geq 0}} c_{m,r} \frac{d^r v^0(\tau)}{\sin^m \tau \, d\tau^r},
\end{equation}
\( v^0(\tau) = -u^0(a\cosh R\cos \tau, a\sinh R\sin \tau) \) and the \( c_{m,r} = c_{m,r}(\tau)  \in C^{\infty}([0,2\pi])\), may be different at different occurrences,
If \(q<0 \), then \( a_{0,q}\) is simply in \(C^{\infty}([0,2\pi])\) fuction.
The notation \( a_{0,q} \) follows the same convention as outlined in the \textit{notation convention} of \cite{JUNG201188}.
\end{lemma}
\begin{proof}
By using an Inductive approach, for integers \( s \geq 0, m = 0, 1, 2 \), we easily obtain
\begin{equation}
\frac{\partial^{s+m}\varphi^0}{\partial \bar{\eta}^{s} \partial \tau^{m}} = \epsilon^{s} \frac{\partial^{s+m}\varphi^0}{\partial \eta^{s} \partial \tau^{m}} = a_{0,-s+m}(\tau) \sum_{k=0}^{m} \left( (-\sin \tau)\bar{\eta} \right)^{k} \exp\left( (A\sin \tau)\bar{\eta} \right) \chi_{[\pi,2\pi]}(\tau).
\end{equation}
\text{Hence, we observe that}
\begin{equation}
(\sin \tau)^{-l} \bar{\eta}^{n} \frac{\partial^{s+m}\varphi^0}{\partial \eta^{s} \partial \tau^{m}} = a_{0,-s+m}(\tau) (\sin \tau)^{-1-l-n} \sum_{k=0}^{m} (-\sin \tau)\left( (-\sin \tau)\bar{\eta} \right)^{k+n} \epsilon^{-s}\exp\left( (A\sin \tau)\bar{\eta} \right) \chi_{[\pi,2\pi]}(\tau),
\end{equation}

\begin{equation}
\left| (\sin \tau)^{-l} \bar{\eta}^{n} \frac{\partial^{s+m}\varphi^0}{\partial \eta^{s} \partial \tau^{m}} \right|_{L^p(D)}
\end{equation}
\begin{equation}
\begin{aligned}
= \left| a_{0,-s+m}(\tau) (\sin \tau)^{-1-l-n} \sum_{k=0}^{m} (-\sin \tau)\left( (-\sin \tau)\bar{\eta} \right)^{k+n} \epsilon^{-s}\exp\left( (A\sin \tau)\bar{\eta} \right) \right|_{L^p((0,R) \times (\pi,2\pi))}
\end{aligned}
\end{equation}

\begin{equation}
\leq k \left\{ \sup_{\tau} |a_{0,-s+m+l+n+1}(\tau)| \right\} \left( \int_{\pi}^{2\pi} \int_{0}^{R} |\sum_{k=0}^{m} \left( (-\sin \tau)\bar{\eta} \right)^{k+n}|^{p} (-\sin \tau)^{p}\epsilon^{-ps}\exp\left( (p A\sin \tau)\bar{\eta} \right)d\eta d\tau \right)^{\frac{1}{p}}
\end{equation}

\begin{equation}
\leq k \left\{ \sup_{\tau} |a_{0,-s+m+l+n+1}(\tau)| \right\} \left( \int_{\pi}^{2\pi} \int_{0}^{R}  \kappa(-\sin \tau)^{p}\epsilon^{-ps}\exp\left( (p A\sin \tau)\bar{\eta} \right)d\eta d\tau \right)^{\frac{1}{p}}
\end{equation}
Using 
\begin{equation}
\int_{\pi}^{2\pi} \int_{0}^{R} (-\sin \tau)^{p} \exp\left( c(A \sin \tau)\bar{\eta}\right) d\eta d\tau \leq \kappa \epsilon,
\end{equation}
we obtain
\begin{equation}
\left| (\sin \tau)^{-l} \left( \frac{\eta}{\varepsilon} \right)^n  \frac{\partial^{s+m} \varphi^{0}}{\partial \eta^s \partial \tau^m} \right|_{L_p(D)} \leq \kappa \sup_{\tau} |a_{0,-s+m+l+n+1}(\tau)| \varepsilon^{\frac{1}{p}-s}.
\end{equation}
\end{proof}
In the following lemma and theorem, we illustrate how the compatibility conditions are applied to derive the regularity and boundedness of the solutions. For simplicity, we introduce \(v^0\) as follows:
\begin{equation}
v^{0}(\tau)=u^0({B_{x}(\tau)}, {B_{y}(\tau)}) = \int_{-a \sinh R \sin \tau}^{a \sinh R \sin \tau} f(a \cosh R \cos \tau, s) \, ds, \quad \pi < \tau < 2\pi.
\end{equation}

\begin{lemma}
If \( q \leq 1 \) {\color{red} or } the compatibility conditions \eqref{e:append_comp_cond} hold for \( 2 + 3n \geq q-2 \geq 0 \), then \( a_{0,q}(\tau) \) is bounded for \( \tau \in [\pi, 2\pi] \).
\end{lemma}
\begin{proof}
It is sufficient to prove that
\begin{equation}\label{partial_proof}
  \frac{1}{(\sin \tau)^{q-r}} \frac{d^{r} v^0}{d \tau^{r}} (\tau) < \infty, \quad \text{for } r = 0, 1, \ldots, q.
\end{equation}
For the \(q\leq 1\), we can easily verify \eqref{partial_proof} by using \eqref{limit solution} .
For the \(q\geq 2\),
we derive the recursive relation for the derivatives of \(v^0(\tau)\).
Using an inductive approach, we easily obtain
\begin{equation}\label{recursive form}
\frac{d^r v^0 (\tau)}{d \tau^r} = \sum_{\substack{l+s \leq r \\ l, s \geq 0}} c_{ls} (\sin \tau)^{2l-r+s} \frac{\partial^{l+s} u^0}{\partial x^l \partial y^s} ({B_{x}(\tau)}, {B_{y}(\tau)}) + \sum_{0 \leq r' \leq r-1} c_{r'} (\sin \tau)^{r'-r} \frac{d^{r'} v^0 (\tau)}{d \tau^{r'}}.
\end{equation}
To bound the second term on the right-hand side of (\ref{recursive form}), we find that
\begin{equation}
\frac{\partial^l u^0}{\partial x^l}  (x, y) = \sum_{l'+s' \leq l-1, l',s' \geq 0} g_{l'ls'} (x) \frac{\partial^{l'+s'} f}{\partial x^{l'} \partial y^{s'}} (x, C_u(x)) \\
+ c_{l}   \int^{C_u(x)}_{y} \frac{\partial^l f}{\partial x^l}   (x, s) ds,
\end{equation}

\begin{equation}
\left\{ \frac{\partial^s}{\partial y^s} \left[ \frac{\partial^l u^0}{\partial x^l} \right] \right\} (x, y) = \sum_{l'+s' \leq l-1, l',s' \geq 0} g_{l'ls's} (x) \frac{\partial^{l'+s'} f}{\partial x^{l'} \partial y^{s'}} (x, C_u(x)) \\
+ c_{ls} \left\{ \frac{\partial^{s-1}}{\partial y^{s-1}} \left[ \frac{\partial^l f}{\partial x^l} \right] \right\} (x, y),
\end{equation}
where
\begin{align}
\frac{\partial^{-1} f}{\partial y^{-1}}(x,y) &:= \int_{y}^{C_u(x)} f(x, s) \, ds, \quad \text{and} \quad \left| g_{l'ls's}(x) \right| \leq \kappa C_u(x)^{-(-1+2l-2l'-s')}.
\end{align}
Then, we rewrite (\ref{recursive form}) as
\begin{align}\label{recursive form expansion}
\frac{1}{(\sin \tau)^{q-r}} \frac{d^{r} v^0}{d \tau^{r}} (\tau) &= \sum_{\substack{l+s \leq r, l, s \geq 0 \\ l'+s' \leq l-1, l', s' \geq 0}} \tilde{g}_{l'ls's}({B_{x}(\tau)}) \left( \frac{\partial^{l'+s'}}{\partial x^{l'} \partial y^{s'}} f \right)({B_{x}(\tau)}, {B_{y}(\tau)}) \nonumber \\
&\quad + \sum_{\substack{l+s \leq r,\\ l, s \geq 0 }} \tilde{c}_{ls} (\sin \tau) \left( \frac{\partial^{s-1}}{\partial y^{s-1}} \left[ \frac{\partial^{l}}{\partial x^{l}} f \right] \right) ({B_{x}(\tau)}, {B_{y}(\tau)}) \nonumber \\
&\quad + \sum_{0 \leq r' \leq r-1} \frac{c_{r'}}{(\sin \tau)^{q-r'}} \frac{d^{r'} v^0(\tau)}{d \tau^{r'}},
\end{align}
where
\[
 \left| \tilde{g}_{l'ls's}({B_{x}(\tau)}) \right| \leq \kappa (\sin \tau)^{-(1-2l'-s'+q-s)} \text{ and } \left| \tilde{c}_{ls}(\sin \tau) \right| \leq \kappa (\sin \tau)^{-(q-2l-s)}.
\]
For \(p'\), \(l'\), \(q'\), and \(s'\), 
with \[
0 \leq 2(p' + l') + q' + s'  \leq -1+q-s-1 \leq -2+q,
\]
we obtain that
\begin{equation}
\frac{\partial^{p'+q'}}{\partial x^{p'} \partial y^{q'}} \left( \frac{\partial^{l'+s'}}{\partial x^{l'} \partial y^{s'}} f \right) = 0 \quad \text{at } (\pm a\cosh R, 0) \quad \text{for } 0 \leq 2p' + q' \leq -1  - 2l' - s' + q - s - 1.
\end{equation}
Similarly, considering
\[
0 \leq 2(p' + l') + q' + s' - 1 \leq -2 + q,
\]
we deduce that
\begin{equation}
\frac{\partial^{p'+q'}}{\partial x^{p'} \partial y^{q'}} \left( \frac{\partial^{s-1}}{\partial y^{s-1}} \left[ \frac{\partial^{l}}{\partial x^{l}} f \right] \right) = 0 \quad \text{at } (\pm a\cosh R, 0) \quad \text{for } 2p' + q' \leq q - 2l - s - 1.
\end{equation}
By applying \eqref{compatibility_lemma} we establish the boundedness of the first and second terms of \eqref{recursive form expansion}.
To demonstrate the boundedness of the third sum on the right-hand side of the equation \eqref{recursive form expansion}, we use an inductive approach on the variable \( r \).
We consider the summation
\[
S_r = \sum_{0 \leq r' \leq r-1} \frac{c_{r'}}{(\sin \tau)^{q-r'}} \frac{d^{r'} v^0(\tau)}{d \tau^{r'}},
\]
where \( c_{r'} \) are coefficients, \(\tau\) is the variable, and \(v^0\) is a function of \(\tau\). We then aim to show that \(S_r\) is bounded for all \(r \geq 0\). We proceed by induction on \(r\).\\
\textbf{Initial Step:} For \(r = 0\),
\[
S_0 = 0,
\]
as there are no terms in the sum; hence, \(S_0\) is trivially bounded.\\
\textbf{Inductive Step:} Assume \(S_r\) is bounded for \(r \leq  R\).
We find that 
\begin{equation}\label{Inductive step}
S_{R+1} = \sum_{0 \leq r' \leq R-1} \frac{c_{r'}}{(\sin \tau)^{q-r'}} \frac{d^{r'} v^0(\tau)}{d \tau^{r'}} + \frac{c_{r-1}}{(\sin \tau)^{q-r+1}} \frac{d^{r-1} v^0(\tau)}{d \tau^{r-1}}.
\end{equation}
If we substitute the right-hand side of (\ref{Inductive step}) with the expression (\ref{recursive form expansion}) for \(S_R\), then under the assumption that \(S_R\) is bounded, \(S_{R+1}\) is also bounded.
\end{proof}

\begin{lemma}
Assume that the compatibility conditions \eqref{e:append_comp_cond} hold, then the following inequality is established:
\end{lemma}

\begin{equation}
\left |(\cosh R-\cosh (R-\eta)) \frac{\partial \varphi^0}{\partial \eta} \right|_{L^2(D)} \leq \kappa \varepsilon^{\frac{1}{2}}, \notag
\end{equation}
\begin{equation}
\left| \frac{\partial \varphi^0}{\partial \tau} \right|_{L^2(D)} \leq \kappa \varepsilon^{\frac{1}{2}}. \notag
\end{equation}

\begin{equation}
\left| L_\varepsilon (\varphi^0 - \bar{\varphi^0}) \right|_{L^2(D_{3R/4})} \leq \kappa \varepsilon^{\frac{1}{2}}.
\end{equation}

\begin{proof}

Applying \eqref{b:corrector_estimate} with \(s=1\), \(m=0\), \(l=0\), and \(n=0\),  we obtain an upper bound
\[
\left| \frac{\partial \varphi^0}{\partial \eta} \right|_{L^2(D^*)} \leq \kappa \epsilon^{1/2}.
\]
By using the Taylor expansion for \(\cosh R - \cosh (R-\eta)\), we obtain that 
\begin{equation}
\left|\left(\cosh R - \cosh (R-\eta)\right) \frac{\partial \varphi^0}{\partial \eta} \right|_{L^2(D)} \leq \left|\sum_{n=1}^{\infty} \frac{(\epsilon \eta)^n}{n!} \cosh R \right|_{L^2(D)} \left|\frac{\partial \varphi^0}{\partial \eta}\right|_{L^2(D)}.
\end{equation}
Given that \(R < 1\), we find that
\begin{equation}
\left|\left(\cosh R - \cosh (R-\eta)\right) \frac{\partial \varphi^0}{\partial \eta} \right|_{L^2(D^*)} \leq  \kappa \left(\sum_{n=1}^{\infty} \epsilon^n\right) \left|\frac{\partial \varphi^0}{\partial \eta}\right|_{L^2(D)} \leq \kappa \varepsilon^{\frac{1}{2}}.
\end{equation}
Similarly, by applying \eqref{b:corrector_estimate}, we can easily obtain the boundedness of the derivative in \(\tau\)
\[
\left| \frac{\partial \varphi^0}{\partial \tau} \right|_{L^2(D)} \leq \kappa \epsilon^{1/2}.
\]
It remains to estimate \( L_{\epsilon}(\varphi^{0} - \bar{\varphi}^{0}) \). From \eqref{b:corrector_estimate}, we deduce that:
\begin{align}\label{approximator_error_bound}
\left|\frac{\partial^2 (\varphi^{0} - \bar{\varphi^{0})}}{\partial \tau^2}\right|_{L^2(D_{3R/4})} &= \left|\frac{\partial^2 \varphi^{0}}{\partial \tau^2} (\delta - 1)\right|_{L^2(D_{3R/4})} \leq \kappa \epsilon^{\frac{1}{2}}, \notag\\
\left|\frac{\partial (\varphi^{0} - \bar{\varphi^{0})})}{\partial \eta}\right|_{L^2(D_{3R/4})} &= \left|\frac{\partial \varphi^{0}}{\partial \eta} (\delta - 1)\right|_{L^2(D_{3R/4})} + \left|\varphi^{0}\frac{\partial (\delta - 1)}{\partial \eta}\right|_{L^2(D_{3R/4})} \leq \kappa \epsilon^{-\frac{1}{2}} + \kappa \epsilon^{\frac{1}{2}} \leq \kappa \epsilon^{-\frac{1}{2}}, \notag\\
\left|\frac{\partial (\varphi^{0} - \bar{\varphi^{0})})}{\partial \tau}\right|_{L^2(D_{3R/4})} &= \left|\frac{\partial \varphi^{0}}{\partial \tau} (\delta - 1)\right|_{L^2(D_{3R/4})} \leq \kappa \epsilon^{\frac{1}{2}}. 
\end{align}
Using \eqref{enriched transforemd pde} and \eqref{approximator_error_bound}, we find that
\begin{align}
\left|L_{\epsilon}(\bar{\varphi}^{0} - \varphi^{0})\right|_{L^2(D_{3R/4})} \leq & \left|\frac{\epsilon}{H} \frac{\partial^2 (\bar{\varphi}^{0}  - \varphi^{0})}{\partial \tau^2}\right|_{L^2(D_{3R/4})} \nonumber \\
&+ \left|\frac{a\sinh (R-\eta) \cos \tau}{H} \frac{\partial (\bar{\varphi}^{0}  - \varphi^{0})}{\partial \tau}\right|_{L^2(D_{3R/4})} \nonumber \\
&+ \left|\frac{a (\cosh (R-\eta)-\cosh R) \sin \tau}{H} \frac{\partial (\bar{\varphi}^{0}  - \varphi^{0})}{\partial \eta}\right|_{L^2(D_{3R/4})} \\
\leq & \kappa \epsilon^{\frac{3}{2}} + \kappa \epsilon^{\frac{1}{2}} + \kappa \epsilon^{\frac{1}{2}} \leq \kappa \epsilon^{\frac{1}{2}}. 
\end{align}
\end{proof}
\begin{theorem}
The following estimate holds:
\begin{equation}
|u^\varepsilon - u^0 - \bar{\varphi}^0|_{L^2(\Omega)} + \sqrt{\varepsilon} |u^\varepsilon - u^0 - \bar{\varphi}^0|_{H^1(\Omega)} \leq \kappa \sqrt{\varepsilon}.
\end{equation}
where \(u^\varepsilon\) and \(u^0\) are the solutions of (\ref{original pde}) and (\ref{limit solution}), respectively, and \(\bar{\varphi}^0\) is the corrector in (\ref{delta corrector}).
\end{theorem}
\begin{proof}
Setting \( w = u^\varepsilon - u^0 - \bar{\varphi}^0 \) from (\ref{Transforemd pde}), (\ref{Approximated corrector equation}) and (\ref{delta corrector}), we deduce that
\begin{equation}
\begin{cases}
  -\varepsilon \Delta w - w_{y} = \text{R.H.S.},\\
  w = 0 \quad \text{on} \quad \partial \Omega.
\end{cases}
\end{equation}
Using the approximate form \(\bar{\varphi^0}\) for \(\varphi^0\), we write
\[
L_\varepsilon \bar{\varphi^0} = L_\varepsilon \varphi^0 + L_\varepsilon (\bar{\varphi^0} - \varphi^0),
\]
and then
\[
(L_\varepsilon \varphi^0, \varphi) = (L_\varepsilon \varphi^0, \varphi \delta(\eta)),
\]
for all \(\varphi \in H^1_0(\Omega)\), where \((,)\) denotes the scalar product in the space \(L^2(\Omega)\) and \(\delta(\eta)\) is a smooth function such that \(\delta(\eta) = 1\) if \(\tau \leq R/2\) and \(\delta(\eta) = 0\) if \(\tau \geq 3R/4\). Note that \(\varphi^0 = 0\) for \(\tau \geq 3R/4\).
We first observe that
\begin{equation}\label{R.H.S}
R.H.S. = \varepsilon \Delta u^0 - L_\varepsilon (\varphi^0) + L_\varepsilon (\varphi^0 - \bar{\varphi^0}).
\end{equation}
Taking the scalar product of (\ref{R.H.S}) with \( e^y w \), we find that
\begin{align}
\ \varepsilon | w |^2_{H^1(\Omega)} + | w |^2_{L^2(\Omega)} \ &\leq \varepsilon| ( \Delta u^\varepsilon, e^y w \delta(\eta)) | + | (L_\varepsilon (\varphi^0), e^y w \delta(\eta)) | + | (L_\varepsilon (\varphi^0 - \bar{\varphi^0}), e^y w \delta(\eta)) |  \\
&\leq \kappa \varepsilon | u^0 |_{H^1(D)} | w |_{H^1(D)} \notag\\
&+ \left| \left( \frac{1}{H} \left( -\varepsilon \varphi_{\tau \tau} - a \sin \tau (\cosh R - \cosh (R-\eta)) \varphi_\eta - a \sinh \eta \cos \tau \varphi_{\tau} \right), e^y w \delta(\eta) \right) \right|
\notag\\
&+ \left| (L_\varepsilon (\varphi^0 - \bar{\varphi^0}), e^y w \delta(\eta)) \right| \\
&\leq \kappa \varepsilon | u^0 |_{H^1(D)} | w |_{H^1(D)} 
+ \kappa \varepsilon \left| \frac{\partial \varphi^0}{\partial \tau} \right|_{L^2(D)} \left( \left| \frac{\partial w}{\partial \tau} \delta(\eta) \right|_{L^2(D)} + \left| w \delta(\eta) \right|_{L^2(D)}\right) \notag\\
&+ \kappa \left|(\cosh R-\cosh (R-\eta)) \frac{\partial \varphi^0}{\partial \eta} \right|_{L^2(D)} \left| w \delta(\eta) \right|_{L^2(D)} \notag\\
&+ \kappa|\cos \tau \frac{\partial \varphi^0}{\partial \tau}|_{L^2(D)} \left| w \delta(\eta) \right|_{L^2(D)} \notag\\
&+ \kappa\left| L_\varepsilon (\varphi^0 - \bar{\varphi^0}) \right|_{L^2(D_{3R/4})} \left| w \delta(\eta) \right|_{L^2(D)}.
\end{align}
Hence, given that
\[
\left| \frac{\partial w}{\partial \tau} \delta(\eta) \right|_{L^2(D)}, \left| \nabla_{\tau, \eta} (w \delta(\eta)) \right|_{L^2(D)} \leq \kappa |w|_{H^1(\Omega)}
\]
and
\[
\left| w \delta(\eta) \right|_{L^2(D)} \leq \kappa |w|_{L^2(\Omega)},
\]
and using \eqref{b:corrector_estimate}, we obtain
\begin{equation}
|w|_{L^2(\Omega)} + \sqrt{\varepsilon} |w|_{H^1(\Omega)} \leq \kappa \sqrt{\varepsilon}.
\end{equation}
\end{proof}

\section*{Acknowledgments}
Gie was partially supported by Ascending Star Fellowship, Office of EVPRI, University of Louisville; Simons Foundation Collaboration Grant for Mathematicians; Research R-II Grant, Office of EVPRI, University of Louisville; Brain Pool Program through the National Research Foundation of Korea (NRF) (2020H1D3A2A01110658). The work of Y. Hong was supported by Basic Science Research Program through the National Research Foundation of Korea (NRF) funded by the Ministry of Education (NRF-2021R1A2C1093579) and by the Korea government(MSIT) (RS-2023-00219980, RS-2024-00440063). Jung was supported by the National Research Foundation of Korea(NRF) grant funded by the Korea government(MSIT) (No. 2023R1A2C1003120).

\bibliographystyle{plain}
\bibliography{references}

\begin{thebibliography}{10}

\bibitem{ainsworth2022galerkin}
Mark Ainsworth and Justin Dong.
\newblock Galerkin neural network approximation of singularly-perturbed elliptic systems.
\newblock {\em Computer Methods in Applied Mechanics and Engineering}, page 115169, 2022.

\bibitem{blpinn}
Amirhossein Arzani, Kevin~W. Cassel, and Roshan~M. D'Souza.
\newblock Theory-guided physics-informed neural networks for boundary layer problems with singular perturbation.
\newblock {\em Journal of Computational Physics}, 473:111768, 2023.

\bibitem{pinn_bl01}
Hassan Bararnia and Mehdi Esmaeilpour.
\newblock On the application of physics informed neural networks (pinn) to solve boundary layer thermal-fluid problems.
\newblock {\em International Communications in Heat and Mass Transfer}, 132:105890, 2022.

\bibitem{choi2022unsupervised}
Junho Choi, Namjung Kim, and Youngjoon Hong.
\newblock Unsupervised legendre–galerkin neural network for solving partial differential equations.
\newblock {\em IEEE Access}, 11:23433--23446, 2023.

\bibitem{CKH2024}
Junho Choi, Taehyun Yun, Namjung Kim, and Youngjoon Hong.
\newblock Spectral operator learning for parametric pdes without data reliance.
\newblock {\em Computer Methods in Applied Mechanics and Engineering}, 420:116678, 2024.

\bibitem{churchill1961mathematical}
Henry~S Churchill.
\newblock Mathematical handbook for scientists and engineers, 1961.

\bibitem{cuomo2022scientific}
Salvatore Cuomo, Vincenzo~Schiano Di~Cola, Fabio Giampaolo, Gianluigi Rozza, Maziar Raissi, and Francesco Piccialli.
\newblock Scientific machine learning through physics--informed neural networks: where we are and what’s next.
\newblock {\em Journal of Scientific Computing}, 92(3):88, 2022.

\bibitem{xpinn}
Ameya D.~Jagtap and George Em~Karniadakis.
\newblock Extended physics-informed neural networks (xpinns): A generalized space-time domain decomposition based deep learning framework for nonlinear partial differential equations.
\newblock {\em Communications in Computational Physics}, 28(5):2002--2041, 2020.

\bibitem{pinn_lim}
Félix {Fernández de la Mata}, Alfonso Gijón, Miguel Molina-Solana, and Juan Gómez-Romero.
\newblock Physics-informed neural networks for data-driven simulation: Advantages, limitations, and opportunities.
\newblock {\em Physica A: Statistical Mechanics and its Applications}, 610:128415, 2023.

\bibitem{pinn_bl02}
Antonio Tadeu~Azevedo Gomes, Larissa~Miguez da~Silva, and Frederic Valentin.
\newblock Physics-aware neural networks for boundary layer linear problems, 2022.

\bibitem{goswami2020transfer}
Somdatta Goswami, Cosmin Anitescu, Souvik Chakraborty, and Timon Rabczuk.
\newblock Transfer learning enhanced physics informed neural network for phase-field modeling of fracture.
\newblock {\em Theoretical and Applied Fracture Mechanics}, 106:102447, 2020.

\bibitem{bookSP}
Makram Hamouda, Gung-Min Gie, Chang-Yeol Jung, and Roger Temam.
\newblock {\em Singular Perturbations and Boundary Layers}.
\newblock Springer International Publishing, 12 2018.

\bibitem{pinn_bl03}
Jihun Han and Yoonsang Lee.
\newblock A neural network approach for homogenization of multiscale problems.
\newblock {\em Multiscale Modeling \& Simulation}, 21(2):716--734, 2023.

\bibitem{hong2014numerical}
Youngjoon Hong, Chang-Yeol Jung, and Roger Temam.
\newblock On the numerical approximations of stiff convection--diffusion equations in a circle.
\newblock {\em Numerische Mathematik}, 127(2):291--313, 2014.

\bibitem{hong2015singular}
Youngjoon Hong, Chang-Yeol Jung, and Roger Temam.
\newblock Singular perturbation analysis of time dependent convection--diffusion equations in a circle.
\newblock {\em Nonlinear Analysis: Theory, Methods \& Applications}, 119:127--148, 2015.

\bibitem{Nam03}
Ameya~D. Jagtap, Ehsan Kharazmi, and George~Em Karniadakis.
\newblock Conservative physics-informed neural networks on discrete domains for conservation laws: Applications to forward and inverse problems.
\newblock {\em Computer Methods in Applied Mechanics and Engineering}, 365:113028, 2020.

\bibitem{jin2021nsfnets}
Xiaowei Jin, Shengze Cai, Hui Li, and George~Em Karniadakis.
\newblock Nsfnets (navier-stokes flow nets): Physics-informed neural networks for the incompressible navier-stokes equations.
\newblock {\em Journal of Computational Physics}, 426:109951, 2021.

\bibitem{jung2005numerical}
Chang-Yeol Jung.
\newblock Numerical approximation of two-dimensional convection-diffusion equations with boundary layers.
\newblock {\em Numerical Methods for Partial Differential Equations: An International Journal}, 21(3):623--648, 2005.

\bibitem{JUNG201188}
Chang-Yeol Jung and Roger Temam.
\newblock Convection–diffusion equations in a circle: The compatible case.
\newblock {\em Journal de Mathématiques Pures et Appliquées}, 96(1):88--107, 2011.

\bibitem{JT13}
Chang-Yeol Jung and Roger Temam.
\newblock Singular perturbations and boundary layer theory for convection-diffusion equations in a circle: The generic noncompatible case.
\newblock {\em SIAM Journal on Mathematical Analysis}, 44(6):4274--4296, 2012.

\bibitem{PINN007}
George~Em Karniadakis, Ioannis~G Kevrekidis, Lu~Lu, Paris Perdikaris, Sifan Wang, and Liu Yang.
\newblock Physics-informed machine learning.
\newblock {\em Nature Reviews Physics}, 3(6):422--440, 2021.

\bibitem{kochkov2021machine}
Dmitrii Kochkov, Jamie~A Smith, Ayya Alieva, Qing Wang, Michael~P Brenner, and Stephan Hoyer.
\newblock Machine learning--accelerated computational fluid dynamics.
\newblock {\em Proceedings of the National Academy of Sciences}, 118(21):e2101784118, 2021.

\bibitem{li2020fourier}
Zongyi Li, Nikola Kovachki, Kamyar Azizzadenesheli, Burigede Liu, Kaushik Bhattacharya, Andrew Stuart, and Anima Anandkumar.
\newblock Fourier neural operator for parametric partial differential equations.
\newblock {\em arXiv preprint arXiv:2010.08895}, 2020.

\bibitem{PINN002}
Lu~Lu, Ming Dao, Punit Kumar, Upadrasta Ramamurty, George~Em Karniadakis, and Subra Suresh.
\newblock Extraction of mechanical properties of materials through deep learning from instrumented indentation.
\newblock {\em Proceedings of the National Academy of Sciences}, 117(13):7052--7062, 2020.

\bibitem{lu2021learning}
Lu~Lu, Pengzhan Jin, Guofei Pang, Zhongqiang Zhang, and George~Em Karniadakis.
\newblock Learning nonlinear operators via {DeepONet} based on the universal approximation theorem of operators.
\newblock {\em Nature Machine Intelligence}, 3(3):218--229, 2021.

\bibitem{PINN001}
Lu~Lu, Xuhui Meng, Zhiping Mao, and George~Em Karniadakis.
\newblock Deepxde: A deep learning library for solving differential equations.
\newblock {\em SIAM Review}, 63(1):208--228, 2021.

\bibitem{PINN005}
Xuhui Meng, Zhen Li, Dongkun Zhang, and George~Em Karniadakis.
\newblock Ppinn: Parareal physics-informed neural network for time-dependent pdes.
\newblock {\em Computer Methods in Applied Mechanics and Engineering}, 370:113250, 2020.

\bibitem{PINN003}
Khemraj Shukla, Patricio~Clark Di~Leoni, James Blackshire, Daniel Sparkman, and George~Em Karniadakis.
\newblock Physics-informed neural network for ultrasound nondestructive quantification of surface breaking cracks.
\newblock {\em Journal of Nondestructive Evaluation}, 39(3):1--20, 2020.

\bibitem{PINN008}
Sifan Wang, Shyam Sankaran, and Paris Perdikaris.
\newblock Respecting causality is all you need for training physics-informed neural networks.
\newblock {\em arXiv preprint arXiv:2203.07404}, 2022.

\bibitem{wang2021eigenvector}
Sifan Wang, Hanwen Wang, and Paris Perdikaris.
\newblock On the eigenvector bias of fourier feature networks: From regression to solving multi-scale pdes with physics-informed neural networks.
\newblock {\em Computer Methods in Applied Mechanics and Engineering}, 384:113938, 2021.

\bibitem{PINN004}
Liu Yang, Xuhui Meng, and George~Em Karniadakis.
\newblock B-pinns: Bayesian physics-informed neural networks for forward and inverse pde problems with noisy data.
\newblock {\em Journal of Computational Physics}, 425:109913, 2021.

\bibitem{tspcircle}
Roger~Temam Youngjoon~Hong, Chang-Yeol~Jung.
\newblock Singular perturbation analysis of time dependent convection–diffusion equations in a circle.
\newblock {\em Nonlinear Analysis: Theory, Methods and Applications}, 119:127--148, 2015.

\bibitem{yu2018deep}
Bing Yu et~al.
\newblock The deep ritz method: a deep learning-based numerical algorithm for solving variational problems.
\newblock {\em Communications in Mathematics and Statistics}, 6(1):1--12, 2018.

\bibitem{PINN006}
Lei Yuan, Yi-Qing Ni, Xiang-Yun Deng, and Shuo Hao.
\newblock A-pinn: Auxiliary physics informed neural networks for forward and inverse problems of nonlinear integro-differential equations.
\newblock {\em Journal of Computational Physics}, page 111260, 2022.

\end{thebibliography}
\end{document}